\documentclass[11pt,a4paper]{article}

\usepackage[utf8]{inputenc}
\usepackage[T1]{fontenc}
\usepackage[english]{babel}

\usepackage{amsmath}
\usepackage{amsfonts}
\usepackage{amsthm}

\IfFileExists{libertinus.sty}{%
  \usepackage{libertinus}%
}{%
  \IfFileExists{newtxtext.sty}{%
    \usepackage{newtxtext}%
    \usepackage{newtxmath}%
  }{%
    \usepackage{mathpazo}%
  }%
}

\usepackage[margin=3cm]{geometry}
\usepackage[letterspace=180]{microtype}

\usepackage{xcolor}
\definecolor{LinkViolet}{HTML}{6D28D9}   
\definecolor{CiteViolet}{HTML}{9333EA}   
\definecolor{RuleGrey}{HTML}{4B5563} 

\definecolor{AvailRed}{HTML}{E41A1C}
\definecolor{AvailOrange}{HTML}{FF7F00}
\definecolor{AvailYellow}{HTML}{A3A300}
\definecolor{AvailGreen}{HTML}{4DAF4A}

\usepackage{pifont}
\usepackage{enumitem}
\setlist{leftmargin=*}
\setlist[itemize,1]{label=\ding{96}}

\usepackage{graphicx}
\usepackage{pdflscape}  
\usepackage{afterpage}  
\usepackage{tikz}
\usetikzlibrary{calc}
\usepackage{booktabs}
\usepackage{subcaption}

\usepackage{hyperref}
\hypersetup{
  colorlinks = true,
  breaklinks = true,
  linkcolor  = LinkViolet,
  citecolor  = CiteViolet,
  urlcolor   = LinkViolet,
  linktoc    = all
}

\usepackage{titling}
\usepackage{authblk}

\newif\ifShowAffiliations
\ShowAffiliationstrue

\makeatletter
\newcommand{\Author}[2][]{%
  \ifShowAffiliations\else
    \renewcommand{\AB@authnote}[1]{}%
    \renewcommand{\@author}{\AB@authlist}%
  \fi
  \author[#1]{#2}%
}
\makeatother

\newcommand{\Affiliation}[2][]{\ifShowAffiliations\affil[#1]{#2}\fi}

\newenvironment{aidisclosure}{%
  \small
  \begin{center}%
    {\bfseries AI Disclosure\vspace{-.5em}\vspace{0pt}}%
  \end{center}%
  \quotation}
  {\endquotation}

\pretitle{\begin{center}\LARGE\bfseries}
\posttitle{\par\end{center}\vspace{0.5ex}}
\preauthor{\begin{center}}
\postauthor{\end{center}}
\predate{}
\postdate{}

\usepackage{titlesec}

\newcommand{\PartRule}{{\color{RuleGrey}\titlerule[0.8pt]}}

\titleformat{\part}[display]
  {\centering\normalfont}
  {\PartRule\vspace{0.35ex}\normalsize\scshape\partname~\thepart}
  {0.4ex}
  {\large\bfseries}
  [\vspace{0.4ex}\PartRule]

\titlespacing*{\part}{0pt}{4ex plus 1ex minus .5ex}{4ex plus .5ex}

\usepackage{tocloft}

\newcommand{\TocFont}{\footnotesize}

\theoremstyle{plain}
\newtheorem{theorem}{Theorem}[section]
\newtheorem{proposition}[theorem]{Proposition}
\newtheorem{lemma}[theorem]{Lemma}
\newtheorem{corollary}[theorem]{Corollary}

\theoremstyle{definition}

\newtheorem{remark}[theorem]{Remark}

\newtheorem{condition}{Condition}

\newcommand{\ProofDeferred}[2][]{%
  \begin{proof}
    #1The proof of this result has been deferred to Appendix~\ref{#2}.
  \end{proof}%
}

\newcommand{\AutoDelims}[3]{%
  {\mathchoice{\left#1#3\right#2}{#1#3#2}{#1#3#2}{#1#3#2}}}
\newcommand{\AutoArg}[1]{%
  {\mathchoice{\!\left(#1\right)}{(#1)}{(#1)}{(#1)}}}

\newcommand{\Naturals}{\mathbb{N}}                 
\newcommand{\Integers}{\mathbb{Z}}                 
\newcommand{\Reals}{\mathbb{R}}                    
\newcommand{\StateSpace}{\mathcal{X}}              
\newcommand{\ParamSpace}{\Theta}                   
\newcommand{\Lspace}[2]{L^{#1}\AutoArg{#2}}
\newcommand{\Ltwopi}{\Lspace{2}{\pi}}

\newcommand{\IntegerRange}[2]{\AutoDelims{\{}{\}}{#1,\dots,#2}}
\newcommand{\IndexBetween}[3]{#2 \leq #1 \leq #3}

\newcommand{\Prob}{\mathbf{P}}                     
\newcommand{\Expect}{\mathbf{E}}                   
\newcommand{\Variance}{\mathbf{Var}}               
\newcommand{\Covariance}{\mathbf{Cov}}             
\newcommand{\Bias}{\mathbf{Bias}}                  
\newcommand{\Trace}{\mathbf{Tr}}                   
\newcommand{\MSE}{\mathbf{MSE}}                    
\newcommand{\AsymptoticMSE}{\MSE_{\infty}}         
\newcommand{\Gaussian}[2]{\mathcal{N}\AutoArg{#1,#2}}
\newcommand{\Dirac}[1]{\delta_{#1}}

\newcommand{\ConvergesAlmostSurely}{\xrightarrow{\;\mathrm{a.s.}\;}}
\newcommand{\ConvergesInDistribution}{\xrightarrow{\;\mathrm{d}\;}}

\newcommand{\BigO}{O}
\newcommand{\LittleO}{o}
\newcommand{\EquivalentTo}{\sim}

\newcommand{\SumOver}[1]{\sum_{#1 \in \StateSpace}}
\newcommand{\SumOverDistinctXY}{%
  \sum_{\substack{x,y,z \in \StateSpace \\ x \neq y}}}
\newcommand{\SumFromTo}[3]{\sum_{#1 = #2}^{#3}}

\newcommand{\NormTV}[1]{\AutoDelims{\|}{\|}{#1}_{\mathrm{TV}}}      
\newcommand{\NormTwo}[1]{\AutoDelims{\|}{\|}{#1}_{2}}               
\newcommand{\NormOperator}[1]{%
  {\mathchoice{\left|\!\left|\!\left|#1\right|\!\right|\!\right|}%
              {|\!|\!|#1|\!|\!|}{|\!|\!|#1|\!|\!|}{|\!|\!|#1|\!|\!|}}}
\newcommand{\NormPi}[1]{\AutoDelims{\|}{\|}{#1}_{\pi}}
\newcommand{\ScalarProduct}[2]{\AutoDelims{\langle}{\rangle}{#1, #2}_{\pi}}

\newcommand{\Kernel}[3][]{P^{#1}\AutoArg{#2\vert#3}}
\newcommand{\KernelParam}[4][]{P_{#2}^{#1}\AutoArg{#3\vert#4}}
\newcommand{\KernelMatrix}[1][]{P^{#1}}
\newcommand{\KernelMatrixParam}[2][]{P_{#2}^{#1}}
\newcommand{\KernelMatrixTrue}[1][]{\KernelMatrixParam[#1]{\ParamTrue}}
\newcommand{\KernelMatrixPlugIn}[1][]{\KernelMatrixParam[#1]{\ParamEstimator}}

\newcommand{\PiHat}[2][n]{\widehat{\pi}_{#1}^{\,#2}}
\newcommand{\PiHatAt}[3][n]{\PiHat[#1]{#2}\AutoArg{#3}}
\newcommand{\PiHatOf}[3][n]{\PiHat[#1]{#2}\AutoArg{#3}}
\newcommand{\Empirical}[1][n]{\PiHat[#1]{0}}
\newcommand{\EmpiricalAt}[2][n]{\PiHatAt[#1]{0}{#2}}
\newcommand{\EmpiricalOf}[2][n]{\PiHatOf[#1]{0}{#2}}

\newcommand{\PiHatParam}[3][n]{\widehat{\pi}_{#2,\,#1}^{\,#3}}
\newcommand{\PiHatParamAt}[4][n]{\PiHatParam[#1]{#2}{#3}\AutoArg{#4}}

\newcommand{\PiHatAggregated}[2]{\widehat{\pi}_{n}^{\,#1}\AutoArg{#2}}

\newcommand{\Contribution}[3]{Y_{#2}^{#1}\AutoArg{#3}}
\newcommand{\ContributionParam}[4]{Y_{#1,#3}^{#2}\AutoArg{#4}}

\newcommand{\KernelAverage}[3]{S_{#2}^{#1}\AutoArg{#3}}

\newcommand{\FrameworkIid}[1]{\ensuremath{\mathcal{I}(#1)}}
\newcommand{\FrameworkMarkov}[2]{\ensuremath{\mathcal{M}(#1,#2)}}
\newcommand{\FrameworkIidStationary}{\FrameworkIid{\pi}}
\newcommand{\FrameworkIidGeneral}{\FrameworkIid{\mu}}
\newcommand{\FrameworkMarkovStationary}{\FrameworkMarkov{\pi}{P}}
\newcommand{\FrameworkMarkovGeneral}{\FrameworkMarkov{\mu}{P}}

\newcommand{\AsymptoticVariance}[2][]{\sigma_{#1}^{2}\AutoArg{#2}} 
\newcommand{\TotalVariance}[1]{V_{#1}} 
\newcommand{\GreenFunction}[2]{G\AutoArg{#1\vert#2}}
\newcommand{\HittingTime}[2]{m_{#1 \to #2}}  
\newcommand{\CommuteTime}[2]{T_{#1,#2}}
\newcommand{\FundamentalMatrix}{Z}  
\newcommand{\StationaryProjector}{\Pi} 
\newcommand{\DiagonalStationary}{D_{\pi}} 
\newcommand{\Identity}{I}
\newcommand{\Ones}{\mathbf{1}}
\newcommand{\HilbertBasis}[1]{e_{#1}}

\newcommand{\ParamTrue}{\theta^{*}}                       
\newcommand{\ParamEstimator}[1][m]{\widehat{\theta}_{#1}} 
\newcommand{\EstimationError}{\ParamEstimator-\ParamTrue} 
\newcommand{\Sensitivity}[2]{A^{#1}\AutoArg{#2}} 
\newcommand{\Gradient}{\nabla}
\newcommand{\Hessian}{\nabla^{2}}
\newcommand{\SplitFraction}{\alpha}                       
\newcommand{\SplitFractionOptimal}{\SplitFraction^{*}}

\newcommand{\CovParam}{\mathbf{\Sigma}}

\newcommand{\RatioLookAhead}{r_{1}}
\newcommand{\RatioParameter}{r_{2}}

\newcommand{\PlotWidth}{\textwidth}
\newcommand{\QuadPlotWidth}{0.49\linewidth}
\newcommand{\PlotGap}{1ex}

\newcommand{\MapWidth}{0,83\textwidth}

\newcommand{\Transpose}{^{\mathsf{T}}}
\newcommand{\Indicator}[1]{\mathbf{1}_{\{#1\}}}
\newcommand{\TextAnd}{\quad \text{and} \quad}

\newcommand{\BDKernel}{P_{\theta, \varepsilon}}
\newcommand{\BDpi}{\pi_{\theta, \varepsilon}}

\newcommand{\Span}{\mathbf{span}}

\ShowAffiliationstrue

\title{Mean squared error reduction of (plug-in)\\
look-ahead estimators for Markov chains}

\Author[1]{Romain Aza\"{\i}s}
\Author[2]{Beno\^{i}t Henry}
\Author[1]{Henri P\'echoux}

\Affiliation[1]{Inria Lyon, France.}
\Affiliation[2]{Univ. Lille, CNRS, IMT Nord Europe, Inria, UMR 8524 -- Laboratoire Paul Painlev\'e, F-59000 Lille, France.}

\date{}

\begin{document}

\maketitle

\begin{abstract}
We consider the estimation of the invariant distribution $\pi$ of a Markov kernel on a discrete state space (or of the mean value $\pi(f)$ of a functional $f$ under $\pi$) when solving the invariance equation is not feasible.
Look-ahead estimators exploit the knowledge of the kernel by propagating the empirical occupation measure through $k$ steps of the chain.
When the kernel is known, as in Markov chain Monte Carlo, we compare their mean squared error with that of the empirical estimator, for independent or Markovian data, stationary or not.
Our main result establishes that, for Markovian data, the one-step estimator always improves on the empirical one asymptotically, and that the improvement grows with $k$ for reversible kernels.
A non-reversible counterexample shows that reversibility cannot simply be dropped.
In the statistical setting where the kernel belongs to a parametric family with unknown parameter, we introduce plug-in look-ahead estimators, which perform the look-ahead step with an estimated kernel.
When the parameter is estimated at the parametric rate, we give a necessary and sufficient condition under which the plug-in look-ahead estimator outperforms the empirical one.
The theoretical results are illustrated by extensive simulation studies, which also shed further light on the behaviour of the estimators at hand.
Both settings are finally applied to real data arising from geographical mobility problems in the Lyon metropolitan area, France.
\end{abstract}

\begin{aidisclosure}
The present manuscript made use of AI tools for \LaTeX{} formatting, coding, literature review, language enhancement and proof verification.
All mathematical results and arguments are human-made, with only marginal use of AI for mathematical content.
This use was limited primarily to answering standard mathematical questions whose answers could, in principle, be found in the existing literature.
One notable exception is the counterexample presented in Appendix~\ref{sssec:counterexample}, which was produced with the assistance of AI.

In particular, we are aware that a version of Theorem~\ref{thm:mse-expansion} can be established under substantially weaker hypotheses, and that more striking examples supporting our framework can be constructed.
However, we deliberately chose to keep the mathematics as human as possible and therefore did not include these generalisations or additional examples.
\end{aidisclosure}
\vspace*{\fill}

\newpage
\vspace*{\fill}
\begin{list}{}{\setlength{\leftmargin}{\leftmargini}%
               \setlength{\rightmargin}{\leftmargini}%
               \setlength{\topsep}{0pt}\setlength{\partopsep}{0pt}}
  \item\relax
  {\TocFont\tableofcontents}
\end{list}
\vspace*{\fill}
\vspace*{1.5em}

\newpage

\section{Introduction}
\label{sec:introduction}

\subsection{Problem formulation}

This paper is concerned with the estimation of a probability distribution $\pi$ on a discrete set $\StateSpace$, or of the mean value $\pi(f) = \SumOver{x} f(x)\pi(x)$ of a functional $f : \StateSpace \to \Reals^{d}$ under $\pi$, when $\pi$ is the unique invariant distribution of a transition kernel $\KernelMatrix$ on $\StateSpace$, i.e.\ when
\begin{equation}\label{eq:invariance}
  \pi = \pi \KernelMatrix,
\end{equation}
an identity we refer to as the invariance equation.

When $\KernelMatrix$ is known, this is precisely the setting of Markov chain Monte Carlo (MCMC) methods, whose purpose is to estimate quantities such as $\pi(f)$ by stochastic simulation.
The kernel is then designed to admit $\pi$ as its (unique) invariant distribution, for instance by the Metropolis--Hastings algorithm \cite{Hastings1970,Metropolis1953} (see also~\cite{AsmussenGlynn2007} and references therein for a general account of the method).
In this context, one simulates the first $n$ steps $(X_i)_{\IndexBetween{i}{1}{n}}$ of a Markov chain with kernel $\KernelMatrix$ and estimates $\pi$ by its empirical counterpart,
\begin{equation}\label{eq:empirical-estimator}
  \Empirical = \frac{1}{n} \SumFromTo{i}{1}{n} \Dirac{X_i},
\end{equation}
and $\pi(f)$ by $\EmpiricalOf{f} = n^{-1} \SumFromTo{i}{1}{n} f(X_i)$.
Irreducibility, aperiodicity and positive recurrence of $\KernelMatrix$ make $\Empirical$ a consistent estimator of $\pi$, whatever the initial distribution.
Under the additional condition of uniform ergodicity, a central limit theorem for $\EmpiricalAt{x}$ holds \cite{DMPS2018}.
Furthermore, $\EmpiricalAt{x}$ is efficient: the empirical occupation measure is the nonparametric maximum likelihood estimator of the invariant distribution, which attains the semiparametric efficiency bound without further structural assumptions~\cite{GreenwoodWefelmeyer1995,Penev1991}.

In a Markovian statistical setting, the exact transition kernel of the data is in general unknown, but it can often be assumed to belong to a parametric family, i.e.\ $\KernelMatrix = \KernelMatrixTrue$ where $\{\KernelMatrixParam{\theta}\}_{\theta \in \ParamSpace}$ is a family of transition kernels on $\StateSpace$.
The quantities $\pi$ and $\pi(f)$ remain of interest in that context, and the estimator $\Empirical$ above retains its consistency and asymptotic normality.
However, it is no longer efficient: once the kernel is known to lie in some parametric model, the empirical estimator ignores that information, and estimators exploiting the structure of the family can achieve a strictly smaller asymptotic variance~\cite{SchickWefelmeyer1999}.

Whether or not $\KernelMatrix$ is known (if it is not, it can be replaced by $\KernelMatrixParam{\ParamEstimator[n]}$, where $\ParamEstimator[n]$ is an estimator of $\ParamTrue$, provided the quantities of interest depend continuously on $\theta$), it is legitimate to think of reaching $\pi$ or $\pi(f)$ by a deterministic approach, exact or approximate, through the invariance equation~\eqref{eq:invariance}.
For a finite state space of moderate size, solving the invariance equation is exact and fast, and therefore preferable \cite[p.~375]{HendersonGlynn2001}.
Nevertheless, three main obstructions rule out that route for certain state spaces, while the stochastic approach described above remains valid.
\begin{itemize}
  \item The first obstruction is one of scale.
    In statistical physics, for instance, the cardinality of the state space is often exponential in the size of the problem (e.g.\ for the Ising model), so that the invariance equation, though finite, cannot even be written down.
    This is the founding argument for MCMC~\cite{Diaconis2009}.
    When the target is a functional, the obstruction is even stronger: (i) solving for $\pi$ in order to obtain $\pi(f)$ determines an entire distribution so as to deliver a single number; (ii) the final summation $\SumOver{x} \pi(x) f(x)$ ranges over the very state space that is out of reach.
    The ergodic average $\EmpiricalOf{f}$ short-circuits both steps, since it estimates the scalar directly and never represents $\pi$.

  \item The second obstruction is not one of degree but of kind: when $\StateSpace$ is countably infinite, there is no finite system to solve.
    The standard remedy is to truncate $\KernelMatrix$ to a finite north-west corner, to redistribute the mass lost by truncation (so as to restore stochasticity), then to solve the finite system obtained.
    This substitutes for the original problem a pair of new ones: whether the approximants converge to the target, and how the truncation level is to be chosen.
    Neither is settled a priori, convergence depending on the augmentation scheme and on the tail behaviour of the chain \cite{LiuLi2018, Tweedie1998}.
    The difficulty is even greater when the target is a functional: the truncation level is chosen so that the discarded states carry negligible probability, but when $f$ is unbounded, negligible probability need not mean negligible contribution to $\pi(f)$.
    A truncation validated on $\pi$ alone therefore carries no guarantee for $\pi(f)$, whereas simulation never has to make that choice.

  \item The third obstruction does not concern size, but specification.
    The state space may be finite and yet be given only implicitly, as the set of configurations satisfying a system of constraints (e.g.\ the proper colourings of a graph), so that, although the transition rule is perfectly explicit locally, in the sense that the row $\Kernel{\cdot}{x}$ can be evaluated at any given $x$, the set $\StateSpace$ is never available as a list.
    Exact computation is then excluded in a strong sense: counting the elements of $\StateSpace$ is typically $\#\mathrm{P}$-hard.
    Simulation is the natural mode of access, since a trajectory never needs more than what is reachable from its current state, and the approximate-counting literature has shown that it yields rigorous approximations, not mere heuristics~\cite{JerrumSinclairVigoda2004}.
\end{itemize}

In this paper we therefore rule out the deterministic route to the quantities of interest $\pi$ and $\pi(f)$.
We nevertheless wish to exploit our knowledge of the kernel $\KernelMatrix$ (perfect in the MCMC setting, imperfect in the statistical one) in order to improve upon the estimator $\Empirical$, in the sense of reducing its (asymptotic) mean squared error (MSE).

\subsection{Look-ahead estimation}

In the MCMC literature, variance reduction is frequently sought through Rao--Black\-wel\-li\-sa\-tion: a Monte Carlo estimator is replaced by its conditional expectation given part of the randomness driving the simulation, which leaves the quantity being estimated unchanged while removing the variability that has been integrated out~\cite{CasellaRobert1996, GelfandSmith1990}.
The term covers several distinct choices of conditioning and does not always denote the same construction~\cite{RobertRoberts2021}.
One such choice is to condition on the state of the chain one step earlier, replacing each $f(X_i)$ by $\Expect[f(X_i) \mid X_{i-1}]$, an operation available as soon as the kernel can be evaluated.
Averaging these conditional expectations amounts to reporting $\PiHat{1} = \Empirical \KernelMatrix$ in place of the empirical occupation measure $\Empirical$.
This construction obviously iterates, giving $\PiHat{k} = \Empirical \KernelMatrix[k]$.
Estimators of this form appear in the literature under the name of look-ahead estimators~\cite{Braun2012,GlynnHenderson1998, HendersonGlynn2001}, but are also considered without reference to that name~\cite{McKeagueWefelmeyer2000,Stachurski2008}.

In statistics, where $\KernelMatrix$ is unknown, the constructions above cannot be applied directly.
One method of inference on $\pi$ is nevertheless directly inspired by the look-ahead idea, under the name of likelihood-look-ahead~\cite{GaribottiTsimikasHorowitz2006}: the unknown parameter is first estimated by maximum likelihood, giving $\ParamEstimator[n]$, and an MCMC method with kernel $\KernelMatrixParam{\ParamEstimator[n]}$ is then run to estimate $\pi$ or $\pi(f)$.
It is worth stressing that the MCMC step is here a numerical method for solving the invariance equation of $\KernelMatrixParam{\ParamEstimator[n]}$: deterministic numerical solvers could be applied just as well.
Consistency then rests on assumptions guaranteeing that the solution of $\mu \KernelMatrixParam{\theta} = \mu$ exists, and depends continuously on $\theta$.
Ideas of this kind have been used successfully for other classes of Markov models, such as piecewise-deterministic Markov processes~\cite{Bertail2008Storage}.

\subsection{Contribution}
\label{subsec:contribution}

In this paper, we study look-ahead methods both in the MCMC setting, where the transition kernel is known (Part~\ref{part:known-kernel}), and in the statistical setting, where the kernel is described by a parametric family whose parameter is unknown and has to be estimated (Part~\ref{part:parametric-kernel}).
The main objective is to identify rigorously, in each of these two settings, frameworks general enough for look-ahead methods to improve on the empirical estimator in terms of MSE.
To the best of our knowledge, in the MCMC setting with known kernel, the comparison with the empirical measure $\Empirical$ has only been carried out in~\cite{McKeagueWefelmeyer2000}: the other references on look-ahead estimators~\cite{Braun2012,GlynnHenderson1998,HendersonGlynn2001,Stachurski2008} work on general state spaces, where a comparison with the occupation measure does not directly make sense.
The main result of~\cite{McKeagueWefelmeyer2000} establishes that, under reversibility, i.e.\ when $\pi(x)\Kernel{y}{x} = \Kernel{x}{y}\pi(y)$ for all $x,y \in \StateSpace$, and under the usual assumptions of irreducibility, aperiodicity and uniform ergodicity, the asymptotic variance of $\PiHatOf{k}{f}$ improves on that of $\EmpiricalOf{f}$ whenever $f \in \Ltwopi$.

In the MCMC setting of Part~\ref{part:known-kernel}, Section~\ref{sec:theory-known-kernel} considers two main simulation schemes: either the observations are the terminal states, at a finite but large horizon, of independent Markov trajectories, or they are collected along a single Markov chain run over a long time.
The data are thus independent and identically distributed in the first case, and Markovian in the second.
We further distinguish the stationary case, where the chain is started from $\pi$, from the case where its initial distribution $\nu$ differs from $\pi$.
In the stationary case, the data are distributed according to $\pi$, whether they are independent or Markovian.
The non-stationary case is more delicate.
Independent observations are identically distributed, but according to $\mu = \nu \KernelMatrix[h]$, where $h$ denotes the simulation horizon, which in general differs from $\pi$; the distribution of Markovian observations, on the other hand, converges to $\pi$ in total variation.
This variety of observation schemes (independent or Markovian, stationary or not) gives rise to non-trivial and sometimes unexpected situations.
By comparison, \cite{McKeagueWefelmeyer2000} only considers Markovian data that are asymptotically stationary.

As mentioned above, the performance of the estimators under consideration is measured by their asymptotic MSE.
For a fixed $n$, the MSE of an estimator $\tilde{\pi}_n$ of $\pi$ is defined by
\[
  \MSE\left(\tilde{\pi}_n\right) = \Expect\left[\NormTwo{\tilde{\pi}_n - \pi}^{2}\right],
\]
and decomposes as
\[
  \MSE\left(\tilde{\pi}_n\right)
  = \SumOver{x} \Bias^{2}\!\left(\tilde{\pi}_n(x)\right)
  + \SumOver{x} \Variance\left(\tilde{\pi}_n(x)\right).
\]
The asymptotic MSE is then obtained by normalising by the parametric rate,
\begin{equation}\label{eq:def:amse}
  \AsymptoticMSE\left(\tilde{\pi}_n\right)
  = \limsup_{n \to +\infty} n \, \MSE\left(\tilde{\pi}_n\right).
\end{equation}

In the stationary setting, whether the data are independent or Markovian, the bias vanishes and the MSE reduces to its variance term.
For independent observations, Proposition~\ref{prop:variance-iid-pointwise} establishes that, for a fixed $n$, the variance decreases with $k$, without any additional assumption on $\KernelMatrix$.
For Markovian observations, the situation is more involved and is the subject of Theorem~\ref{thm:monotone-amse}.
Under reversibility, we show that the asymptotic variance decreases with $k$.
We also show that the asymptotic variance of the one-step estimator is smaller than that of the empirical estimator, even without reversibility.
Finally, we exhibit a non-reversible example in which the asymptotic variance of the two-step estimator exceeds that of the one-step estimator, which makes reversibility a relevant assumption in this context.
These results extend those of~\cite{McKeagueWefelmeyer2000} in a non-trivial way.
In particular, they show that, for stationary data, the one-step look-ahead estimator always improves on the empirical estimator.

In the non-stationary setting, whether the data are Markovian or independent, Proposition~\ref{prop:bias-decreasing-iid} shows that, for a fixed $n$, the squared bias decreases with $k$ as soon as $\KernelMatrix$ is doubly stochastic.
Without this assumption, there exist (reversible) kernels for which the bias increases with $k$.
For Markovian data, however, the squared bias remains of order $n^{-2}$ (Proposition~\ref{prop:bias-negligible}).
As a consequence, the asymptotic MSE in the non-stationary Markovian setting only involves the variance term, which behaves as in the stationary case studied in Theorem~\ref{thm:monotone-amse} (Proposition~\ref{prop:variance-nonstationary-markov}).

The look-ahead method naturally provides a whole family of estimators of $\pi$ or of $\pi(f)$, and a natural idea is to aggregate them so as to reduce the variance further.
The linear combination of these estimators that minimises the variance depends on their covariance matrix~\cite{LavancierRochet2016}, which makes the computation of this matrix a crucial issue.
For look-ahead estimators, we show that this covariance matrix can be computed explicitly for independent stationary data (Proposition~\ref{prop:covariance-iid}) and for stationary Markovian data (Proposition~\ref{prop:covariance-markov}), which makes the aggregated estimator directly computable.

Section~\ref{sec:random-surfer} is devoted to a numerical study of look-ahead estimators from Markovian observations.
The simulation model is the random surfer model on connected non-bipartite graphs, on which the random walk is aperiodic.
On Erd\H{o}s--R\'enyi graphs, we illustrate the asymptotic results established for the bias and the variance in the non-stationary setting.
We also study the effect of the spectral gap (which is related to the topological properties of the graph) on the improvement factors as $k$ grows.
The methods are finally applied to real graphs modelling the public transport network of the Lyon metropolitan area, in order to assess the importance of its stations.
The numerical results confirm the relevance of look-ahead methods compared with the empirical estimator.

In Part~\ref{part:parametric-kernel}, devoted to the statistical problem in which the transition kernel is unknown, we introduce a new approach, also inspired by look-ahead methods, which we call plug-in look-ahead.
It is a natural approach, yet to the best of our knowledge it has never been considered in the literature: it consists of the family of estimators $\Empirical \KernelMatrixParam[k]{\ParamEstimator[n]}$.
This approach differs markedly from the likelihood-look-ahead construction above~\cite{GaribottiTsimikasHorowitz2006}, since it never solves the invariance equation and therefore avoids the numerical cost attached to doing so.
It also enjoys, a priori, a second advantage.
As soon as $\KernelMatrixParam{\theta}$ depends continuously on $\theta$, the consistency of $\Empirical$ and of $\ParamEstimator[n]$ suffices to guarantee convergence towards the target $\pi$.
A transition kernel may, however, depend continuously on a parameter without its ergodic properties depending continuously on it as well, which necessarily limits the scope of likelihood-look-ahead methods while leaving the consistency of $\Empirical \KernelMatrixParam[k]{\ParamEstimator[n]}$ unaffected.
We have found no such example in the literature, but the following elementary construction makes the point.

\begin{proposition}
\label{prop:no-continuity}
Let $\StateSpace = \Integers$ and, for $\theta \in \Reals$, let $\KernelMatrixParam{\theta}$ be defined by
\[
  4\KernelParam{\theta}{-1}{0} = 4\KernelParam{\theta}{1}{0}
  = 2 \KernelParam{\theta}{0}{0} = 1,
\]
and, for every $x \neq 0$,
\[
  \KernelParam{\theta}{x + \mathrm{sgn}(x)}{x}
  = 1 - \frac{1}{1 + |\theta| x^{2}}
  = 1 - \KernelParam{\theta}{0}{x}.
\]
\begin{itemize}
  \item If $\theta = 0$, the chain is irreducible, aperiodic, positive recurrent and uniformly ergodic, and admits a unique invariant distribution.
  \item If $\theta \neq 0$, the chain escapes to $+\infty$ or to $-\infty$ with positive probability, and no invariant probability distribution exists.
\end{itemize}
\end{proposition}
\begin{proof}
For $\theta = 0$, the chain is supported on $\{-1,0,1\}$, where it is irreducible and aperiodic.
In addition, Doeblin's condition holds since, for any $x$,
\[
  2 \KernelParam{0}{\cdot}{x} \geq \Dirac{0}(\cdot),
\]
which proves the first claim by~\cite[Theorem~16.0.2]{MeynTweedie2009}.
If $\theta \neq 0$, the probability
\[
  \Prob\left(\forall n \in \Naturals,\ X_n \neq 0 \mid X_0 = 1\right)
  = \prod_{x \geq 1} \left(1 - \frac{1}{1 + |\theta| x^{2}}\right)
\]
is positive since $\sum_{x \geq 1} (1 + |\theta| x^{2})^{-1} \leq |\theta|^{-1} \sum_{x \geq 1} x^{-2} < +\infty$.
\end{proof}

Section~\ref{sec:theory-parametric-kernel} studies the theoretical properties of a sample-split version of the plug-in look-ahead estimator.
We show that its MSE decomposes into the error of the look-ahead estimator with known kernel (investigated in Part~\ref{part:known-kernel}) and additional terms due to the estimation of $\ParamTrue$ (Theorem~\ref{thm:mse-expansion}).
Building on this expansion, we establish three results, according to the rate of convergence of the estimator of $\ParamTrue$.
When this estimator converges strictly faster than the parametric rate, the asymptotic MSE is exactly that of the look-ahead estimator with known kernel (Corollary~\ref{cor:data-partition2}).
When it converges strictly slower, the MSE of the empirical estimator has a strictly better order of convergence than that of the plug-in look-ahead estimator (Corollary~\ref{cor:data-partition1}).
Finally, Corollary~\ref{cor:data-partition3} addresses the most interesting case, where the estimator of $\ParamTrue$ converges at the same rate as the empirical estimator of $\pi$.
In that case, we identify precisely, through a necessary and sufficient condition on the transition kernel, the situations in which the asymptotic MSE of the plug-in look-ahead estimator improves on that of the empirical estimator.

Section~\ref{sec:birth-death} illustrates the asymptotic properties of the one-step plug-in look-ahead estimator on a birth-and-death process.
As an application of Corollary~\ref{cor:data-partition3}, Proposition~\ref{prop:birth-death-comparison} identifies the parameter values for which the asymptotic MSE of the sample-split plug-in look-ahead estimator improves on that of the empirical estimator.
The simulations of Subsection~\ref{subsec:birth-death-simulations} illustrate these theoretical results.
They also consider the plug-in look-ahead estimator without sample splitting, in which all the data are used both to estimate $\ParamTrue$ and to compute the empirical estimator.
The methods are finally applied to real data recording the availability of bikes in the bike-sharing system of the Lyon metropolitan area.
Taken together, the numerical results confirm the strong practical interest of plug-in look-ahead methods.

\part{Fully-specified transition kernel: an MCMC approach}
\label{part:known-kernel}

In this part, the transition kernel $\KernelMatrix$ is assumed to be known, and we are interested in estimating its unique invariant distribution $\pi$, or the mean value $\pi(f)$ of a functional $f$ under $\pi$, which places us exactly in an MCMC framework.

\section{Mean squared error of look-ahead estimators}
\label{sec:theory-known-kernel}

This section is dedicated to the theoretical study of look-ahead estimators in terms of bias and variance, from independent data as well as from data with Markovian dependence.

\subsection{Look-ahead estimation of the invariant distribution}
\label{subsec:known-kernel-framework}

We consider a transition kernel $\KernelMatrix$ on a countable state space $\StateSpace$.
Throughout this part, $\KernelMatrix$ is assumed to be irreducible, aperiodic, positive recurrent, and uniformly ergodic.
Under these assumptions, there exists a unique invariant probability distribution $\pi$ on $\StateSpace$, characterised by the invariance relation $\pi = \pi \KernelMatrix$.
Our general objective is to estimate the distribution $\pi$ by exploiting the knowledge of $\KernelMatrix$.
To this end, four observation schemes are considered throughout.
\begin{itemize}
  \item Either the data are distributed according to $\pi$, independently (scheme denoted by \FrameworkIidStationary{}) or with a Markovian dependence with kernel $\KernelMatrix$ (scheme denoted by \FrameworkMarkovStationary): such stationary data are studied in Subsection~\ref{subsec:known-kernel-stationary-data}.

  \item Or they are distributed according to a probability distribution different from $\pi$ (which should be thought of as close to $\pi$), again either independently (scheme denoted by \FrameworkIidGeneral) or with a Markovian dependence with kernel $\KernelMatrix$ (scheme denoted by \FrameworkMarkovGeneral): non-stationary data are studied in Subsection~\ref{subsec:known-kernel-non-stationary}.
    In the independent case, this observation scheme corresponds to generating several independent non-stationary chains with kernel $\KernelMatrix$, only the terminal state of each of them, at a finite horizon, being observed.
    In the Markovian case, it corresponds to observing a single non-stationary chain.
\end{itemize}
For the sake of readability, each statement below carries, next to its number, the observation scheme under which it is established.

In each of these four situations, we observe a data set $(X_i)_{\IndexBetween{i}{1}{n}}$ and, for every $\IndexBetween{i}{1}{n}$, the transition probabilities $\Kernel{\cdot}{X_i}$ are assumed to be known.
Our goal is to use both the observations and the knowledge of these transition probabilities to estimate the invariant distribution $\pi$.

To this end, we consider the family of look-ahead estimators~\cite{Braun2012,GlynnHenderson1998, HendersonGlynn2001,McKeagueWefelmeyer2000,Stachurski2008} given, for $k \in \Naturals$ and $x \in \StateSpace$, by
\begin{equation}\label{eq:estimator-pointwise}
  \PiHatAt{k}{x} = \frac{1}{n} \SumFromTo{i}{1}{n} \Contribution{k}{i}{x},
\end{equation}
where
\[
  \forall \IndexBetween{i}{1}{n},
  \quad
  \Contribution{k}{i}{x} = \Kernel[k]{x}{X_i}.
\]
It is worth noting that, for $k = 0$, the estimator $\Empirical$ coincides with the empirical estimator~\eqref{eq:empirical-estimator} of $\pi$.
In addition, one has the matrix form used in the references mentioned above,
\begin{equation}\label{eq:estimator-global}
  \PiHat{k} = \Empirical \KernelMatrix[k].
\end{equation}

\begin{remark}
\label{rem:local-versus-global}
From a theoretical perspective, \eqref{eq:estimator-pointwise} and \eqref{eq:estimator-global} define the same estimator.
From a computational perspective, however, the two representations differ significantly.
The formulation given in \eqref{eq:estimator-pointwise} is \emph{local}, since it only requires access to the transition probabilities $\Kernel[k]{\cdot}{y}$ for the states $y$ that are actually visited by the observations $(X_i)_{\IndexBetween{i}{1}{n}}$.
In contrast, the formulation \eqref{eq:estimator-global} is \emph{global}, as computing $\PiHat{k}$ requires access to the entire transition matrix $\KernelMatrix[k]$.
This distinction becomes particularly relevant when the state space is very large, or only implicitly specified, as noted in the introduction.
In such settings, storing or even accessing the full transition matrix $\KernelMatrix[k]$ may be computationally prohibitive.
As a consequence, the global approach may become impractical or impossible to implement, while the local formulation remains feasible since it only relies on the transition probabilities associated with the visited states.
\end{remark}

\subsection{Estimation from stationary data}
\label{subsec:known-kernel-stationary-data}

In this subsection, each $X_i$ is distributed according to the invariant distribution $\pi$.
The first result below shows that, whatever the dependence scheme of the data, the family of look-ahead estimators consists of consistent and unbiased estimators of $\pi$, two fundamental properties of the classical empirical estimator $\Empirical$.

\begin{proposition}[\FrameworkIidStationary\,$\cup$\,\FrameworkMarkovStationary]
\label{prop:consistency-unbiasedness}
Assume that the data $(X_i)_{\IndexBetween{i}{1}{n}}$ are either independent with common distribution $\pi$, or form a stationary Markov chain with kernel $\KernelMatrix$.
For $k \in \Naturals$, the estimator $\PiHat{k}$ satisfies, in the pointwise topology,
\[
  \PiHat{k} \ConvergesAlmostSurely \pi.
\]
Moreover, for any $x \in \StateSpace$, $\Expect[\PiHatAt{k}{x}] = \pi(x)$.
\end{proposition}

\begin{proof}
This result follows directly from the strong law of large numbers (for independent or Markovian data) and from the fact that, for any $x \in \StateSpace$ and $\IndexBetween{i}{1}{n}$,
\[
  \Expect\left[\Contribution{k}{i}{x}\right]
  = \SumOver{y} \pi(y) \Kernel[k]{x}{y} = \pi(x). \qedhere
\]
\end{proof}

As noted in the introduction, when $\KernelMatrix$ is not assumed to be known, the empirical estimator $\Empirical$ is the nonparametric maximum likelihood estimator of $\pi$, and therefore asymptotically minimises the variance among all unbiased estimators of $\pi$.
It is thus natural to ask what variance reduction is achieved by the look-ahead estimators $\PiHat{k}$, first under independence and then under Markovian dependence.
When the data are independent, the answer requires no asymptotic framework and takes the following simple form.

\begin{proposition}[\FrameworkIidStationary]
\label{prop:variance-iid-pointwise}
Assume that the data $(X_i)_{\IndexBetween{i}{1}{n}}$ are independent with common distribution $\pi$.
For $k \in \Naturals$ and $x \in \StateSpace$,
\[
  \Variance\left(\PiHatAt{k}{x}\right)
  = \frac{1}{n} \left(\SumOver{y} \pi(y) \Kernel[k]{x}{y}^{2} - \pi(x)^{2}\right).
\]
Moreover,
\[
  \Variance\left(\PiHatAt{k+1}{x}\right)
  \leq \Variance\left(\PiHatAt{k}{x}\right)
  \leq \Variance\left(\EmpiricalAt{x}\right).
\]
\end{proposition}
\begin{proof}
This result is a special case of Proposition~\ref{prop:variance-iid-functional}, whose proof is given in Appendix~\ref{app:proof-variance-monotone}.
\end{proof}

Under Markovian dependence, the situation is markedly more delicate and requires more work.
In particular, we shall only establish asymptotic results, starting with the following componentwise central limit theorem, whose main interest lies in the explicit expression it provides for the asymptotic variance of the limiting Gaussian distribution.

\begin{proposition}[\FrameworkMarkovStationary]
\label{prop:clt-pointwise}
Assume that the data $(X_i)_{\IndexBetween{i}{1}{n}}$ form a stationary Markov chain with kernel $\KernelMatrix$.
For $k \in \Naturals$ and $x \in \StateSpace$,
\[
  \sqrt{n}\left(\PiHatAt{k}{x} - \pi(x)\right)
  \ConvergesInDistribution \Gaussian{0}{\AsymptoticVariance[k]{x}},
\]
with
\begin{equation}\label{eq:formula:sigmakcarre}
  \AsymptoticVariance[k]{x}
  = \SumOver{y} \pi(y) \Kernel[k]{x}{y}^{2} - \pi(x)^{2}
  + 2 \SumFromTo{i}{1}{+\infty}
      \left(\SumOver{y} \pi(y) \Kernel[k]{x}{y}
      \left\{\Kernel[i+k]{x}{y} - \pi(x)\right\}\right).
\end{equation}
\end{proposition}
\begin{proof}
This result is a special case of Proposition~\ref{prop:clt-functional}, whose proof is given in Appendix~\ref{app:proof-clt}.
\end{proof}

Summing over the state space, we also obtain the following global result.

\begin{lemma}[\FrameworkMarkovStationary]
\label{lem:summability}
Assume that the data $(X_i)_{\IndexBetween{i}{1}{n}}$ form a stationary Markov chain with kernel $\KernelMatrix$.
For $k \in \Naturals$,
\[
  \SumOver{x} \AsymptoticVariance[k]{x} < +\infty,
\]
and, when $n$ goes to infinity,
\begin{equation}\label{eq:sumvar:markov:stationary}
  \SumOver{x} \Variance\left(\PiHatAt{k}{x}\right)
  = \frac{1}{n} \SumOver{x} \AsymptoticVariance[k]{x}
  + \LittleO\!\left(\frac{1}{n}\right).
\end{equation}
\end{lemma}

\ProofDeferred{app:proof-summability}

Since the look-ahead estimators are unbiased by Proposition~\ref{prop:consistency-unbiasedness}, this shows that
\begin{equation}\label{eq:mseinfty:stationaryMC}
    \AsymptoticMSE\left(\PiHat{k}\right) = \SumOver{x} \AsymptoticVariance[k]{x},
\end{equation}
where $\AsymptoticMSE$ is defined in~\eqref{eq:def:amse}.
Consequently, comparing asymptotic MSEs amounts to studying the variations of the sequence $(\SumOver{x} \AsymptoticVariance[k]{x})_{k \in \Naturals}$, which is the content of the following theorem.

\begin{theorem}[\FrameworkMarkovStationary]
\label{thm:monotone-amse}
Assume that the data $(X_i)_{\IndexBetween{i}{1}{n}}$ form a stationary Markov chain with kernel $\KernelMatrix$.
It holds that
\[
  \SumOver{x} \AsymptoticVariance[1]{x}
  \leq \SumOver{x} \AsymptoticVariance[0]{x}.
\]
If $\KernelMatrix$ is also assumed to be reversible, then for any $k \in \Naturals$,
\[
  \SumOver{x} \AsymptoticVariance[k+1]{x}
  \leq \SumOver{x} \AsymptoticVariance[k]{x}.
\]
On the contrary, there exists a non-reversible kernel $\KernelMatrix$ such that
\[
\SumOver{x} \AsymptoticVariance[2]{x}
  > \SumOver{x} \AsymptoticVariance[1]{x}.
\]
\end{theorem}
\ProofDeferred{proof:thm:monotone-amse}

This result extends the main theorem of~\cite{McKeagueWefelmeyer2000} in several respects.
First, it shows that the one-step look-ahead estimator $\PiHat{1}$ asymptotically performs at least as well as $\Empirical$ even without the reversibility assumption, which is an original result.
Moreover, it establishes the decrease of the asymptotic MSE for the estimation of the invariant distribution in the reversible setting.
Finally, without claiming that reversibility is a minimal assumption, it shows that it is a relevant one, since there exists a non-reversible transition kernel for which the asymptotic variance does not decrease with $k$.

\subsection{Estimation from non-stationary data}
\label{subsec:known-kernel-non-stationary}

In this subsection, we extend the analysis to the case where the data do not necessarily follow the invariant distribution $\pi$.
In the independent setting, all the observations are distributed according to a probability distribution $\mu$ assumed to differ from $\pi$.
In the Markovian setting, the initial state $X_1$ is distributed according to $\mu$ and the chain then evolves according to the kernel $\KernelMatrix$, that is, $X_i$ is distributed according to $\mu \KernelMatrix[i-1]$.

In contrast with the stationary setting, the look-ahead estimators, $\Empirical$ included, are in general biased as soon as $\mu \neq \pi$.
Whether the data are independent or Markovian, the (non-asymptotic) bias is studied in the following proposition.

\begin{proposition}[\FrameworkIidGeneral\,$\cup$\,\FrameworkMarkovGeneral]
\label{prop:bias-decreasing-iid}
Assume that the data $(X_i)_{\IndexBetween{i}{1}{n}}$ are either independent with common distribution $\mu$, or form a Markov chain with kernel $\KernelMatrix$ and initial distribution $\mu$.
Suppose that $\KernelMatrix$ is doubly stochastic.
Then, for $k \in \Naturals$,
\[
  \SumOver{x} \Bias^{2}\!\left(\PiHatAt{k+1}{x}\right) \leq \SumOver{x} \Bias^{2}\!\left(\PiHatAt{k}{x}\right).
\]
On the contrary, there exists a reversible kernel $\KernelMatrix$ that is not doubly stochastic and such that
\[
  \SumOver{x} \Bias^{2}\!\left(\PiHatAt{1}{x}\right) > \SumOver{x} \Bias^{2}\!\left(\PiHatAt{0}{x}\right).
\]
\end{proposition}

\ProofDeferred{app:proof-bias-decreasing-iid}

Hence, in the non-stationary setting, even for a reversible kernel, the bias may increase when $\KernelMatrix$ is applied to $\Empirical$.
Under the doubly stochastic assumption, by contrast, the bias always decreases.
On the other hand, for independent data, the variance is always reduced compared to the standard empirical estimator, even for a finite number of observations, as stated in the next proposition.

\begin{proposition}[\FrameworkIidGeneral]
\label{prop:variance-nonstationary-iid}
Assume that the data $(X_i)_{\IndexBetween{i}{1}{n}}$ are independent with common distribution $\mu$.
For $k \in \Naturals$,
\[
  \SumOver{x} \Variance\!\left(\PiHatAt{k}{x}\right)
  \leq \SumOver{x} \Variance\!\left(\EmpiricalAt{x}\right).
\]
\end{proposition}

\ProofDeferred{app:proof-variance-nonstationary-iid}

As a consequence, the independent setting still guarantees a systematic variance reduction over the standard empirical estimator, reflecting an improvement akin to the stationary case.
However, the bias term may increase with $k$ and always persists as $n$ tends to infinity.
The look-ahead estimators are therefore not necessarily better than $\Empirical$ in terms of (asymptotic or not) MSE.
In the Markovian setting, the situation is different: the squared bias may increase, but it remains of order $\BigO(n^{-2})$, as shown in the next proposition.

\begin{proposition}[\FrameworkMarkovGeneral]
\label{prop:bias-negligible}
Assume that the data $(X_i)_{\IndexBetween{i}{1}{n}}$ form a Markov chain with kernel $\KernelMatrix$ and initial distribution $\mu$.
For $k \in \Naturals$,
\[
  \SumOver{x} \Bias^{2}\!\left(\PiHatAt{k}{x}\right)
  = \BigO\!\left(\frac{1}{n^{2}}\right).
\]
\end{proposition}

\ProofDeferred{app:proof-bias-negligible}

Moreover, the asymptotic expansion of the variance given in Lemma~\ref{lem:summability} for stationary data remains valid for non-stationary data.

\begin{proposition}[\FrameworkMarkovGeneral]
\label{prop:variance-nonstationary-markov}
Assume that the data $(X_i)_{\IndexBetween{i}{1}{n}}$ form a Markov chain with kernel $\KernelMatrix$ and initial distribution $\mu$.
For $k \in \Naturals$, when $n$ goes to infinity,
\[
  \SumOver{x} \Variance\!\left(\PiHatAt{k}{x}\right)
  = \frac{1}{n} \SumOver{x} \AsymptoticVariance[k]{x}
  + \LittleO\!\left(\frac{1}{n}\right).
\]
\end{proposition}

\ProofDeferred{app:proof-variance-nonstationary-markov}

Combining the two previous results, although the squared bias of the look-ahead estimators may increase with $k$ (Proposition~\ref{prop:bias-decreasing-iid}), it is negligible with respect to the variance, which satisfies the same asymptotic formula~\eqref{eq:sumvar:markov:stationary} as in the stationary case.
As a consequence, the asymptotic MSE of the look-ahead estimators from non-stationary Markovian data is still given by~\eqref{eq:mseinfty:stationaryMC}, and Theorem~\ref{thm:monotone-amse} applies unchanged: the one-step estimator asymptotically performs at least as well as the empirical one, and the asymptotic MSE decreases with $k$ for reversible kernels.

\begin{remark}
\label{rem:uniformity-initial-distribution}
It is worth noting that Propositions~\ref{prop:bias-negligible} and~\ref{prop:variance-nonstationary-markov} hold uniformly in the initial distribution $\mu$.
Indeed, $\mu$ only enters the proofs of these two results through the bound $\NormTV{\mu \KernelMatrix[i] - \pi} \leq C \rho^{i}$, where the constants $C > 0$ and $\rho \in [0,1[$ given by uniform ergodicity do not depend on $\mu$ (see Appendices~\ref{app:proof-bias-negligible} and~\ref{app:proof-variance-nonstationary-markov}).
This uniformity is required for the theoretical results of Section~\ref{sec:theory-parametric-kernel}.
\end{remark}


\subsection{\texorpdfstring{Estimation of $\pi(f)$}{Estimation of π(f)}}
\label{subsec:known-kernel-functional-expectation}

It is common in estimation problems that, even though the distribution $\pi$ is unknown, it is not the quantity of interest: one rather wants to estimate $\pi(f)$ for some function $f \in \Ltwopi$.
For Markovian data with a reversible kernel, the reduction of the asymptotic MSE of the look-ahead estimators $\PiHatOf{k}{f}$ with respect to the empirical estimator $\EmpiricalOf{f}$ is established in~\cite{McKeagueWefelmeyer2000}.
We complement that result by a non-asymptotic study of the MSE from independent data distributed according to $\pi$.

\begin{proposition}[\FrameworkIidStationary]
\label{prop:variance-iid-functional}
Assume that the data $(X_i)_{\IndexBetween{i}{1}{n}}$ are independent with common distribution $\pi$.
For $k \in \Naturals$, the estimator $\PiHatOf{k}{f}$ is square-integrable and satisfies
\[
  \Variance\left(\PiHatOf{k}{f}\right)
  = \frac{1}{n} \left(\SumOver{y} \pi(y)\left(\SumOver{z} f(z) \Kernel[k]{z}{y}\right)^{2}
    - \pi(f)^{2}\right).
\]
Moreover,
\[
  \Variance\left(\PiHatOf{k+1}{f}\right)
  \leq \Variance\left(\PiHatOf{k}{f}\right)
  \leq \Variance\left(\EmpiricalOf{f}\right).
\]
\end{proposition}

\ProofDeferred{app:proof-variance-monotone}

In the stationary independent setting, we therefore observe a decrease of the MSE for the estimation of $\pi(f)$, with no additional requirement on $\KernelMatrix$ (in particular, regardless of whether it is reversible).
In the stationary Markovian case, without investigating the variance reduction properties brought out in~\cite{McKeagueWefelmeyer2000}, we state below a central limit theorem for $\PiHatOf{k}{f}$, similar to Proposition~\ref{prop:clt-pointwise}, which provides an explicit formula for the asymptotic variance.

\begin{proposition}[\FrameworkMarkovStationary]
\label{prop:clt-functional}
Assume that the data $(X_i)_{\IndexBetween{i}{1}{n}}$ form a stationary Markov chain with kernel $\KernelMatrix$.
For $k \in \Naturals$, when $n$ goes to infinity,
\[
  \sqrt{n}\left(\PiHatOf{k}{f} - \pi(f)\right)
  \ConvergesInDistribution \Gaussian{0}{\AsymptoticVariance[k]{f}},
\]
where, writing $\KernelAverage{k}{y}{f} = \SumOver{z} f(z)\Kernel[k]{z}{y}$ for $y \in \StateSpace$,
\[
  \AsymptoticVariance[k]{f}
  = \SumOver{y} \pi(y) \KernelAverage{k}{y}{f}^{2} - \pi(f)^{2}
  + 2 \SumFromTo{i}{1}{+\infty}
     \left(\SumOver{y,z} \pi(y)\Kernel[i]{z}{y}
     \KernelAverage{k}{y}{f}\KernelAverage{k}{z}{f} - \pi(f)^{2}\right).
\]
\end{proposition}

\ProofDeferred{app:proof-clt}

\subsection{Aggregation of look-ahead estimators}
\label{subsec:known-kernel-aggregation}

As shown above, the reduction of the MSE is not a systematic property of look-ahead methods.
The most emblematic case is probably the non-reversible stationary Markovian setting, where the asymptotic variance may increase when $\KernelMatrix$ is applied to $\PiHat{1}$.
In any case, when $\KernelMatrix$ is known and in the stationary setting, look-ahead methods provide a family of consistent and unbiased estimators of $\pi(f)$.
A natural way to reduce the MSE is then to aggregate these estimators.
To this end, we consider a finite set $I \subset \Naturals$ and (possibly negative) weights $(\lambda_k)_{k \in I}$ summing to $1$.
The aggregated look-ahead estimator is then defined by
\[
  \PiHatAggregated{(\lambda_k)_{k \in I}}{f}
  = \sum_{k \in I} \lambda_k \PiHatOf{k}{f}.
\]
In the stationary setting, this estimator remains unbiased.
The question then arises of identifying optimal weights that minimise its variance.
An answer is given by the following general result (see~\cite{LavancierRochet2016} and references therein), expressed in our context in terms of the covariance matrix of look-ahead estimators,
\[
    \Sigma(f) = \left(\Covariance\left(\PiHatOf{k}{f},\PiHatOf{l}{f}\right)\right)_{k,l \in I}.
\]

\begin{proposition}
\label{prop:optimal-weights}
If $\Sigma(f)$ is non-degenerate, the weights summing to $1$ that minimise the variance of $\PiHatAggregated{(\lambda_k)_{k \in I}}{f}$ are given by
\[
  \lambda^{*}(f)
  = \frac{\Sigma(f)^{-1} \Ones}{\Ones\Transpose \Sigma(f)^{-1} \Ones}.
\]
\end{proposition}
\begin{proof}
See for instance~\cite[Subsection~2.1]{LavancierRochet2016}.
\end{proof}

For look-ahead estimators, determining the optimal weights therefore requires identifying the cases in which $\Sigma(f)$ is non-degenerate and evaluating its coefficients.
Remarkably, the non-degeneracy of $\Sigma(f)$ is expressed simply in terms of $\KernelMatrix$ and $f$, whether the data are independent or Markovian.

\begin{proposition}[\FrameworkIidStationary\,$\cup$\,\FrameworkMarkovStationary]
\label{prop:invertibility-sigma}
Assume that the data $(X_i)_{\IndexBetween{i}{1}{n}}$ are either independent with common distribution $\pi$, or form a stationary Markov chain with kernel $\KernelMatrix$.
Then $\Sigma(f)$ is non-degenerate if and only if the family $\{\Ones\} \cup \{\KernelMatrix[k] f\}_{k \in I}$ is linearly independent in $\Reals^{\StateSpace}$.
\end{proposition}

\ProofDeferred{app:proof-inversibility}

The two results below provide an exact evaluation of $\Sigma(f)$ for independent and Markovian data.
In practice, having an explicit formula eliminates the need to evaluate the estimators multiple times to estimate their covariance, thereby allowing a direct application of the optimal aggregation given in Proposition~\ref{prop:optimal-weights}.

\begin{proposition}[\FrameworkIidStationary]
\label{prop:covariance-iid}
Assume that the data $(X_i)_{\IndexBetween{i}{1}{n}}$ are independent with common distribution $\pi$.
For $k,l \in \Naturals$,
\[
  \Covariance\left(\PiHatOf{k}{f},\PiHatOf{l}{f}\right)
  = \frac{1}{n}\left\{\SumOver{y} \pi(y)
    \KernelAverage{k}{y}{f} \KernelAverage{l}{y}{f} - \pi(f)^{2}\right\}.
\]
\end{proposition}

\ProofDeferred{app:proof-covariance-iid}

\begin{proposition}[\FrameworkMarkovStationary]
\label{prop:covariance-markov}
Assume that the data $(X_i)_{\IndexBetween{i}{1}{n}}$ form a stationary Markov chain with kernel $\KernelMatrix$.
For $k,l \in \Naturals$,
\[
\begin{aligned}
  &\Covariance\left(\PiHatOf{k}{f},\PiHatOf{l}{f}\right)
  = \frac{1}{n} \left\{\SumOver{y} \pi(y)
     \KernelAverage{k}{y}{f}\KernelAverage{l}{y}{f} - \pi(f)^{2}\right\} \\
  &\enspace + \frac{1}{n^{2}} \SumFromTo{i}{1}{n-1} (n-i)
     \left(\SumOver{y,z} \pi(y)\Kernel[i]{z}{y}
     \left\{\KernelAverage{k}{y}{f}\KernelAverage{l}{z}{f}
     + \KernelAverage{l}{y}{f}\KernelAverage{k}{z}{f}\right\}
     - 2\pi(f)^{2}\right).
\end{aligned}
\]
\end{proposition}

\ProofDeferred{app:proof-covariance-markov}

\section{Look-ahead estimation in the random surfer model}
\label{sec:random-surfer}

This section is devoted to the illustration and the empirical validation of the theoretical results established for look-ahead methods in the previous section, through the random surfer model on randomly generated graphs.
It is also devoted to putting look-ahead methods to work in estimating the invariant distribution of that same model on real public transport networks.

\subsection{Simulation study on Erd\H{o}s--R\'enyi graphs}
\label{subsec:random-surfer-simulations}

We study numerically the look-ahead methods on a random walk model called the random surfer model, applied to connected and aperiodic Erd\H{o}s--R\'enyi graphs.

Recall that the Erd\H{o}s--R\'enyi model $G(N,p)$~\cite{ER1959,Gilbert1959} produces a random graph on $N$ labelled vertices in which each of the $N(N-1)/2$ possible edges is present independently of the others with probability $p$.
Such a graph is neither connected nor aperiodic in general, so we condition on the event that it is connected and not bipartite: on that event the random walk below is irreducible and aperiodic, since a connected graph carries a periodic walk exactly when it is bipartite.

On a realisation $G = (V,E)$ of that conditioned model, the random surfer model (in its undamped form)~\cite{BP1998, PBMW1999} moves from the current vertex to one of its neighbours chosen uniformly at random, that is,
\[
  \Kernel{y}{x} =
  \begin{cases}
    \dfrac{1}{\deg(x)} & \text{if } \{x,y\} \in E, \\[1.5ex]
    0 & \text{otherwise},
  \end{cases}
\]
where $\deg(x)$ denotes the degree of the vertex $x$.
The kernel $\KernelMatrix$ is reversible and its invariant distribution is explicit~\cite[Section~1]{Lovasz1993},
\[
  \forall x \in V, \quad \pi(x) \propto \deg(x).
\]
In this simulation study, we illustrate the behaviour of the look-ahead estimators $\PiHat{k} = \Empirical \KernelMatrix[k]$ on data with Markovian dependence with kernel $\KernelMatrix$.

For Erd\H{o}s--R\'enyi graphs of size $N = 20$ and edge probability $p = 0.5$, and for non-stationary chains started from the uniform distribution, Figure~\ref{fig:surfer-bias-variance} illustrates the decrease of the squared bias and of the variance as the horizon of the chain grows: the empirical biases and variances, estimated from $10^{6}$ replicates, fit the theoretical curves perfectly.
It should be noted that formula~\eqref{eq:formula:bias:markov} for the theoretical bias is non-asymptotic.
For the theoretical variance, we use the asymptotic approximation established in Proposition~\ref{prop:variance-nonstationary-markov} together with formula~\eqref{eq:totalvariance:from:clt}, established in the proof of Theorem~\ref{thm:monotone-amse}, which we prefer to formula~\eqref{eq:formula:sigmakcarre} since the latter involves infinite sums.

\begin{figure}[htbp]
  \centering
  \begin{subfigure}{\PlotWidth}
    \centering
    \includegraphics[width=\linewidth]{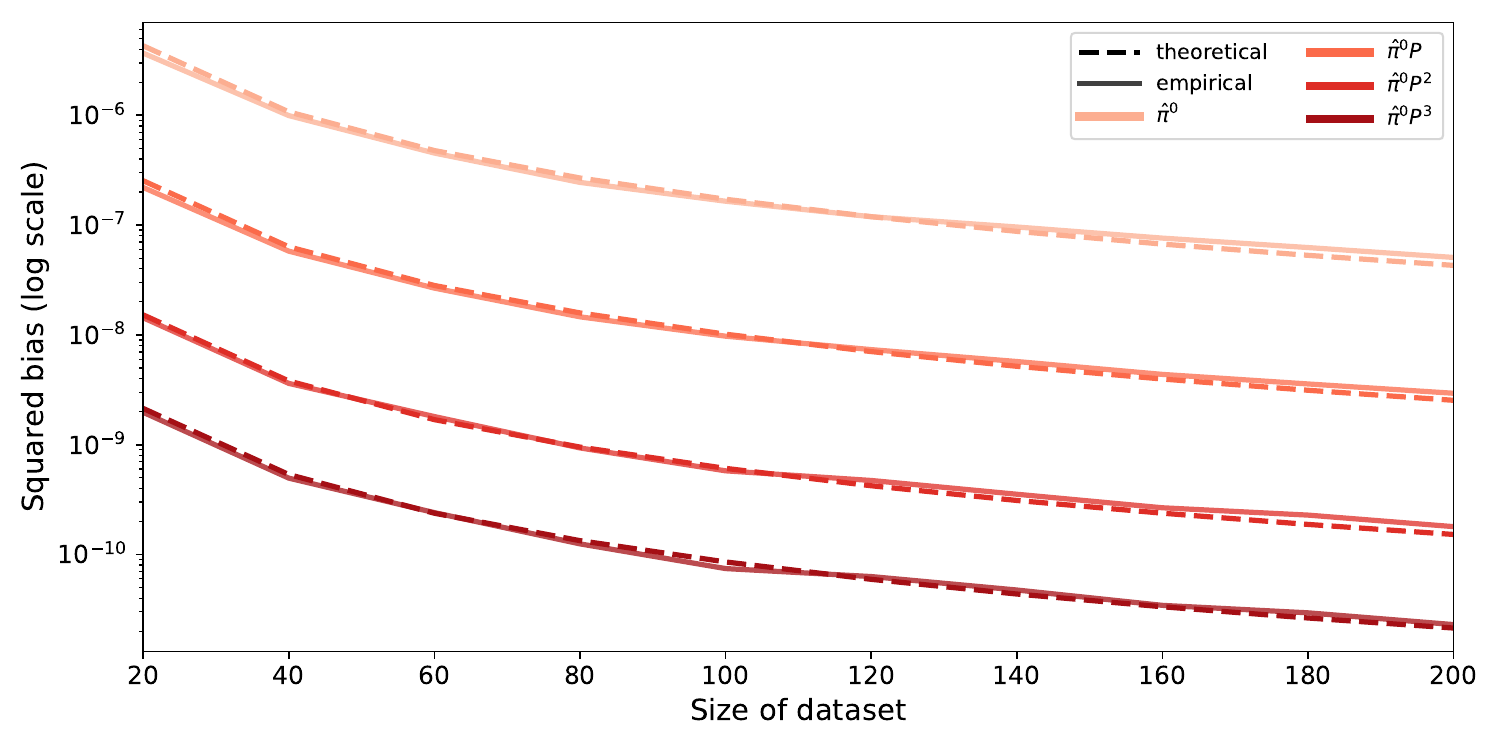}
    \caption{}
    \label{fig:surfer-bias}
  \end{subfigure}

  \vspace{\PlotGap}

  \begin{subfigure}{\PlotWidth}
    \centering
    \includegraphics[width=\linewidth]{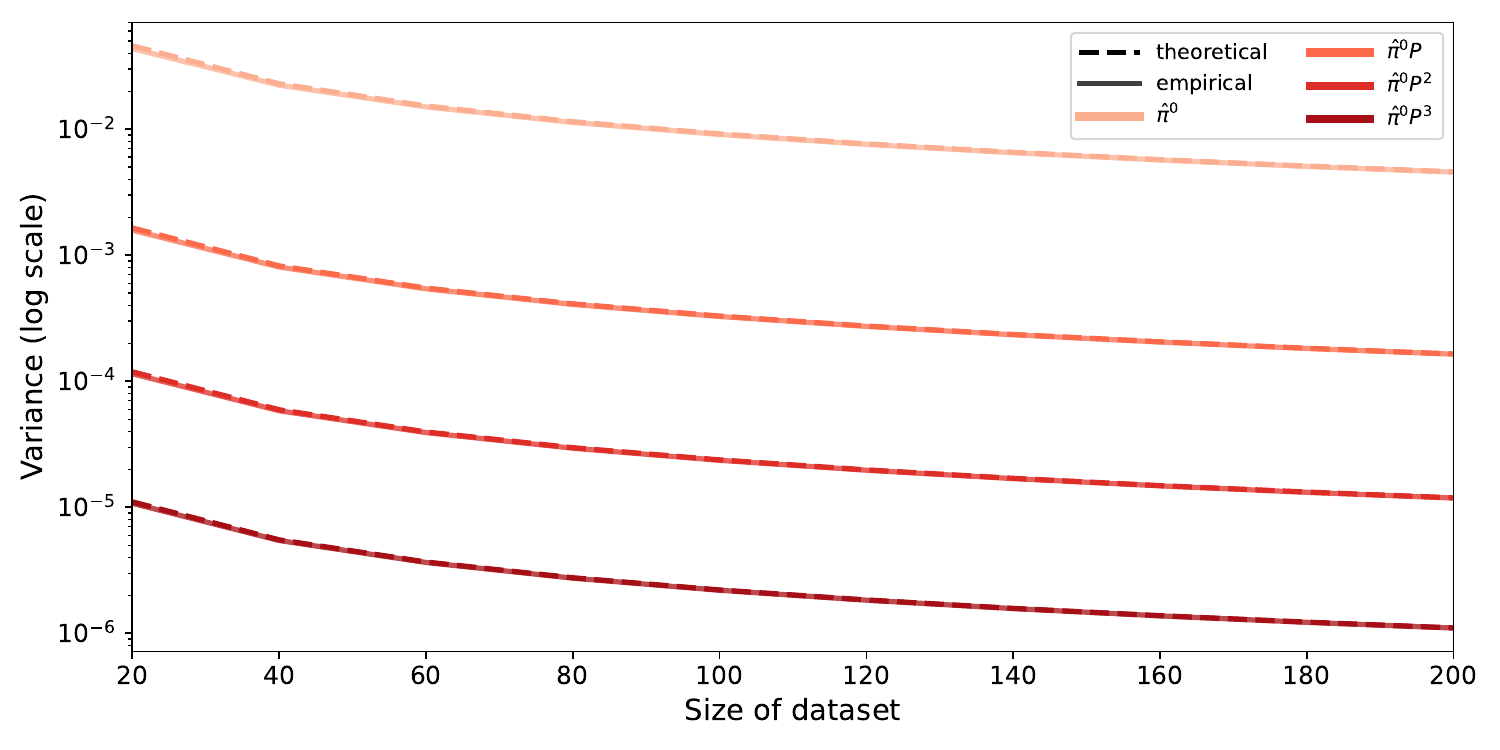}
    \caption{}
    \label{fig:surfer-variance}
  \end{subfigure}
  \caption{Behaviour of look-ahead estimators $\Empirical$, $\PiHat{1} = \Empirical \KernelMatrix$, $\PiHat{2} = \Empirical \KernelMatrix[2]$ and $\PiHat{3} = \Empirical \KernelMatrix[3]$ as a function of the size of the data set.
    Dashed lines are the theoretical values, solid lines the empirical ones.
    \textbf{A}~Squared bias.
    \textbf{B}~Variance.}
  \label{fig:surfer-bias-variance}
\end{figure}

Figure~\ref{fig:surfer-bias-variance} also brings out the incremental improvement of the look-ahead estimators, in both squared bias and variance, in a non-stationary setting.
Since the kernel $\KernelMatrix$ is reversible, the decrease of the variance is predicted, at least asymptotically, by Theorem~\ref{thm:monotone-amse}.
The observed decrease of the squared bias, on the other hand, does not fall within the assumptions of Proposition~\ref{prop:bias-decreasing-iid}, since the kernel is not doubly stochastic.

We now look for criteria that quantify the level of improvement brought by the look-ahead estimators.
A good candidate is the spectral gap of $\KernelMatrix$: since $\KernelMatrix$ is reversible, it is self-adjoint on $\Ltwopi$ and its eigenvalues are real, so that, ordering them as $1 = \lambda_1 > \lambda_2 \geq \dots \geq \lambda_N \geq -1$, one may define
\[
  \gamma = 1 - \max_{\IndexBetween{j}{2}{N}} \left|\lambda_j\right| .
\]
It is the relevant quantity here because it measures the geometric rate at which the chain forgets its initial condition: in $\Ltwopi$, each step of the chain multiplies the distance to equilibrium by at most $1-\gamma$.
The inverse of the spectral gap is thus interpreted as the relaxation time.
By tuning the parameters $N$ and $p$ of the Erd\H{o}s--R\'enyi graphs, the spectral gap of the random surfer model can easily be adjusted, as illustrated in Figure~\ref{fig:surfer-spectralgap-degree}.

\begin{figure}[htbp]
  \centering
  \includegraphics[width=\PlotWidth]{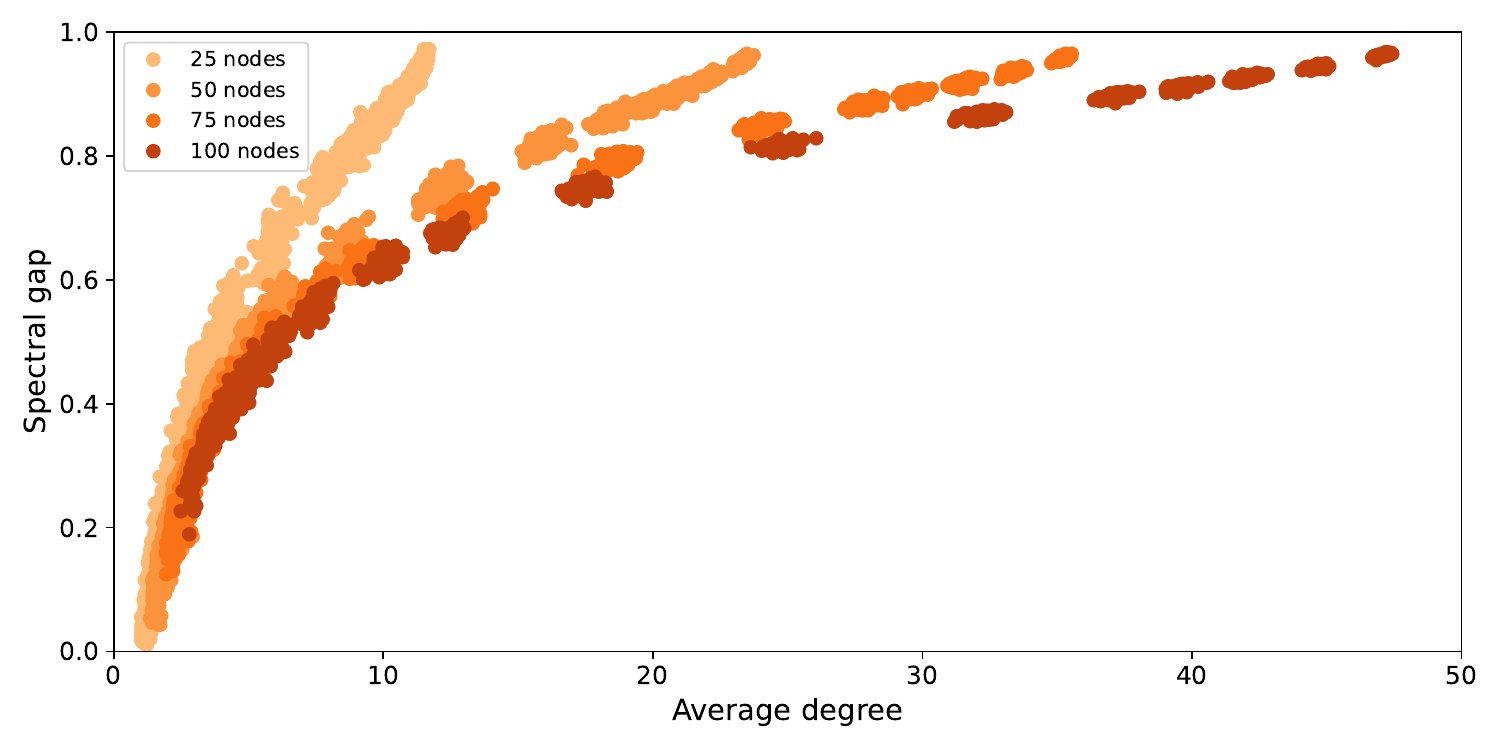}
  \caption{Spectral gap of the random surfer model on Erd\H{o}s--R\'enyi graphs of size $25$, $50$, $75$ and $100$ nodes as a function of the average degree.}
  \label{fig:surfer-spectralgap-degree}
\end{figure}

From stationary observations, for which the estimators are unbiased by Proposition~\ref{prop:consistency-unbiasedness}, Figure~\ref{fig:surfer-error-ratio} brings out that the improvement factor of the look-ahead methods, both in asymptotic variance and in the distribution of the squared error, depends significantly on the spectral gap: the larger the spectral gap, the greater the benefit of the look-ahead estimators, which is in line with intuition.
A large spectral gap means that the chain forgets its current position quickly, so that $\Kernel{\cdot}{X_i}$ is already close to $\pi$ and retains little of the randomness of $X_i$.
Averaging these conditional distributions can therefore remove a large part of the fluctuation.
A small gap, on the contrary, leaves $\Kernel{\cdot}{X_i}$ close to $\Dirac{X_i}$, and the look-ahead step changes almost nothing.

\begin{figure}[htbp]
  \centering
  \begin{subfigure}{\PlotWidth}
    \centering
    \includegraphics[width=\linewidth]{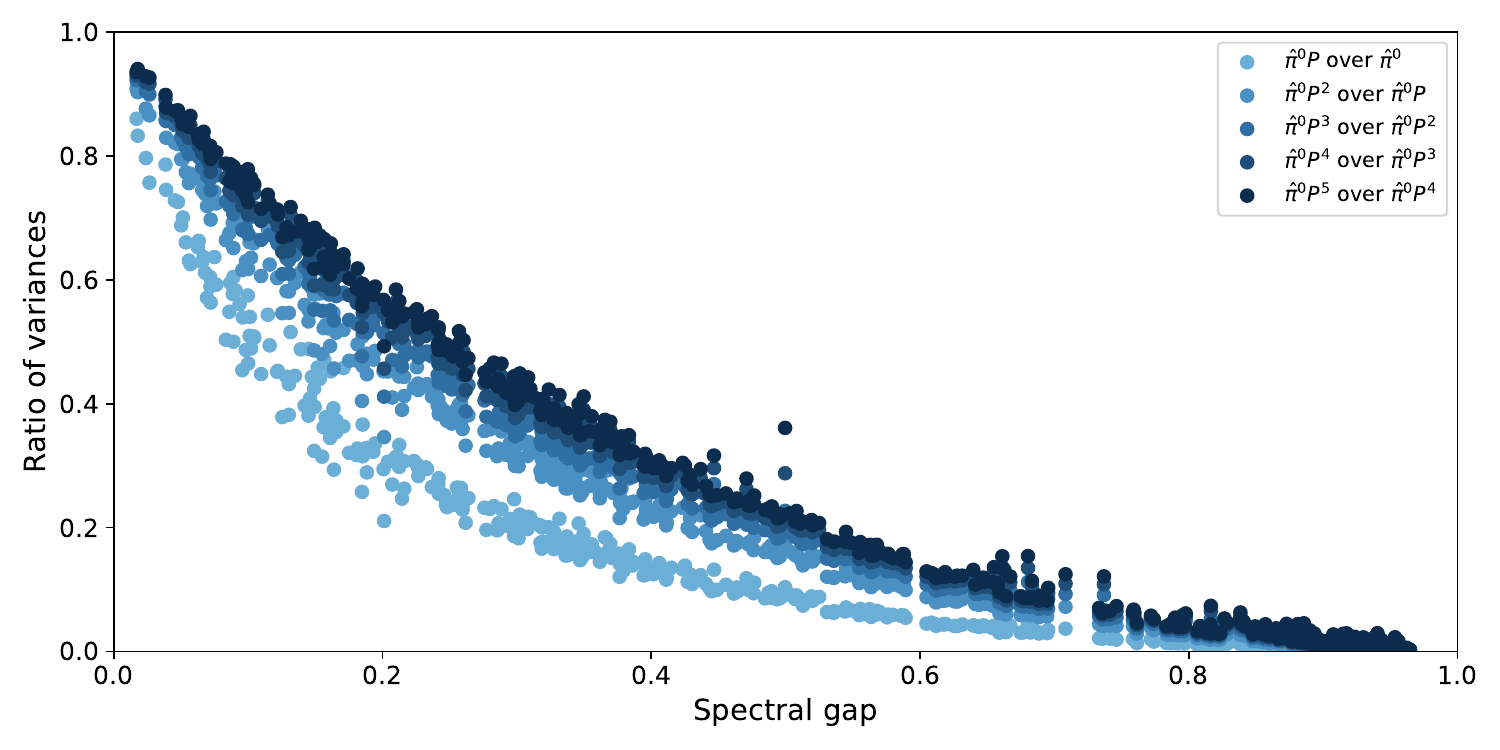}
    \caption{}
    \label{fig:surfer-ratiovar}
  \end{subfigure}

  \vspace{\PlotGap}

  \begin{subfigure}{\PlotWidth}
    \centering
    \includegraphics[width=\linewidth]{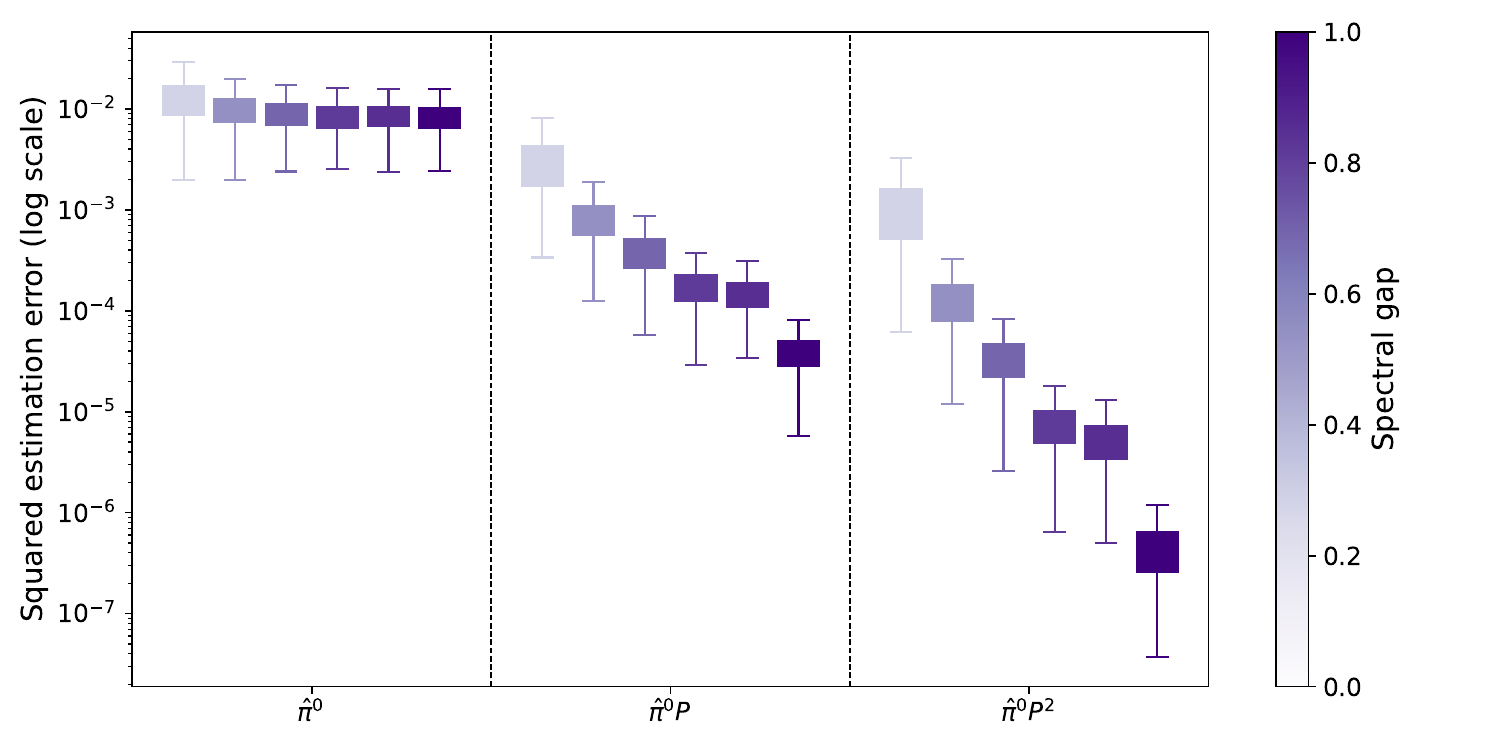}
    \caption{}
    \label{fig:surfer-boxplot}
  \end{subfigure}
  \caption{Influence of the spectral gap.
    \textbf{A}~Ratio of the variances of two consecutive look-ahead estimators as a function of the spectral gap.
    \textbf{B}~Boxplots over $200$ replicates of the squared estimation error of look-ahead estimators $\Empirical$, $\PiHat{1} = \Empirical \KernelMatrix$ and $\PiHat{2} = \Empirical \KernelMatrix[2]$, the shade of each box encoding the spectral gap.}
  \label{fig:surfer-error-ratio}
\end{figure}

\subsection{Analysis of real data: the public transport network of Lyon}
\label{subsec:random-surfer-real-data}

Beyond web search, for which it was introduced, the random surfer model is used across a wide range of domains as surveyed in~\cite{Gleich2015}: classical instances include the ranking of scientific publications in citation networks~\cite{CXMR2007} and the assessment of the importance of species in food webs~\cite{AP2009}.
Transport networks provide another typical example: the score that the random surfer attaches to a street predicts pedestrian and vehicle flows from the sole topology of the street network~\cite{Jiang2009}, and variants of the model are used to rank the nodes of an urban network~\cite{AOTV2012}.

In this subsection, we are interested in applying the random surfer model to the public transport network of the Lyon metropolitan area, France.
The transport network is seen as a temporal graph built from the General Transit Feed Specification (GTFS) feed of the Lyon public transport network (TCL), released as open data by the Grand Lyon metropolitan authority and valid from 15 October 2025 to 13 January 2026.\footnote{\href{https://data.grandlyon.com/portail/en/jeux-de-donnees/horaires-theoriques-reseau-transports-commun-lyonnais/info}{Grand Lyon open data portal}, data set \emph{Horaires th\'eoriques du r\'eseau Transports en Commun Lyonnais}.}
The feed describes the theoretical timetable of the network: its stops, its lines, and the scheduled time at which every vehicle run calls at every stop.

In what follows, the nodes of the network are fixed and consist of a spatially aggregated version of the stops appearing in the feed.
Indeed, a single location appears in the feed as several stops, typically one per direction of travel and sometimes one per line.
Stops carrying the same name, or very similar names, and lying very close to one another are merged into a single node, which simplifies the graph considerably while still describing faithfully the topology of the transport network and the dynamics that take place on it.
This yields $4\,329$ vertices, each placed at the mean position of the stops it aggregates.

The day is divided into twenty-four one-hour windows, and one weighted undirected graph is built for each of them.
Two vertices are joined whenever some vehicle run calls consecutively, within the window, at a stop of the first and at a stop of the second, and the weight $w(x,y)$ of the edge $\{x,y\}$ counts how many runs do so.
The vertex set is common to the twenty-four graphs, so that only the edges and their weights vary over the day.
Figure~\ref{fig:surfer-realdata-features} reports the main features (size and spectral gap) of the transport network at each hour of the day.

At each hour of the day, the random surfer model is applied to the largest connected component of the transport network just constructed, so as to model the importance of each of its nodes over time.
Here the surfer takes the edge weights into account: from a vertex $x$, it moves to a neighbour $y$ with probability proportional to the weight $w(x,y)$ of the edge joining them,
\[
  \Kernel{y}{x} = \frac{w(x,y)}{\sum_{z \in V} w(x,z)},
\]
which is, on a weighted graph, the natural counterpart of the uniform choice made in Subsection~\ref{subsec:random-surfer-simulations}.
The chain remains reversible, and its invariant distribution is proportional to the weighted degrees,
\[
	  \forall x \in V, \quad \pi(x) \propto \sum_{y \in V} w(x,y),
\]
which provides the ground truth against which the estimators are assessed below.
To estimate that distribution, we apply the look-ahead method in a stationary Markov framework, independently at each hour.

Figure~\ref{fig:surfer-realdata-features} shows that the topological properties of the network (including the spectral gap) vary significantly over the course of the day.
In the light of the numerical results of Subsection~\ref{subsec:random-surfer-simulations}, reaching the same target MSE at every hour requires the horizon of the simulated Markov chains to depend strongly on the network.
To that end, we fix a target MSE of $10^{-5}$ for $\Empirical$.
Using the theoretical results of Section~\ref{sec:theory-known-kernel}, in particular formula~\eqref{eq:mseinfty:stationaryMC} and Proposition~\ref{prop:clt-pointwise}, we are able to determine the horizon $n$ required for the MCMC simulations to reach that target as $n = 10^{5}\,\SumOver{x} \AsymptoticVariance[0]{x}$.
It is worth noting that formula~\eqref{eq:formula:sigmakcarre} involves infinite sums.
We therefore prefer formula~\eqref{eq:totalvariance:from:clt}, established in the proof of Theorem~\ref{thm:monotone-amse}.
Table~\ref{tab:variances} gives the sums of the theoretical variances for the estimators $\Empirical$, $\PiHat{1} = \Empirical \KernelMatrix$ and $\PiHat{2} = \Empirical \KernelMatrix[2]$ of the invariant distribution of the transport network at each hour of the day.
Figure~\ref{fig:surfer-realdata-estim-error} gives the data sizes thus inferred, together with the boxplots of the squared error over $200$ replicates of these same estimators.

\begin{table}[htbp]
  \centering
  \caption{Global asymptotic variance $\SumOver{x}\AsymptoticVariance[k]{x}$ of the look-ahead estimators $\PiHat{k} = \Empirical \KernelMatrix[k]$, $\IndexBetween{k}{0}{2}$, on the largest connected component of the transport network at each hour of the day.}
  \label{tab:variances}
  \footnotesize
  \setlength{\tabcolsep}{3.5pt}
  \begin{tabular}{l*{12}{r}}
    \toprule
    & 00{:}00 & 01{:}00 & 02{:}00 & 03{:}00 & 04{:}00 & 05{:}00
    & 06{:}00 & 07{:}00 & 08{:}00 & 09{:}00 & 10{:}00 & 11{:}00 \\
    \midrule
    $\SumOver{x}\AsymptoticVariance[0]{x}$
      & $11.94$ & $12.67$ & $9.85$ & $10.17$ & $18.89$ & $11.56$
      & $11.73$ & $10.48$ & $13.22$ & $14.91$ & $28.51$ & $35.89$ \\
    $\SumOver{x}\AsymptoticVariance[1]{x}$
      & $10.54$ & $11.26$ & $8.45$ & $8.76$ & $17.46$ & $10.16$
      & $10.33$ & $9.10$ & $11.83$ & $13.52$ & $27.12$ & $34.50$ \\
    $\SumOver{x}\AsymptoticVariance[2]{x}$
      & $9.83$ & $10.52$ & $7.71$ & $8.02$ & $16.72$ & $9.49$
      & $9.67$ & $8.44$ & $11.17$ & $12.86$ & $26.46$ & $33.83$ \\
    \midrule
    & 12{:}00 & 13{:}00 & 14{:}00 & 15{:}00 & 16{:}00 & 17{:}00
    & 18{:}00 & 19{:}00 & 20{:}00 & 21{:}00 & 22{:}00 & 23{:}00 \\
    \midrule
    $\SumOver{x}\AsymptoticVariance[0]{x}$
      & $11.93$ & $13.36$ & $16.67$ & $22.57$ & $13.23$ & $11.33$
      & $11.80$ & $12.04$ & $11.85$ & $10.85$ & $10.93$ & $11.51$ \\
    $\SumOver{x}\AsymptoticVariance[1]{x}$
      & $10.54$ & $11.97$ & $15.28$ & $21.18$ & $11.84$ & $9.94$
      & $10.41$ & $10.65$ & $10.46$ & $9.46$ & $9.53$ & $10.11$ \\
    $\SumOver{x}\AsymptoticVariance[2]{x}$
      & $9.89$ & $11.31$ & $14.62$ & $20.52$ & $11.18$ & $9.28$
      & $9.75$ & $9.99$ & $9.80$ & $8.79$ & $8.85$ & $9.42$ \\
    \bottomrule
  \end{tabular}
\end{table}

Figure~\ref{fig:surfer-realdata-estim-error} shows first that the empirical squared errors are close to the target variance of $10^{-5}$, with variations most likely due to the finite-horizon approximation of the asymptotic MSE: what is observed here is the $\LittleO(n^{-1})$ of Lemma~\ref{lem:summability}, which itself appears to be related to the spectral gap of the network.
It also shows that the numbers of observations required to achieve that target variance are highly variable.
Furthermore, in cases where a large theoretical sample size is required to reach the $10^{-5}$ threshold, the empirical variance achieved in practice turns out to be lower than the target value.
Finally, the model at hand is reversible, so that it is no surprise to observe a decrease of the error when passing from $\Empirical$ to $\PiHat{1}$ and then to $\PiHat{2}$.
Figure~\ref{fig:surfer-realdata-map} and Table~\ref{tab:heaviest-places} present the results obtained on the transport network at 12:00.

\begin{figure}[htbp]
  \centering
  \includegraphics[width=\PlotWidth]{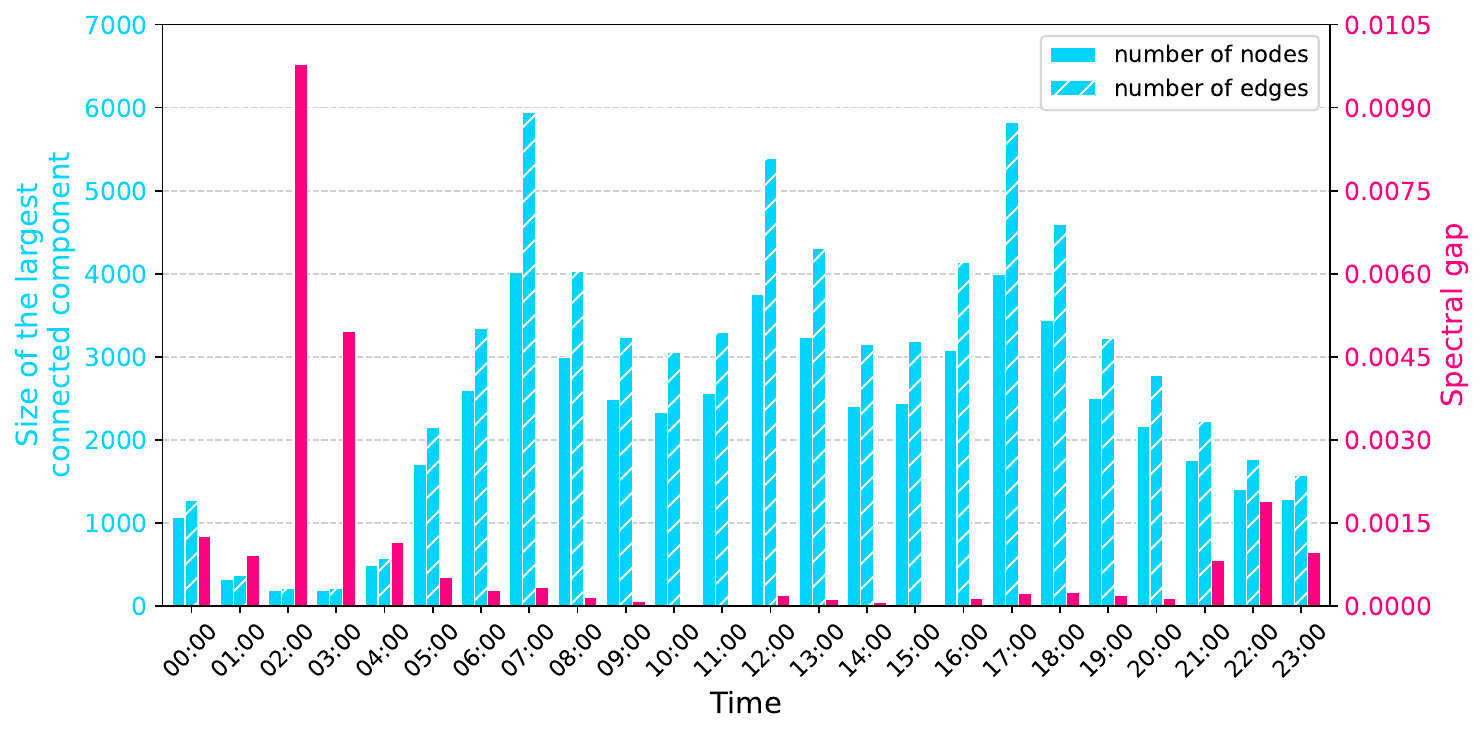}
  \caption{Features of the transport network.
    For each hour of the day, the size of the largest connected component, in number of nodes and in number of edges (left axis), and the spectral gap of the associated random surfer model (right axis).}
  \label{fig:surfer-realdata-features}
\end{figure}

\begin{figure}[htbp]
  \centering
  \includegraphics[width=\PlotWidth]{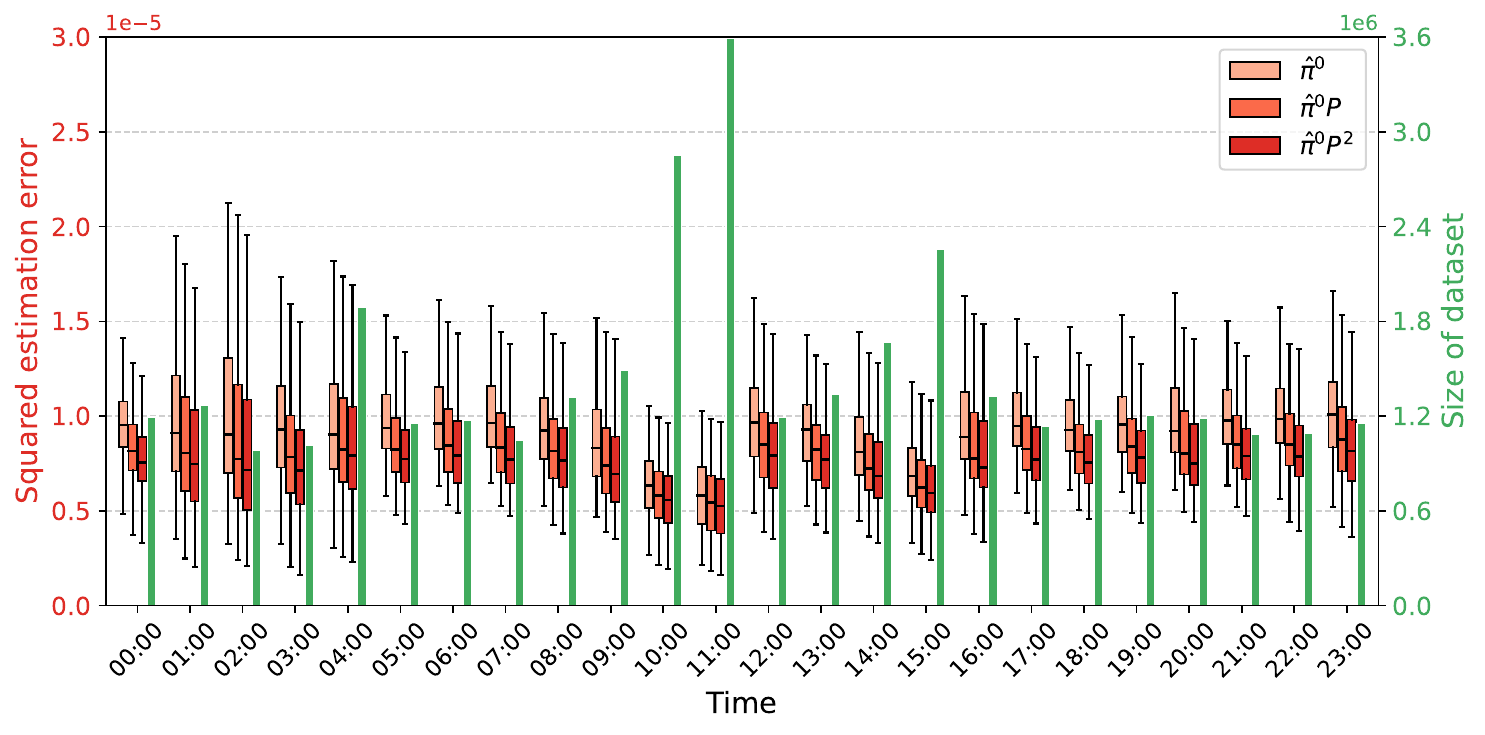}
  \caption{Boxplots over $200$ replicates of the squared estimation error of $\Empirical$, $\PiHat{1} = \Empirical \KernelMatrix$ and $\PiHat{2} = \Empirical \KernelMatrix[2]$ for each hour of the day; the size of the data set is reported on the right axis.}
  \label{fig:surfer-realdata-estim-error}
\end{figure}

\begin{figure}[htbp]
  \centering
  \includegraphics[trim=122 0 122 0, clip, width=\MapWidth]%
    {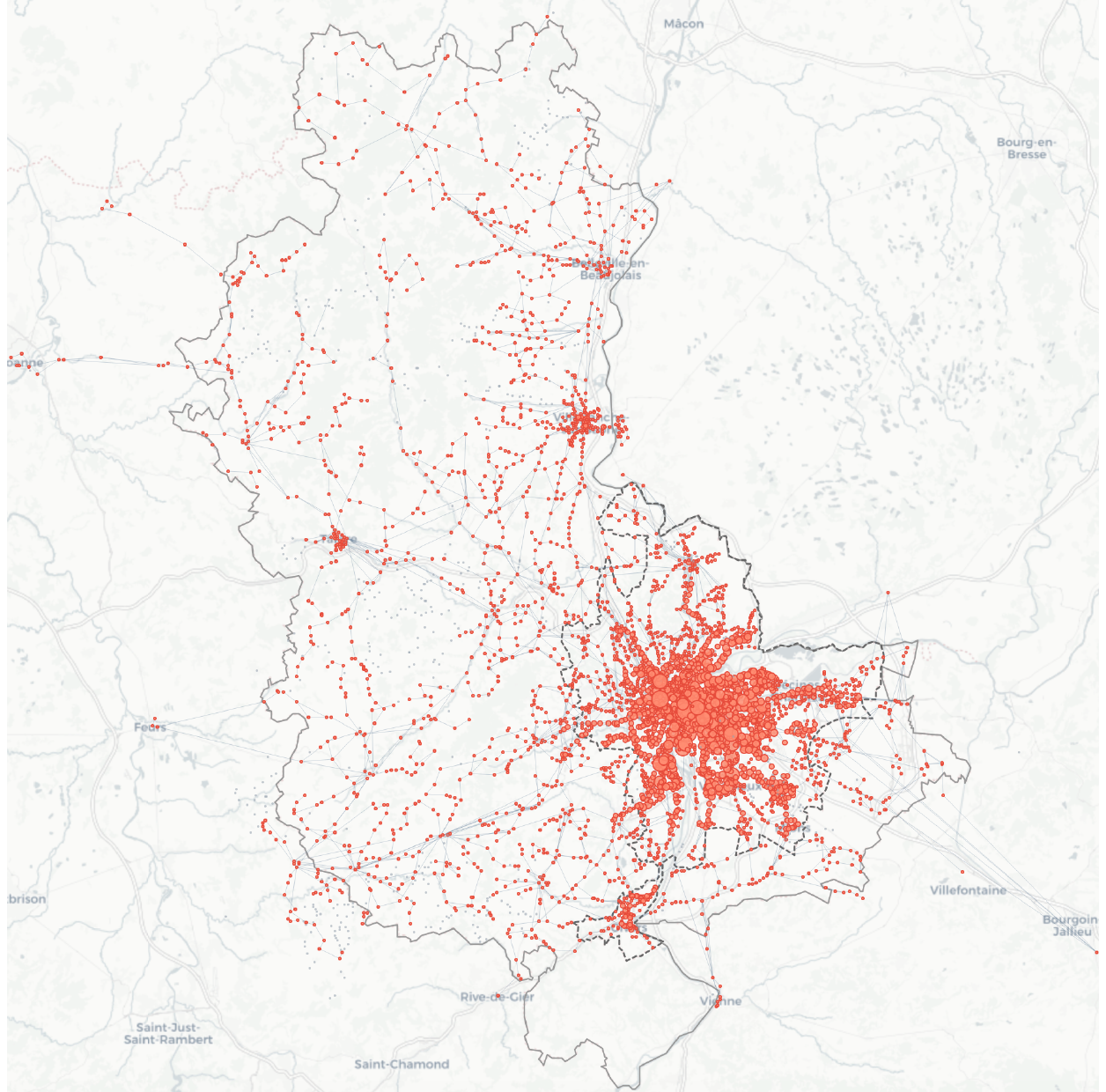}
  \caption{Theoretical transport network at 12:00; the size of a node is determined by its importance, itself defined from the invariant distribution of the random surfer model.}
  \label{fig:surfer-realdata-map}
\end{figure}

\begin{table}[htbp]
  \centering
  \caption{The ten most important places of the transport network at 12:00, that is, those carrying the largest invariant probability.
    Probabilities are given in units of $10^{-3}$.}
  \label{tab:heaviest-places}
  \small
  \setlength{\tabcolsep}{4pt}
  \begin{tabular}{rlr@{\hspace{1.5em}}rlr}
    \toprule
    \multicolumn{6}{c}{\emph{Ten most important places of the transport network at 12:00}} \\
    \midrule
    Rank & Place & Probability & Rank & Place & Probability \\
    \cmidrule(r){1-3} \cmidrule(l){4-6}
    1 & Saxe--Gambetta           & $10.05$ &  6 & Bellecour                     & $6.07$ \\
    2 & Gare Part-Dieu V.\,Merle & $7.49$  &  7 & Vieux Lyon Cat.\ St-Jean      & $5.84$ \\
    3 & Charpennes Charles Hernu & $7.35$  &  8 & Guilloti\`ere Gabriel P\'eri  & $5.73$ \\
    4 & Grange Blanche           & $6.60$  &  9 & Valmy                         & $5.65$ \\
    5 & Jean Mac\'e              & $6.52$  & 10 & Brotteaux                     & $5.58$ \\
    \bottomrule
  \end{tabular}
\end{table}

\newpage

\part{Parametrised transition kernel: a statistical approach}
\label{part:parametric-kernel}

In this part, the transition kernel is no longer assumed to be known, but to belong to a parametric family, which places us de facto in a statistical framework.

\section{Mean squared error of plug-in look-ahead estimators}
\label{sec:theory-parametric-kernel}

This section is dedicated to the theoretical study of the asymptotic MSE of plug-in look-ahead methods from stationary Markovian data.

\subsection{Plug-in look-ahead estimation of the invariant distribution}
\label{subsec:parametric-framework}

We consider a family $\{\KernelMatrixParam{\theta}\}_{\theta \in \ParamSpace}$, with $\ParamSpace \subset \Reals^{d}$, of transition kernels on a countable state space $\StateSpace$, and we observe a data set $(X_i)_{\IndexBetween{i}{1}{n}}$ with Markovian dependence of kernel $\KernelMatrixTrue$ for some $\ParamTrue\in\ParamSpace$.
Throughout this part, $\KernelMatrixTrue$ is assumed to be irreducible, aperiodic, positive recurrent and uniformly ergodic.
Consequently, $\KernelMatrixTrue$ admits a unique invariant distribution, denoted by $\pi$, and we assume that $\pi$ is the invariant distribution of no other transition kernel in the family, that is, $\pi \KernelMatrixParam{\theta} \neq \pi$ as soon as $\theta \in \ParamSpace \setminus \{\ParamTrue\}$.
Finally, the observed data are assumed to be stationary, that is, $X_1$ is distributed according to $\pi$.

Our objective is to estimate the invariant distribution $\pi$, taking into account our knowledge that the associated transition kernel belongs to the family $\{\KernelMatrixParam{\theta}\}_{\theta \in \ParamSpace}$.
Recall that we exclude here the approaches that consist in solving (exactly or approximately) the invariance equation for $\KernelMatrixParam{\ParamEstimator[n]}$, where $\ParamEstimator[n]$ is an estimator of $\ParamTrue$: the continuity of $\theta \mapsto \KernelMatrixParam{\theta}$ does not guarantee that its ergodic properties remain true around $\ParamTrue$, as shown by Proposition~\ref{prop:no-continuity}.

In analogy with the definition~\eqref{eq:estimator-pointwise} of look-ahead estimators with known kernel, it is natural to consider, for some $k \geq 1$ (assumed to be fixed in this section), for every $\theta \in \ParamSpace$ and every $x \in \StateSpace$,
\begin{equation}\label{eq:def:plugin:lookahead}
  \PiHatParamAt{\theta}{k}{x}
  = \frac{1}{n} \SumFromTo{i}{1}{n} \ContributionParam{\theta}{k}{i}{x},
\end{equation}
where
\[
  \forall \IndexBetween{i}{1}{n},
  \quad
  \ContributionParam{\theta}{k}{i}{x} = \KernelParam[k]{\theta}{x}{X_i}.
\]
The difference with the look-ahead estimators investigated in Sections~\ref{sec:theory-known-kernel} and~\ref{sec:random-surfer} is that the transition kernel that is applied does not in general (unless $\theta = \ParamTrue$) admit $\pi$ as its invariant distribution.
Since the parameter $\ParamTrue$ is unknown, it is natural to replace it in \eqref{eq:def:plugin:lookahead} by an estimator $\ParamEstimator[n]$, which defines the plug-in look-ahead estimator.
It should be remarked that \eqref{eq:def:plugin:lookahead} is equivalent to $\PiHatParam{\ParamEstimator[n]}{k} = \Empirical \KernelMatrixParam[k]{\ParamEstimator[n]}$.

To facilitate a rigorous theoretical analysis of the MSE, we consider in this paper a sample-split version of that estimator:
\begin{itemize}
  \item The first part of the data set, $(X_i)_{\IndexBetween{i}{1}{m}}$, is used to estimate $\ParamTrue$.
    The resulting estimator is denoted by $\ParamEstimator$; it may for instance be the maximum likelihood estimator.

  \item The second one, $(X_i)_{\IndexBetween{i}{m+1}{n}}$, is used to estimate $\pi$ by its empirical counterpart.
    The resulting estimator is denoted by $\Empirical[n-m]$.\footnote{Many references would rather write $\PiHat[m+1:n]{0}$, a notation free of any ambiguity, but a heavier one.}
\end{itemize}
The plug-in look-ahead estimator of $\pi(x)$ thus becomes
\[
  \PiHatParamAt[n-m]{\ParamEstimator}{k}{x}
  = \frac{1}{n-m} \SumFromTo{i}{m+1}{n}
    \ContributionParam{\ParamEstimator}{k}{i}{x}.
\]
In the terminology introduced in Remark~\ref{rem:local-versus-global}, this expression is local.
Equivalently, in global form, we can write
\begin{equation}\label{eq:def:plugin:lookahead:samplesplit}
  \PiHatParam[n-m]{\ParamEstimator}{k}
  = \Empirical[n-m] \KernelMatrixParam[k]{\ParamEstimator}.
\end{equation}
The objective of this section is to study theoretically the asymptotic MSE of $\PiHatParam[n-m]{\ParamEstimator}{k}$.
In particular, we aim to identify settings general enough for a reduction of the asymptotic MSE to be observed, compared with the empirical estimator $\Empirical$ (evaluated from the whole data set).
To this end, a sample-split sequence $m=m(n)$ is henceforth assumed to satisfy $m \leq n$, with both $m(n)$ and $n-m(n)$ increasing monotonically to infinity as $n$ goes to infinity.

\subsection{Asymptotic expansion of the mean squared error}
\label{subsec:parametric-mse-expansion}

Two factors may contribute to the MSE of $\PiHatParam[n-m]{\ParamEstimator}{k}$ given by \eqref{eq:def:plugin:lookahead:samplesplit}: an error due to $\Empirical[n-m]$, and an error due to $\ParamEstimator$.
Our aim here is to evaluate the contribution of each of them, that is, to obtain an expression of the MSE that separates them.
To that end, we are led to make regularity assumptions, in $\theta$, on the model.

The first condition bears on $\KernelMatrixParam[k]{\theta}$ directly: it is a uniform Lipschitz assumption, which involves the following metric (which will be motivated in Remark~\ref{rem:residual-norm}).
For a matrix $Q$ indexed on $\StateSpace$, we define
\[
  \NormOperator{Q}_{2,\text{res}}
  = \sup_{\substack{\mu, \nu \in \mathcal{P}(\StateSpace) \\ \mu \neq \nu}}
    \frac{\NormTwo{\left(\mu-\nu\right)Q}}{\NormTwo{\mu-\nu}},
\]
where $\mathcal{P}(\StateSpace)$ denotes the set of all probability distributions on $\StateSpace$.
In contrast to the standard matrix norm $\NormOperator{\cdot}_{2}$, this quantity does not define a norm on the space of matrices over $\StateSpace$.
Instead, it measures the maximum divergence that can be induced between two probability distributions under the action of a matrix $Q$.

\begin{condition}
\label{ass:regularity1}
The mapping $\theta \mapsto \NormOperator{\KernelMatrixParam[k]{\theta}}_{2,\text{res}}$ is uniformly bounded over $\ParamSpace$, and there exists a constant $[\KernelMatrix[k]]_{\text{Lip}} > 0$ such that,
\[
  \forall \theta \in \ParamSpace, \quad
  \NormOperator{\KernelMatrixParam[k]{\theta} - \KernelMatrixTrue[k]}_{2,\text{res}}
  \leq [\KernelMatrix[k]]_{\text{Lip}}
       \NormTwo{\theta - \ParamTrue}.
\]
\end{condition}

\begin{remark}
\label{rem:residual-norm}
The use of the metric $\NormOperator{\cdot}_{2,\text{res}}$ in Condition~\ref{ass:regularity1} may appear technical or ad hoc.
When the state space $\StateSpace$ is finite, one could adopt the stronger but more natural operator norm $\NormOperator{\cdot}_{2}$ instead of $\NormOperator{\cdot}_{2,\text{res}}$, sacrificing some generality for a more standard hypothesis.
However, when $\StateSpace$ is countably infinite, replacing $\NormOperator{\cdot}_{2,\text{res}}$ with $\NormOperator{\cdot}_{2}$ conflicts with the uniform ergodicity of the matrix $\KernelMatrixTrue$.
Indeed, uniformly ergodic transition kernels on a countable state space have an infinite $\NormOperator{\cdot}_{2}$ norm, which would result in an empty global set of assumptions.
\end{remark}

We then turn to the regularity of $\theta \mapsto \pi \KernelMatrixParam[k]{\theta}$.
Without much loss of generality, we first assume that $\theta \mapsto \KernelParam[k]{\theta}{x}{y}$ is differentiable at $\ParamTrue$ for every $x,y \in \StateSpace$, and we denote by $\Gradient \KernelParam[k]{\theta}{x}{y}$ its gradient, for every $\theta$ in a neighbourhood of $\ParamTrue$.
Adding the summability condition,
\[
  \forall x \in \StateSpace, \quad
  \SumOver{y} \pi(y) \NormTwo{\Gradient \KernelParam[k]{\ParamTrue}{x}{y}}
  < +\infty,
\]
we obtain, for every $x \in \StateSpace$, the differentiability of $\theta \mapsto (\pi \KernelMatrixParam[k]{\theta})(x)$ around $\ParamTrue$, with
\[
  \Gradient \left(\pi \KernelMatrixParam[k]{\theta}\right)\!\left(x\right)
  = \SumOver{y} \pi(y) \Gradient \KernelParam[k]{\theta}{x}{y}.
\]
We shall write $\Gradient(\pi \KernelMatrixParam[k]{\theta})$ for the matrix whose general term is,
\[
  \forall \IndexBetween{i}{1}{d},\ \forall x \in \StateSpace, \quad
  \left(\Gradient \left(\pi \KernelMatrixParam[k]{\theta}\right)\!\left(x\right)\right)_{i}
  = \partial_{\theta_i}\left(\pi\KernelMatrixParam[k]{\theta}\right)\!\left(x\right).
\]

\begin{condition}
\label{ass:regularity2}
For every $\theta \in \ParamSpace$,
\[
  \NormTwo{\pi\left(\KernelMatrixParam[k]{\theta} - \KernelMatrixTrue[k]\right)}^{2}
  = \Trace\left(\Gradient\left(\pi \KernelMatrixTrue[k]\right)\Transpose
    \left(\theta-\ParamTrue\right) \left(\theta-\ParamTrue\right)\Transpose
    \Gradient\left(\pi \KernelMatrixTrue[k]\right)\right)
  + \BigO\!\left(\NormTwo{\theta-\ParamTrue}^{3}\right).
\]
\end{condition}

While these conditions are essential for the analysis that follows, they are not really restrictive.
In fact, they hold naturally when the state space $\StateSpace$ is finite, as highlighted by the following proposition.

\begin{proposition}
\label{prop:regularity-finite}
If $\StateSpace$ is finite and $\ParamSpace$ is convex and compact, then Conditions~\ref{ass:regularity1} and~\ref{ass:regularity2} are satisfied provided that $\theta \mapsto \KernelMatrixParam{\theta}$ is of class $\mathcal{C}^{2}$.
\end{proposition}

\ProofDeferred{app:proof-regularity-finite}

\begin{remark}
\label{rem:relaxed-parameter-space}
If the assumptions on the parameter space are relaxed to require only the convexity of $\ParamSpace$, one must additionally assume that the gradient and Hessian of the transition kernel, namely,
\[
  \forall x,y \in \StateSpace, \quad
  \theta \mapsto \Gradient \KernelParam[k]{\theta}{x}{y}
  \TextAnd
  \theta \mapsto \Hessian \KernelParam[k]{\theta}{x}{y},
\]
are uniformly bounded on $\ParamSpace$.
In practice, the parameter space $\ParamSpace$ is commonly chosen to be a hypercube, making it both convex and compact.
Consequently, in a finite state space setting, this geometric property ensures that the required regularity conditions reduce to the standard assumption that $\theta \mapsto \KernelMatrixParam{\theta}$ is of class $\mathcal{C}^{2}$.
\end{remark}

With this setup established, we are now ready to state our main theoretical result regarding the expansion of the MSE.

\begin{theorem}
\label{thm:mse-expansion}
Under Conditions~\ref{ass:regularity1} and~\ref{ass:regularity2},
\[
\begin{aligned}
  &\MSE\left(\PiHatParam[n-m]{\ParamEstimator}{k}\right)
  = \MSE\left(\PiHatParam[n-m]{\ParamTrue}{k}\right) \\
  &\enspace + \Trace\left(\Gradient\left(\pi \KernelMatrixTrue[k]\right)\Transpose
     \Expect\left[\left(\EstimationError\right)\left(\EstimationError\right)\Transpose\right]
     \Gradient\left(\pi \KernelMatrixTrue[k]\right)\right) \\
  &\enspace + \BigO\!\left(\Expect\left[\NormTwo{\EstimationError}^{3}\right]\right)
   + \BigO\!\left(\frac{\Expect\left[\NormTwo{\EstimationError}^{2}\right]^{1/2}}{n-m}\right)
   + \BigO\!\left(\frac{\Expect\left[\NormTwo{\EstimationError}^{2}\right]}{n-m}\right).
\end{aligned}
\]
\end{theorem}

\ProofDeferred{app:proof-mse-expansion}

\subsection{Reduction of the asymptotic mean squared error}
\label{subsec:parametric-data-partition}

By Theorem~\ref{thm:mse-expansion}, the MSE of $\PiHatParam[n-m]{\ParamEstimator}{k}$ is controlled by (i) the MSE of the estimator $\PiHatParam[n-m]{\ParamTrue}{k}$ (which assumes the transition kernel to be known and was studied theoretically in Section~\ref{sec:theory-known-kernel}), and by (ii) the error made in estimating $\ParamTrue$.
To account for this second term without restricting the choice of the estimation procedure, we simply require $\ParamEstimator$ to satisfy the following general conditions.

First, we assume that
\[
  \Expect\left[\NormTwo{\EstimationError}^{3}\right]
  = \BigO\!\left(m^{-3/2}\right).
\]
For instance, this holds when $\ParamEstimator$ is a smooth function of empirical averages of bounded functions of the chain, since such averages satisfy a Hoeffding-type inequality as soon as the chain is uniformly ergodic~\cite{GlynnOrmoneit2002,Paulin2015}.
In addition, $g(m)$ denotes the rate of the quadratic error of $\ParamEstimator$, i.e.\ we assume that there exists a positive semidefinite matrix $\CovParam$ (of dimension $d$) such that, as $m$ goes to infinity,
\[
  \Expect\left[\left(\EstimationError\right)\left(\EstimationError\right)\Transpose\right]
  \EquivalentTo g(m) \CovParam.
\]
We further assume that
\[
  \Trace\left(\Gradient\left(\pi \KernelMatrixTrue[k]\right)\Transpose
  \CovParam \Gradient\left(\pi \KernelMatrixTrue[k]\right)\right) \neq 0 .
\]
This last assumption induces a link with the invariance equation, but remains very mild.
Indeed, it amounts to assuming that the second term in the MSE expansion of $\PiHatParam[n-m]{\ParamEstimator}{k}$ given in Theorem~\ref{thm:mse-expansion} does not vanish asymptotically.
By construction, were that term to vanish, this could only improve the performance of the plug-in look-ahead estimator.

\begin{remark}\label{rem:mse:zero}
In the sequel, we exclude the degenerate case where the asymptotic MSE of the look-ahead estimator $\PiHatParam[n-m]{\ParamTrue}{k}$, which appears in the expansion of Theorem~\ref{thm:mse-expansion}, vanishes, that is, $\SumOver{x} \AsymptoticVariance[k]{x} = 0$.
Interestingly, this occurs if and only if every row of the kernel $\KernelMatrixTrue[k]$ coincides with the invariant distribution, i.e.\ $\KernelParam[k]{\ParamTrue}{\cdot}{y} = \pi$ for every $y \in \StateSpace$.
In that degenerate case, the whole error comes from the estimation of $\ParamTrue$, and there is no point in allocating any data to the empirical estimation of $\pi$.
\end{remark}

Under all these assumptions, the following results compare the asymptotic MSE of the estimator of interest, $\PiHatParam[n-m]{\ParamEstimator}{k}$, with the asymptotic MSE of $\Empirical$, according to the rate $g(m)$.
Three regimes are considered: the case where $g(m)$ is slower than the parametric rate, in Corollary~\ref{cor:data-partition1}; the case where $g(m)$ is faster than the parametric rate, in Corollary~\ref{cor:data-partition2}; and the most interesting case, where $g(m)$ is the parametric rate, in Corollary~\ref{cor:data-partition3}.

\begin{corollary}
\label{cor:data-partition1}
Under Conditions~\ref{ass:regularity1} and~\ref{ass:regularity2}, if $\ParamEstimator$ satisfies the general assumptions above with $m^{-1} = \LittleO(g(m))$, then, for any split,
\[
  \AsymptoticMSE\left(\PiHatParam[n-m]{\ParamEstimator}{k}\right) = +\infty ,
\]
while $\AsymptoticMSE(\Empirical)$ is finite.
\end{corollary}

\ProofDeferred{app:proof-data-partition1}

In other words, when $\ParamEstimator$ converges at a slower rate than $\Empirical$, the plug-in look-ahead estimator should not be considered.

\begin{corollary}
\label{cor:data-partition2}
Under Conditions~\ref{ass:regularity1} and~\ref{ass:regularity2}, if $\ParamEstimator$ satisfies the general assumptions above with $g(m) = \LittleO(m^{-1})$, then, for any split such that $m = \LittleO(n)$, $n = \LittleO(g(m)^{-1})$ and $n = \LittleO(m^{3/2})$,
\[
  \AsymptoticMSE\left(\PiHatParam[n-m]{\ParamEstimator}{k}\right)
  = \AsymptoticMSE\left(\PiHatParam{\ParamTrue}{k}\right).
\]
\end{corollary}

\ProofDeferred{app:proof-data-partition2}

When $\ParamEstimator$ converges at a faster rate than $\Empirical$, allocating a number of observations negligible with respect to $n$ to the estimation of $\ParamTrue$ is enough to reach the performance of the look-ahead estimator with known kernel.
The comparison with $\AsymptoticMSE(\Empirical)$ is studied theoretically in Theorem~\ref{thm:monotone-amse}: for $k = 1$, the improvement in asymptotic MSE always holds; for $k > 1$, reversibility guarantees it.

The following result explores the case where $\ParamEstimator$ and $\Empirical$ converge at the same rate through the variance ratios
\[
\begin{aligned}
  \RatioLookAhead
  &= \frac{\SumOver{x} \AsymptoticVariance[k]{x}}
          {\SumOver{x} \AsymptoticVariance[0]{x}}, \\[0.8em]
  \RatioParameter
  &= \frac{\Trace\left(\Gradient\left(\pi \KernelMatrixTrue[k]\right)\Transpose
          \CovParam \Gradient\left(\pi \KernelMatrixTrue[k]\right)\right)}
          {\SumOver{x} \AsymptoticVariance[0]{x}}.
\end{aligned}
\]

\begin{corollary}
\label{cor:data-partition3}
Under Conditions~\ref{ass:regularity1} and~\ref{ass:regularity2}, if $\ParamEstimator$ satisfies the general assumptions above with $g(m) \EquivalentTo m^{-1}$,
\begin{equation}\label{eq:def:optimalsplit}
  \SplitFractionOptimal
= \operatorname*{argmin}_{\SplitFraction \in \left]0,1\right[}
     \AsymptoticMSE\left(\PiHatParam[n-m]{\ParamEstimator}{k}\right)
= \frac{\sqrt{\RatioParameter}}
          {\sqrt{\RatioLookAhead}+\sqrt{\RatioParameter}}.
\end{equation}
Moreover, for any split of the form $m \EquivalentTo \SplitFractionOptimal n$,
\[
  \AsymptoticMSE\left(\PiHatParam[n-m]{\ParamEstimator}{k}\right)
  < \AsymptoticMSE\left(\Empirical\right),
\]
if and only if $\sqrt{\RatioLookAhead}+\sqrt{\RatioParameter} < 1$.
\end{corollary}

\ProofDeferred{app:proof-data-partition3}

In the most interesting case where $\ParamEstimator$ and $\Empirical$ converge at the same rate, the comparison of the asymptotic MSEs rests on a delicate evaluation of the asymptotic MSE of $\Empirical$, of $\PiHat{k}$ and of $\ParamEstimator[m]$.

\begin{remark}
\label{rem:additional-knowledge-corollary}
From a practical viewpoint, one does not need to use the exact optimal split ratio $\SplitFractionOptimal$ to guarantee that $\PiHatParam[n-m]{\ParamEstimator}{k}$ improves upon $\Empirical$.
Indeed, any split of the form $m \EquivalentTo \SplitFraction n$ will yield a strictly smaller asymptotic MSE, provided that
\[
  \SplitFraction \in \left]\psi^{-}(\RatioLookAhead,\RatioParameter),
  \psi^{+}(\RatioLookAhead,\RatioParameter)\right[,
\]
where the expressions for the boundaries $\psi^{\pm}(\RatioLookAhead,\RatioParameter)$ are provided by \eqref{eq:def:interval:alpha} in Appendix~\ref{app:proof-data-partition3}.
Conversely, any split sequence where $m$ and $n-m$ are not proportional asymptotically (i.e.\ $m = \LittleO(n)$ or $n-m = \LittleO(n)$) will lead to an estimator converging at a slower rate than the standard parametric rate, making such splitting strategies irrelevant.
\end{remark}

\section{Plug-in look-ahead estimation in a birth-and-death process}
\label{sec:birth-death}

In this section we study, both theoretically and numerically, the behaviour of the (sample-split) one-step plug-in look-ahead estimator ($k=1$) on a birth-and-death process.
Numerically only, we further consider the estimator $\PiHatParam{\ParamEstimator[n]}{1} = \Empirical \KernelMatrixParam{\ParamEstimator[n]}$, in which the whole trajectory is used to build both $\Empirical$ and $\ParamEstimator[n]$, without any splitting of the data set.
Finally, the method is applied to a real data set recording the availability of bikes in the stations of a bike-sharing system.

\subsection{Theoretical investigations}
\label{subsec:birth-death-theory}

The model of interest is a birth-and-death process on $\IntegerRange{0}{L}$, $L \geq 2$, defined from two parameters $\theta, \varepsilon \in \, ]0,1[$ by its transition matrix $\BDKernel$,
\[
\begin{alignedat}{2}
  \KernelParam{\theta,\varepsilon}{0}{0}
    &= \KernelParam{\theta,\varepsilon}{L}{L}   &&= \varepsilon, \\
  \KernelParam{\theta,\varepsilon}{1}{0}
    &= \KernelParam{\theta,\varepsilon}{L-1}{L} &&= 1-\varepsilon,
\end{alignedat}
\]
and, for $\IndexBetween{x}{1}{L-1}$,
\[
\begin{aligned}
  \KernelParam{\theta,\varepsilon}{x-1}{x} &= (1-\varepsilon)(1-\theta), \\
  \KernelParam{\theta,\varepsilon}{x}{x} &= \varepsilon, \\
  \KernelParam{\theta,\varepsilon}{x+1}{x} &= (1-\varepsilon)\theta.
\end{aligned}
\]
All other entries of $\BDKernel$ are zero.
The parameter $\varepsilon$ controls how lazy the chain is, and makes it aperiodic; it is assumed to be known in what follows.
The parameter $\theta$ controls whether the model tends to favour births or deaths.
The following proposition describes the elementary properties of the model.

\begin{proposition}
\label{prop:birth-death-law}
For $\theta, \varepsilon \in \, ]0,1[$, $\BDKernel$ is irreducible, aperiodic and reversible.
Moreover, if $\theta < 0.5$, the invariant distribution $\BDpi$ is given by
\[
\begin{aligned}
  \BDpi(0) &= \frac{1-\rho}{2\left(1-\rho^{L}\right)}, \\
  \forall \IndexBetween{x}{1}{L-1}, \quad \BDpi(x)
    &= \BDpi(0)(1+\rho)\rho^{x-1}, \\
  \BDpi(L) &= \BDpi(0)\rho^{L-1},
\end{aligned}
\]
where $\rho = \theta/(1-\theta)$.
\end{proposition}

\ProofDeferred{app:proof-birth-death-law}

For the sake of readability, we drop from now on the indexing by $\theta$ and $\varepsilon$ of the transition kernel and of its invariant distribution.
Since its state space is finite, the model satisfies the main assumptions made in the general theoretical study of Section~\ref{sec:theory-parametric-kernel}.
The main objective of this subsection is to apply Corollary~\ref{cor:data-partition3} to the one-step plug-in look-ahead estimator ($k=1$) for this process.
To that end, we begin by studying the asymptotic variance of the empirical estimator.

\begin{lemma}
\label{prop:birth-death-variance}
For $f \in \Ltwopi$, the asymptotic MSE of the empirical estimator $\EmpiricalOf{f}$ of $\pi(f)$ is given by
\[
  \AsymptoticVariance[0]{f}
  = 2 \SumFromTo{j}{0}{L-1} \frac{1}{\pi(j)\Kernel{j+1}{j}}
    \left(\SumFromTo{i}{0}{j} \pi(i)\bar{f}(i)\right)^{2}
  - \pi\left( {\bar{f}}^2 \right),
\]
where $\bar{f} = f - \pi(f)$.
\end{lemma}

\ProofDeferred{app:proof-birth-death-variance}

Applying it to the elementary functions $\Dirac{x}$, $x \in \StateSpace$, yields the first quantity of interest, namely the asymptotic MSE of the empirical estimator.

\begin{proposition}
\label{prop:birth-death-MSE-pi0}
For $\theta < 0.5$ and $\varepsilon \in \, ]0,1[$, when $L$ goes to infinity,
\[
  \SumFromTo{x}{0}{L} \AsymptoticVariance[0]{x}
  = \frac{2}{1-\varepsilon} \times \frac{(1+\rho)^{2}}{4(1-\rho)}
  - \frac{1+\rho}{2} + \LittleO_{L}(1).
\]
\end{proposition}

\ProofDeferred{app:proof-birth-death-MSE-pi0}

We now compute the asymptotic variance of the one-step look-ahead estimator (with known kernel).

\begin{proposition}
\label{prop:birth-death-MSE-pi1}
For $\theta < 0.5$ and $\varepsilon \in \, ]0,1[$, when $L$ goes to infinity,
\[
  \SumFromTo{x}{0}{L} \AsymptoticVariance[1]{x}
  = \frac{2}{1-\varepsilon} \times \frac{(1+\rho)^{2}}{4(1-\rho)}
  - \frac{1+\rho}{2} + 1 - \varepsilon^{2}
  - \frac{(1-\varepsilon)^{2}}{1+\rho} - 2(\varepsilon + \rho)
  + \LittleO_{L}(1).
\]
\end{proposition}

\ProofDeferred{app:proof-birth-death-MSE-pi1}

We shall turn to the estimation of the parameter $\theta$.
From Markovian data $(X_i)_{\IndexBetween{i}{1}{n}}$, we introduce the estimator
\[
  \ParamEstimator[n]
  = \frac{\#\left\{\IndexBetween{i}{1}{n-1} \,:\,
      X_i \notin \left\{0,L\right\},\ X_{i+1} = X_i+1\right\}}
         {\#\left\{\IndexBetween{i}{1}{n-1} \,:\,
      X_i \notin \left\{0,L\right\},\ X_{i+1} \neq X_i\right\}}.
\]
This estimator satisfies the following central limit theorem.

\begin{proposition}
\label{prop:birth-death-MSE-parameter}
For $\theta < 0.5$ and $\varepsilon \in \, ]0,1[$, when $n$ goes to infinity,
\[
  \sqrt{n}\left(\ParamEstimator[n]-\theta\right)
  \ConvergesInDistribution
  \Gaussian{0}{s^2},
\]
where, when $L$ goes to infinity,
\[ s^2 = \frac{2\rho}{(1-\varepsilon)(1+\rho)^{3}} + \LittleO_{L}(1).\]
\end{proposition}

\ProofDeferred{app:proof-birth-death-MSE-parameter}

It is worth noting that, by symmetry arguments, the formulas of Propositions~\ref{prop:birth-death-MSE-pi0}, \ref{prop:birth-death-MSE-pi1} and~\ref{prop:birth-death-MSE-parameter} also hold for $\theta > 0.5$, through the change of variable $\varphi(\theta) = 1-\theta$.
All this leads to the following result.

\begin{proposition}
\label{prop:birth-death-comparison}
There exists a threshold $\bar{\theta} \in \; ]0, 0.5[$, with $\bar{\theta} \simeq 0.289$, such that the following holds.
  \begin{itemize}
    \item If $\theta \in \; ]0, \bar{\theta}[$, then, for any split,
      \[
        \lim_{L \to +\infty} \lim_{\varepsilon \to 0}
        \AsymptoticMSE\left(\PiHatParam[n-m]{\ParamEstimator}{1}\right)
        > \lim_{L \to +\infty} \lim_{\varepsilon \to 0}
        \AsymptoticMSE\left(\Empirical\right).
      \]

    \item If $\theta \in \; ]\bar{\theta}, 0.5[$, there exists $\SplitFractionOptimal $ (seen as the limit in $\varepsilon$, then in $L$, of the optimal split ratio \eqref{eq:def:optimalsplit} given in Corollary~\ref{cor:data-partition3}) such that, if $m \EquivalentTo \SplitFractionOptimal n$,
      \[
        \lim_{L \to +\infty} \lim_{\varepsilon \to 0}
        \AsymptoticMSE\left(\PiHatParam[n-m]{\ParamEstimator}{1}\right)
        < \lim_{L \to +\infty} \lim_{\varepsilon \to 0}
        \AsymptoticMSE\left(\Empirical\right).
      \]
  \end{itemize}
\end{proposition}

\ProofDeferred{app:proof-birth-death-comparison}

When the state space is large enough and the chain is not too lazy, we have therefore identified precisely the region of the parameter space on which the plug-in look-ahead estimator outperforms the empirical estimator in terms of asymptotic MSE.

\subsection{Simulation study}
\label{subsec:birth-death-simulations}

We now illustrate Proposition~\ref{prop:birth-death-comparison} by numerical simulations of the model.
Throughout this subsection, we take $L = 40$ and $\varepsilon = 0.01$.
Figure~\ref{fig:simus-example} displays two trajectories of the process, for $\theta = 0.475$ and $\theta = 0.575$, together with the corresponding invariant distributions.
The two values lie on either side of $0.5$ and illustrate the role of the parameter: the chain drifts towards $0$ when deaths are favoured and towards $L$ when births are, the invariant distribution concentrating accordingly at the corresponding end of the state space.

\begin{figure}[htbp]
  \centering
  \includegraphics[width=\PlotWidth]{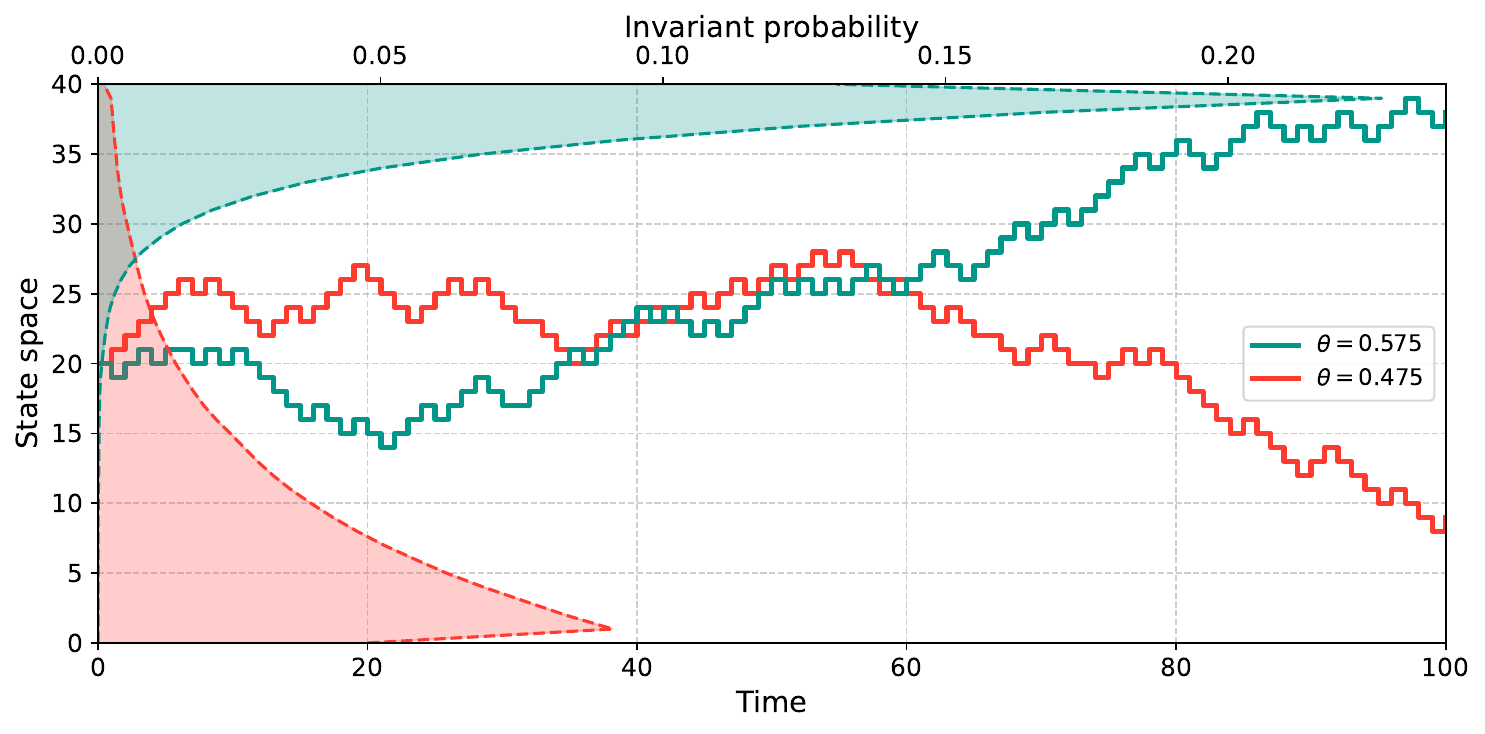}
  \caption{Two trajectories of the birth-and-death process, for $\theta = 0.475$ and $\theta = 0.575$, together with the corresponding invariant distributions (dashed lines, top axis).
    In both cases $L = 40$, $\varepsilon = 0.01$, and the chain is started at $L/2$.}
  \label{fig:simus-example}
\end{figure}

For the model with parameter $\theta = 0.4$, we compare numerically the performances of $\Empirical$, of $\PiHatParam[n-m]{\ParamEstimator}{1}$ (for various split ratios $m/n$) and of $\PiHatParam{\ParamEstimator[n]}{1}$ (evaluated from the whole data set) in Figure~\ref{fig:simus-splitratio}.
We first observe that the theoretical range of split ratios for which the asymptotic MSE of $\PiHatParam[n-m]{\ParamEstimator}{1}$ is smaller than the asymptotic MSE of $\Empirical$ (mentioned in Remark~\ref{rem:additional-knowledge-corollary} and taken for the limit model in $L$ and $\varepsilon$) matches the split ratios for which the actual MSE of the plug-in look-ahead estimator, evaluated at a finite horizon and from the non-asymptotic model, improves on that of the empirical estimator.
The best split ratio given by Corollary~\ref{cor:data-partition3}, whether oracle or estimated by plug-in, improves the distribution of the estimation error compared with the empirical estimator, with entirely comparable performance in both cases.
Finally, the plug-in look-ahead estimator evaluated from the whole data set improves the performance further still, compared with the estimator built on a split data set.

In Figure~\ref{fig:simus-theta}, we compare the performance of $\Empirical$, of $\PiHatParam[n-m]{\ParamEstimator}{1}$ (with the best estimated split ratio), and of $\PiHatParam{\ParamEstimator[n]}{1}$ (evaluated from the whole data set), as a function of the parameter $\theta$.
Although we work on the non-asymptotic model and consider a finite-horizon MSE, we recover the theoretical property established in Proposition~\ref{prop:birth-death-comparison}: the MSE ratio of the sample-split plug-in look-ahead estimator crosses $1$ at $\theta \simeq 0.29$, in close agreement with the theoretical threshold $\bar{\theta} \simeq 0.289$.
The picture is less clear-cut for the distribution of the error, since the boxplots overlap widely and the medians are almost indistinguishable: the gain is carried by the mean rather than by the bulk of the distribution.
Whatever the value of the parameter, the plug-in look-ahead estimator built from the whole data set improves on the performance of $\Empirical$.

\begin{figure}[htbp]
  \centering
  \includegraphics[width=\PlotWidth]{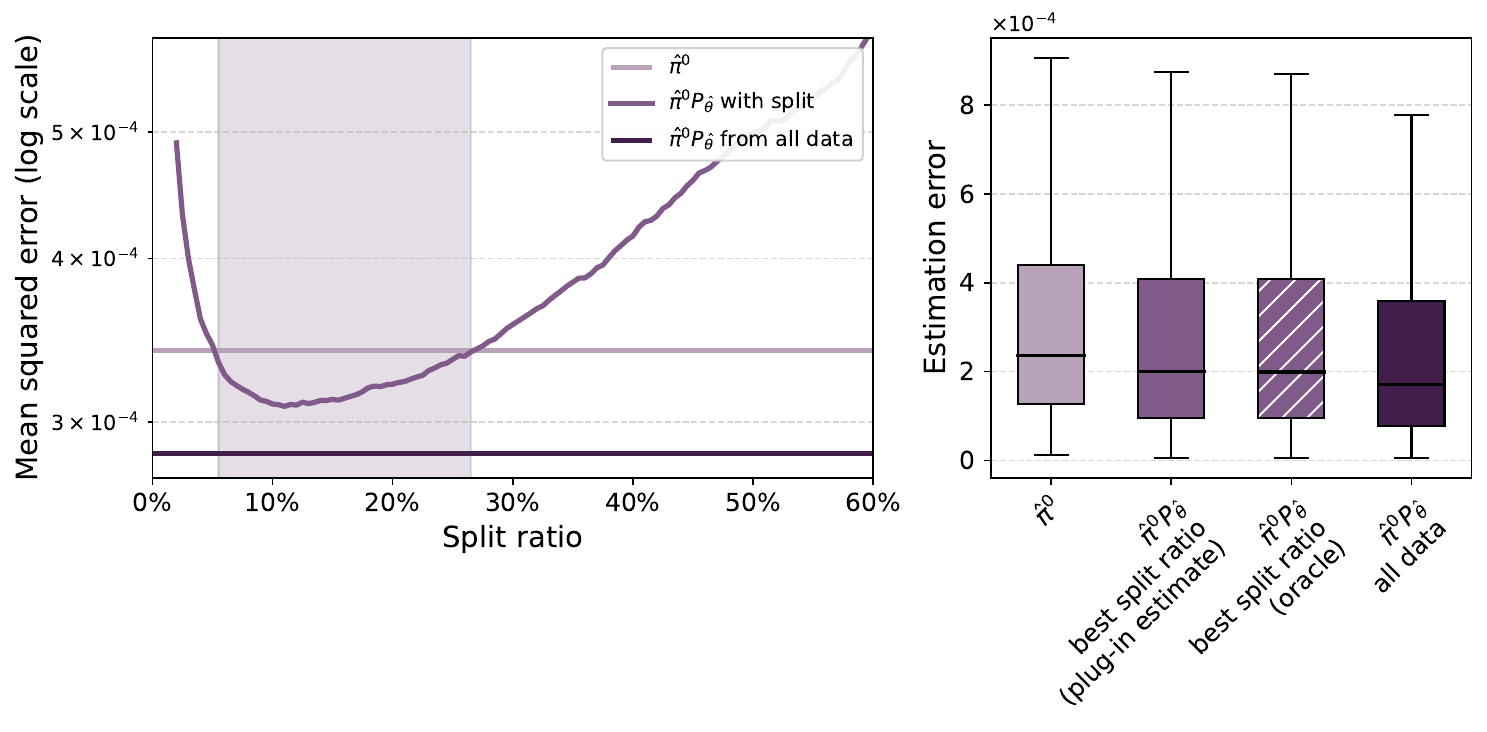}
  \caption{MSE of $\Empirical$, of $\PiHatParam[n-m]{\ParamEstimator}{1} = \Empirical[n-m] \KernelMatrixParam{\ParamEstimator}$ (for split ratios between $0$ and $0.6$) and of $\PiHatParam{\ParamEstimator[n]}{1} = \Empirical \KernelMatrixParam{\ParamEstimator[n]}$ (left), and boxplots of the squared estimation error for the same estimators, the optimal split ratio \eqref{eq:def:optimalsplit} given by Corollary~\ref{cor:data-partition3} being either theoretical (oracle) or estimated by plug-in (right).
    MSE values and boxplots are computed from $1000$ trajectories of horizon $10\,000$, for the model with parameter $\theta = 0.4$.
    On the left panel, the highlighted range of split ratios is the theoretical interval given in Remark~\ref{rem:additional-knowledge-corollary} for the limit model in $L$ and $\varepsilon$.}
  \label{fig:simus-splitratio}
\end{figure}

\begin{figure}[htbp]
  \centering
  \includegraphics[width=\PlotWidth]{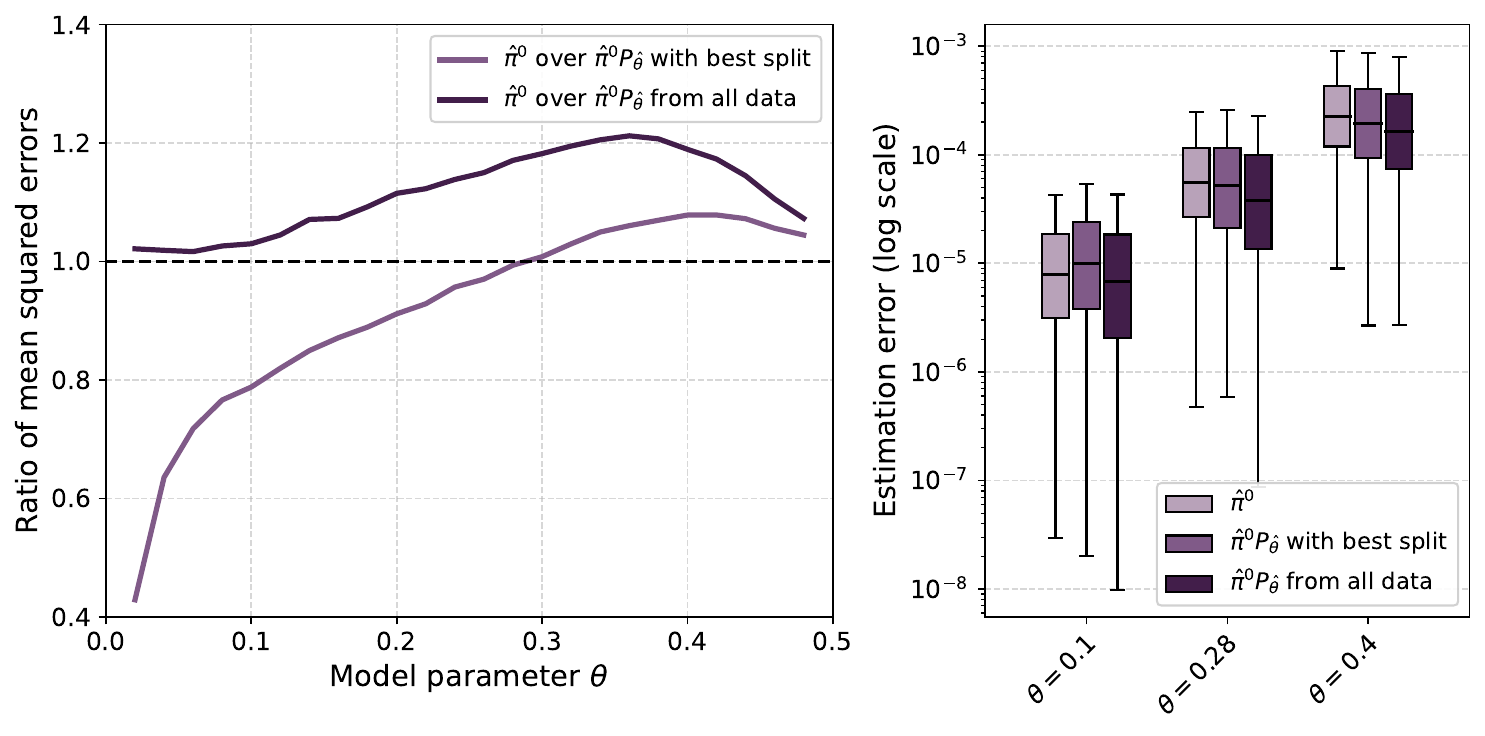}
  \caption{Comparison of $\Empirical$, of $\PiHatParam[n-m]{\ParamEstimator}{1} = \Empirical[n-m] \KernelMatrixParam{\ParamEstimator}$ (with the best estimated split ratio) and of $\PiHatParam{\ParamEstimator[n]}{1} = \Empirical \KernelMatrixParam{\ParamEstimator[n]}$, when $L = 40$ and $\varepsilon = 0.01$.
    Left: ratio of the MSE of $\Empirical$ to the MSE of each of the two plug-in look-ahead estimators, as a function of the parameter $\theta$; a value larger than $1$ indicates an improvement over $\Empirical$.
    Right: boxplots of the squared estimation error for three values of $\theta$.
    MSE values and boxplots are estimated from $10\,000$ trajectories of horizon $10\,000$ for each value of $\theta$.}
  \label{fig:simus-theta}
\end{figure}

The effect of the horizon, for the model with $\theta = 0.4 > \bar{\theta}$, is studied in Table~\ref{tab:simus-mse}.
For the sample-split estimator, with the best split ratio estimated by plug-in, the MSE of $\PiHatParam[n-m]{\ParamEstimator}{1}$ decreases with $n$ while remaining smaller than the MSE of $\Empirical$: the ratio rises from $1.035$ at $n = 200$ and settles around $1.075$ from $n = 1000$ on.
For the estimator built from the whole data set, the improvement is markedly stronger and essentially stable, the ratio decreasing slowly from $1.218$ to $1.194$.
Even for data sets of small size, the asymptotic approximations on which Corollary~\ref{cor:data-partition3} rests therefore seem to remain informative, at least on this model.

\begin{table}[htbp]
  \centering
  \caption{Finite-horizon MSE of $\Empirical$, of $\PiHatParam[n-m]{\ParamEstimator}{1} = \Empirical[n-m] \KernelMatrixParam{\ParamEstimator}$ (with the best split ratio estimated by plug-in) and of $\PiHatParam{\ParamEstimator[n]}{1} = \Empirical \KernelMatrixParam{\ParamEstimator[n]}$, as a function of the horizon $n$ of the data set, estimated from $10\,000$ trajectories for each horizon, for the model with $\theta = 0.4$, $L = 40$ and $\varepsilon = 0.01$.
    All MSE values are given in units of $10^{-3}$.
    Each ratio is that of the MSE of $\Empirical$ to the MSE of the estimator of the corresponding block; a value larger than $1$ indicates an improvement over $\Empirical$.}
  \label{tab:simus-mse}
  \begin{tabular}{lcccccc}
    \toprule
    Horizon $n$ & $200$ & $500$ & $1\,000$ & $2\,000$ & $5\,000$ & $10\,000$ \\
    \midrule
    $\MSE(\Empirical)$
      & $15.397$ & $6.548$ & $3.290$ & $1.647$ & $0.674$ & $0.334$ \\
    \addlinespace[0.8ex]
    \multicolumn{7}{l}{\emph{Plug-in estimator, best split ratio (plug-in estimate)}} \\
    \quad $\MSE(\PiHatParam[n-m]{\ParamEstimator}{1})$
      & $14.873$ & $6.182$ & $3.059$ & $1.541$ & $0.627$ & $0.310$ \\
    \quad ratio
      & $1.035$ & $1.059$ & $1.076$ & $1.069$ & $1.075$ & $1.075$ \\
    \addlinespace[0.8ex]
    \multicolumn{7}{l}{\emph{Plug-in estimator, whole data set}} \\
    \quad $\MSE(\PiHatParam{\ParamEstimator[n]}{1})$
      & $12.644$ & $5.442$ & $2.744$ & $1.376$ & $0.565$ & $0.280$ \\
    \quad ratio
      & $1.218$ & $1.203$ & $1.199$ & $1.197$ & $1.193$ & $1.194$ \\
    \bottomrule
  \end{tabular}
\end{table}

\subsection{Analysis of real data: bike availability in the V\'elo'v system of Lyon}
\label{subsec:birth-death-real-data}

The purpose of this subsection is to apply the plug-in look-ahead method to a real data set that the birth-and-death process of Subsection~\ref{subsec:birth-death-theory} can plausibly describe.
We consider the V\'elo'v bike-sharing system of the Lyon metropolitan area, France, a network of $454$ stations.
Each station carries a fixed number of docks, every one of which is either free or holds a bike (mechanical or electric); a user may take a bike from an occupied dock at any time of the day or of the night, and must return it to a free dock of some station.
The network is known, together with the GPS position of each station and its number of docks.\footnote{\href{https://data.grandlyon.com/portail/fr/jeux-de-donnees/stations-velo-v-metropole-lyon/info}{Grand Lyon open data portal}, data set \emph{Stations V\'elo'v de la M\'etropole de Lyon}.}
Figure~\ref{fig:velov-map} shows the stations of the network, together with the boundary of the 9th district of Lyon, on which part of the study is focused.

The usage history of the system between January and June 2026 is also available,\footnote{\href{https://data.grandlyon.com/portail/fr/jeux-de-donnees/stations-velo-v-de-la-metropole-de-lyon---disponibilites-temps-reel/info}{Grand Lyon open data portal}, data set \emph{Stations V\'elo'v de la M\'etropole de Lyon --- disponibilit\'es temps r\'eel}.} that is, for each station, the number of available bikes over time.
The feed is not sampled on a regular grid: two consecutive records are a median of $119$~seconds and a mean of $194$~seconds apart.

Two families of stations are discarded before any analysis.
For $61$ of them, the number of available bikes reported in the raw feed exceeds the number of docks of the station at least once.
For $25$ further stations, the history is missing or empty over at least one of the six months.
Removing both families leaves the $368$ stations on which the study is carried out.

For a given station, what happens between two consecutive records cannot be observed.
A bike may for instance have been returned and another one taken in the meantime, leaving the number of available bikes unchanged in spite of two hidden events.
We therefore assume that, between two records, the smallest number of events explaining the observed variation has occurred: for instance, a variation of $0$ bikes is read as no event at all, and a variation of $-2$ bikes as two bikes taken in succession.
Such events are imputed at regular intervals between the two records, and the counts they produce are referred to as imputed.
Table~\ref{tab:velov-summary} summarises the data set at the scale of the whole network (imputed counts included) and Figure~\ref{fig:velov-lyon9} details the stations of the 9th district.

\begin{figure}[htbp]
  \centering
  \includegraphics[trim=122 0 122 0, clip, width=\MapWidth]%
    {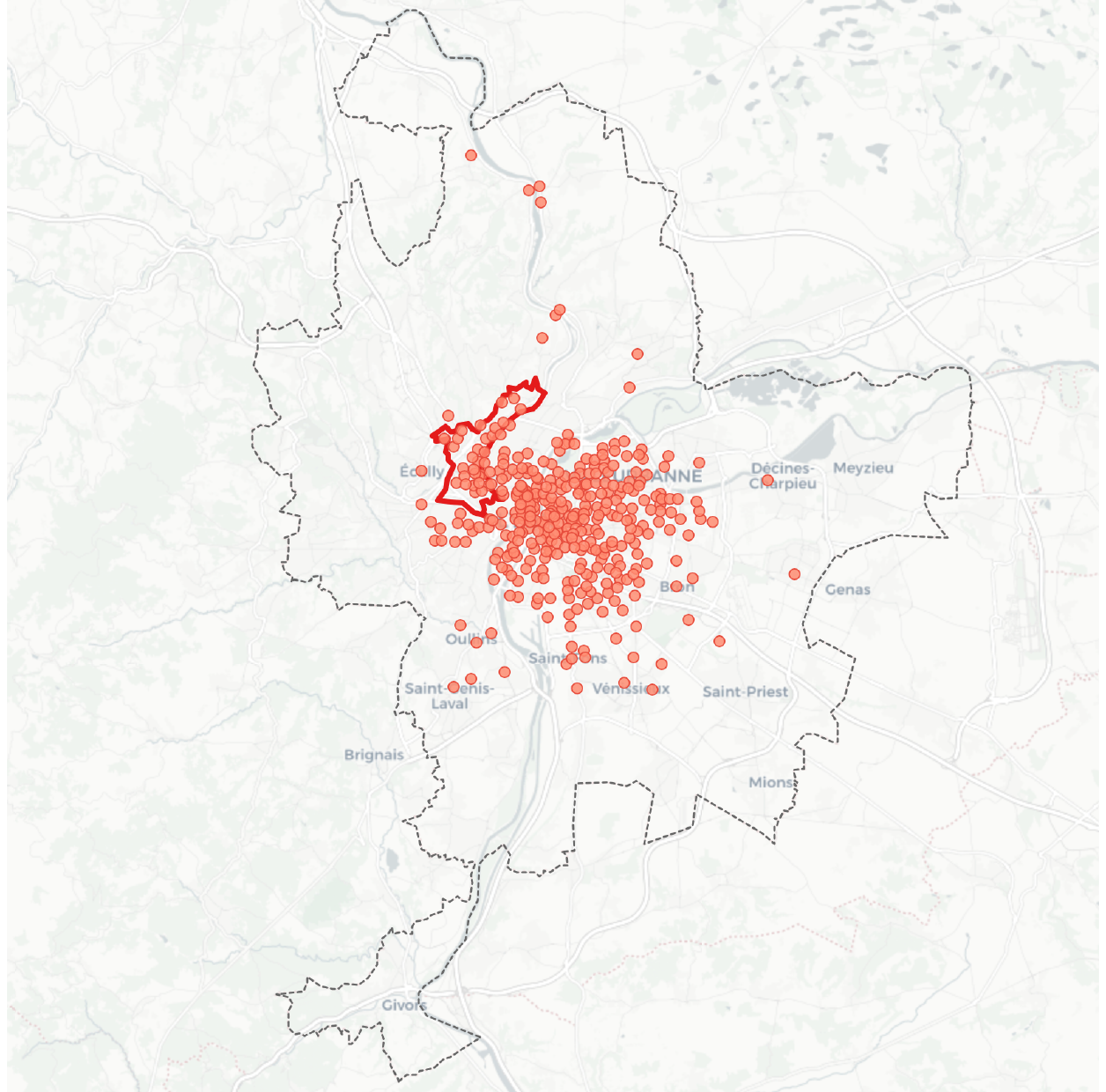}
  \caption{The $368$ stations of the V\'elo'v network of the Lyon metropolitan area retained for the statistical study, together with the boundary of the 9th district of Lyon, on which part of the study is focused.}
  \label{fig:velov-map}
\end{figure}

\begin{table}[htbp]
  \centering
  \caption{Summary of the usage history of the V\'elo'v system between January and June 2026, imputed counts included.}
  \label{tab:velov-summary}
  \small
  \begin{tabular}{lr@{\hspace{3em}}lr}
    \toprule
    \multicolumn{4}{c}{\emph{Network}} \\
    \midrule
    Number of stations & $368$ & Average number of docks & $20.99$ \\
    \midrule
    \multicolumn{4}{c}{\emph{Number of counts per time slot}} \\
    \midrule
    06:00--10:00 & $2\,500\,350$ & 14:00--18:00 & $3\,826\,854$ \\
    10:00--14:00 & $2\,785\,519$ & 18:00--06:00 & $3\,809\,677$ \\
    \midrule
    Total number of counts & $12\,922\,400$
      & Proportion of imputed counts & $13.80\%$ \\
    \bottomrule
  \end{tabular}
\end{table}

\begin{figure}[htbp]
  \centering
  \includegraphics[width=\PlotWidth]{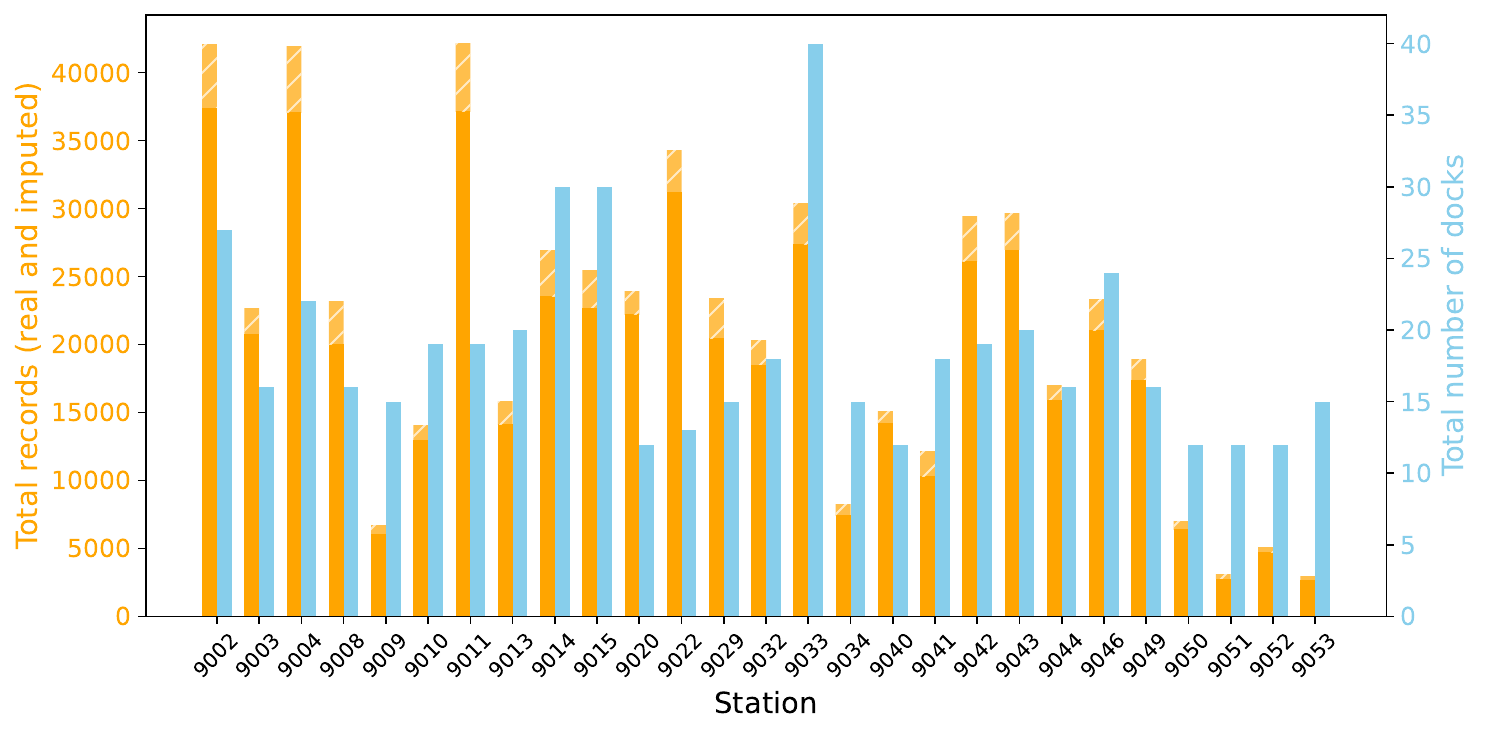}
  \caption{Number of counts available between January and June 2026 for each station of the 9th district of Lyon, split into real and imputed counts, together with the number of docks of the station.}
  \label{fig:velov-lyon9}
\end{figure}

Only the numbers of available bikes at the instants of variation are retained, those variations being of at most one unit by the convention above.
For each station this yields a discrete-time process on $\IntegerRange{0}{L}$, where $L$ is the number of docks of the station, whose increments are exactly $\pm 1$: a process of the birth-and-death type studied in Subsection~\ref{subsec:birth-death-theory}.
Fitting that model requires the dynamics to be homogeneous in time, which is hard to sustain over a whole day, if only because of commuting patterns.
The day is therefore split into four time slots, namely $06$:$00$--$10$:$00$, $10$:$00$--$14$:$00$, $14$:$00$--$18$:$00$ and $18$:$00$--$06$:$00$.
We assume that, for each station and within each slot, the data are generated by a time-homogeneous birth-and-death process with $\varepsilon = 0$ and some $\theta \in \; ]0,1[$.
The objective is then to estimate, for each station and each slot, the invariant distribution of that model, so as to draw relevant conclusions on the underlying process.
The plug-in look-ahead estimator used throughout is $\PiHatParam{\ParamEstimator[n]}{1}$, built from the whole data of the slot without any splitting, as in the last estimator compared in Subsection~\ref{subsec:birth-death-simulations}.
Figures~\ref{fig:velov-9022} and~\ref{fig:velov-9002} display the data and the estimates, for the empirical and for the plug-in look-ahead estimators, at two stations of the 9th district whose behaviours differ markedly.
In both cases, the plug-in look-ahead estimator uses the information carried by the model to smooth the empirical distribution.

\afterpage{%
\begin{landscape}
  \begin{figure}
    \centering
    \begin{subfigure}{\QuadPlotWidth}
      \centering
      \includegraphics[width=\linewidth]{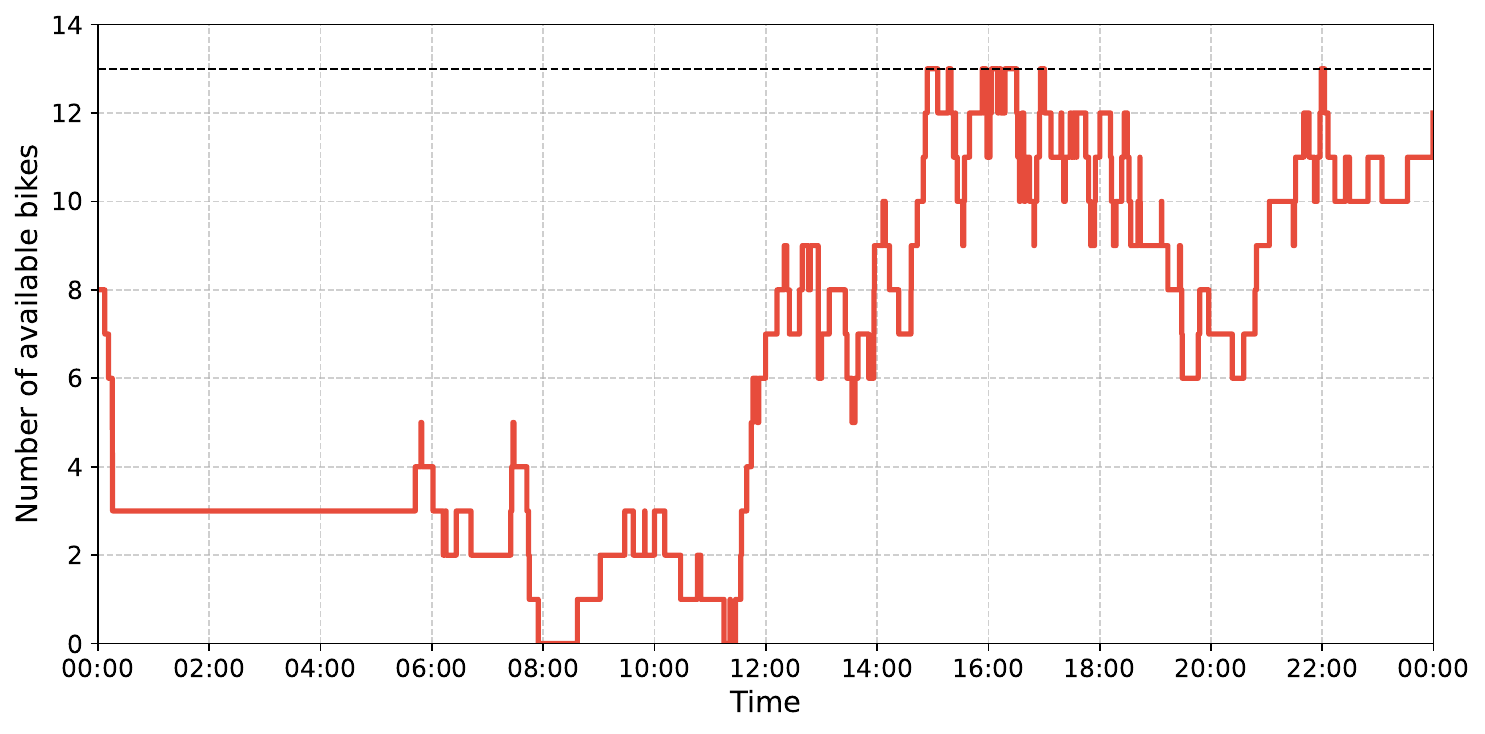}
      \caption{}
      \label{fig:velov-9022-trajectory}
    \end{subfigure}
    \hfill
    \begin{subfigure}{\QuadPlotWidth}
      \centering
      \includegraphics[width=\linewidth]{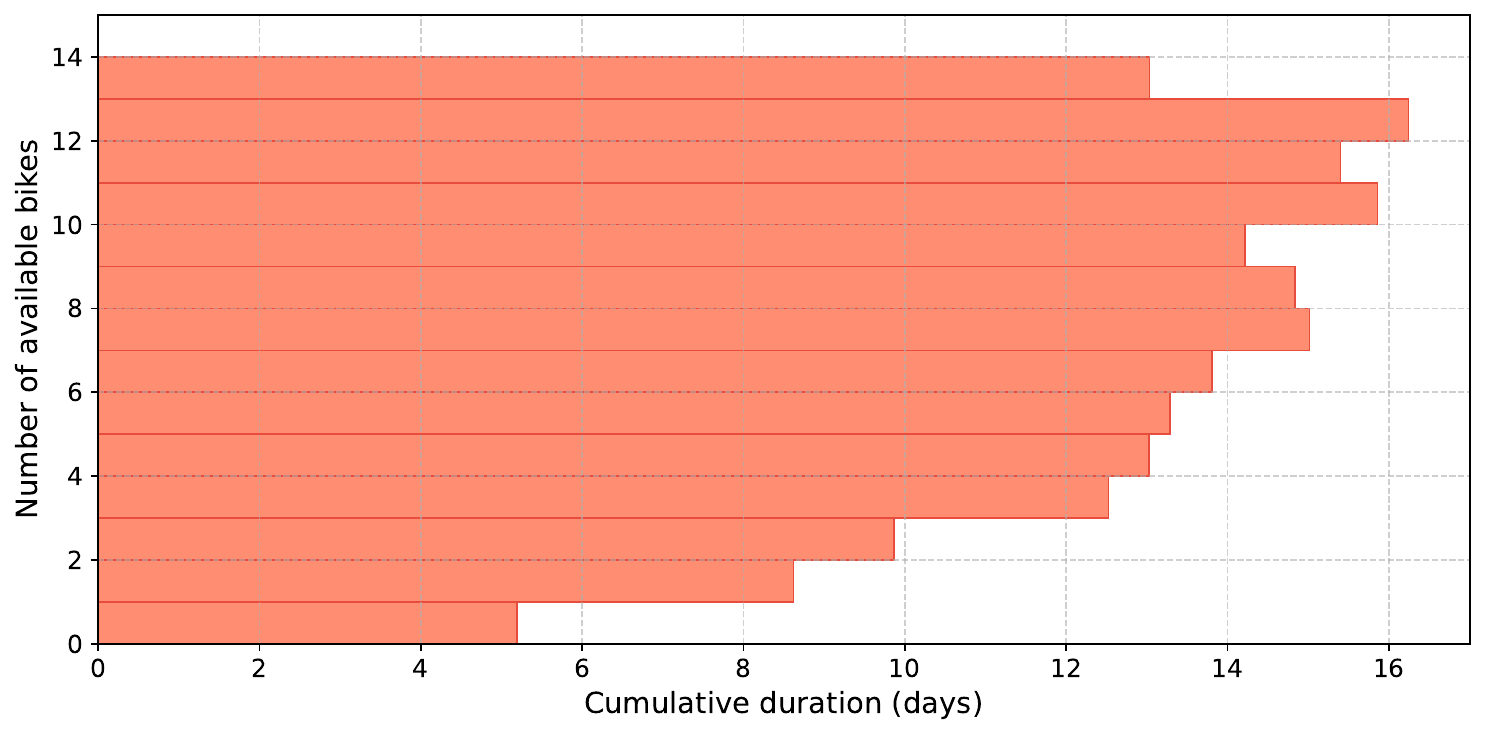}
      \caption{}
      \label{fig:velov-9022-duration}
    \end{subfigure}

    \vspace{\PlotGap}

    \begin{subfigure}{\QuadPlotWidth}
      \centering
      \includegraphics[width=\linewidth]{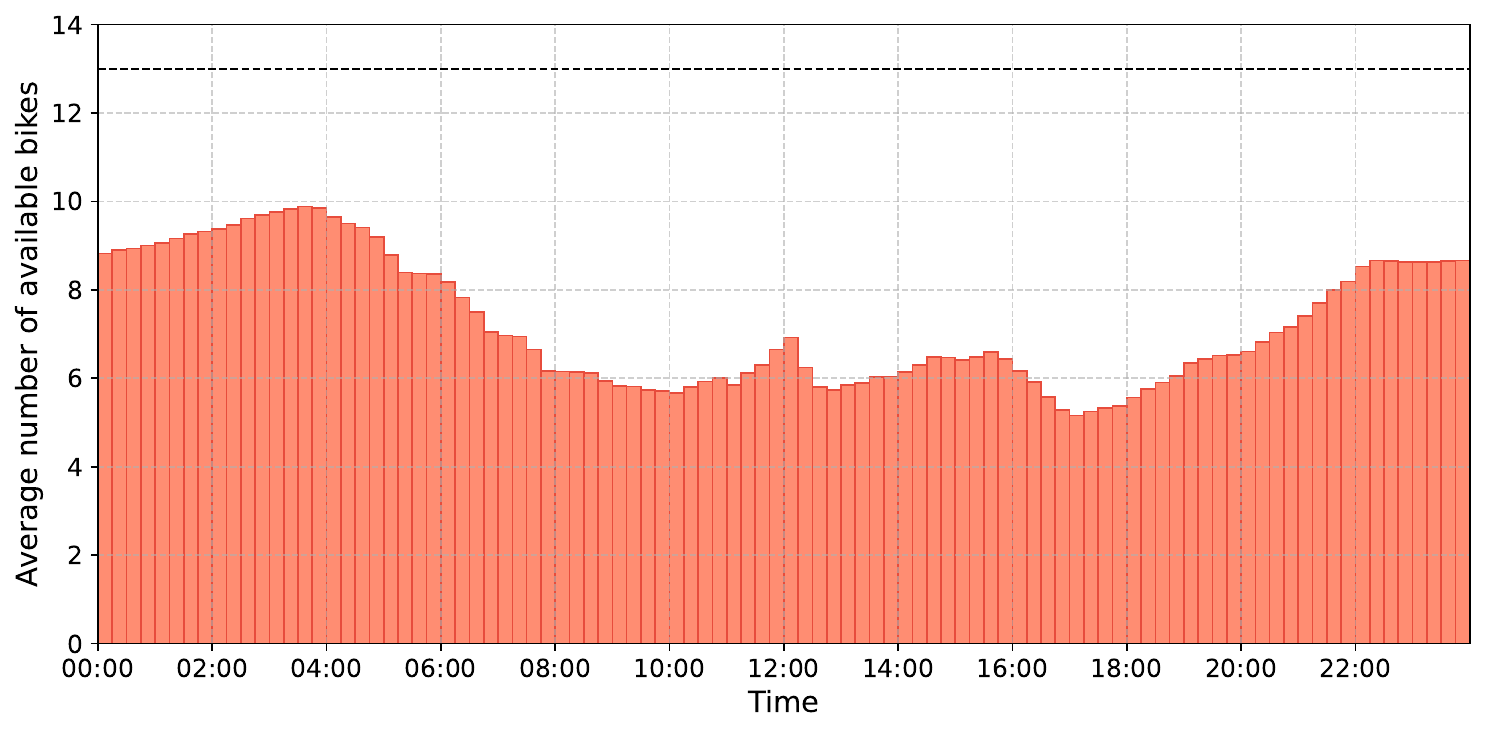}
      \caption{}
      \label{fig:velov-9022-distribution}
    \end{subfigure}
    \hfill
    \begin{subfigure}{\QuadPlotWidth}
      \centering
      \includegraphics[width=\linewidth]{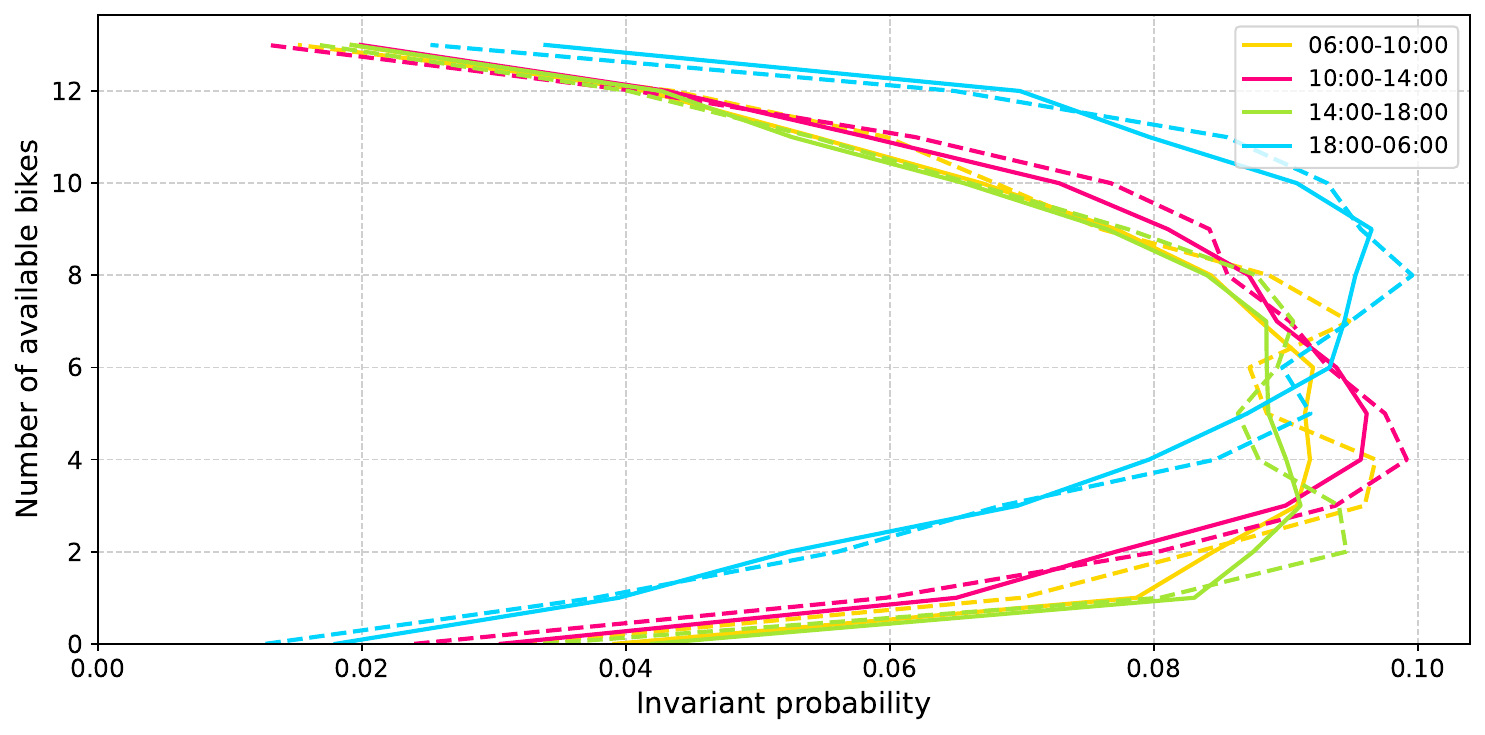}
      \caption{}
      \label{fig:velov-9022-invariant}
    \end{subfigure}
    \caption{Station $9022$, which has $13$ docks.
      \textbf{A}~Number of available bikes on Thursday 2 April 2026.
      \textbf{B}~Cumulative time spent at each level of availability between January and June 2026.
      \textbf{C}~Average number of available bikes over the day.
      \textbf{D}~Invariant distribution estimated on each time slot; dashed lines correspond to the empirical estimator, solid lines to the plug-in look-ahead estimator.}
    \label{fig:velov-9022}
  \end{figure}
\end{landscape}
\begin{landscape}
  \begin{figure}
    \centering
    \begin{subfigure}{\QuadPlotWidth}
      \centering
      \includegraphics[width=\linewidth]{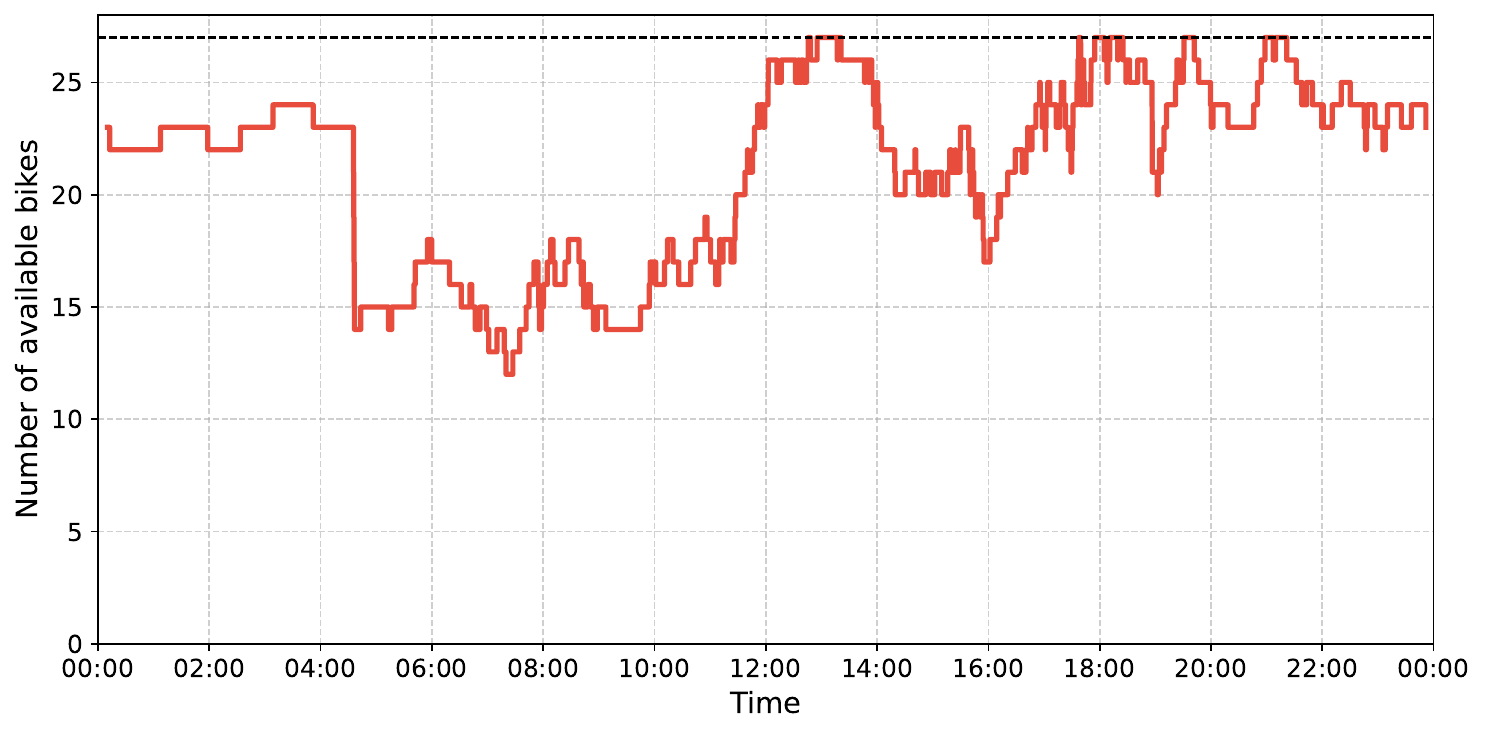}
      \caption{}
      \label{fig:velov-9002-trajectory}
    \end{subfigure}
    \hfill
    \begin{subfigure}{\QuadPlotWidth}
      \centering
      \includegraphics[width=\linewidth]{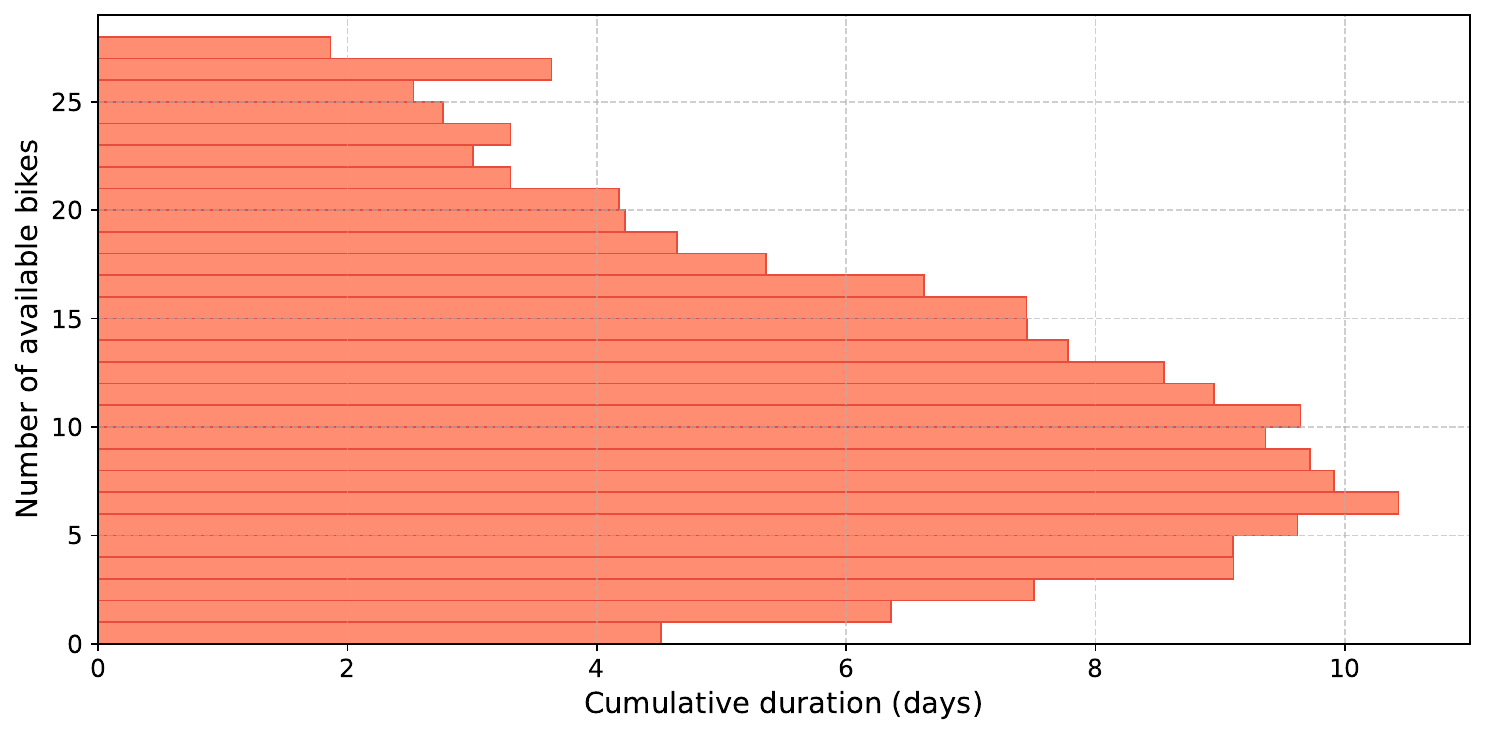}
      \caption{}
      \label{fig:velov-9002-duration}
    \end{subfigure}

    \vspace{\PlotGap}

    \begin{subfigure}{\QuadPlotWidth}
      \centering
      \includegraphics[width=\linewidth]{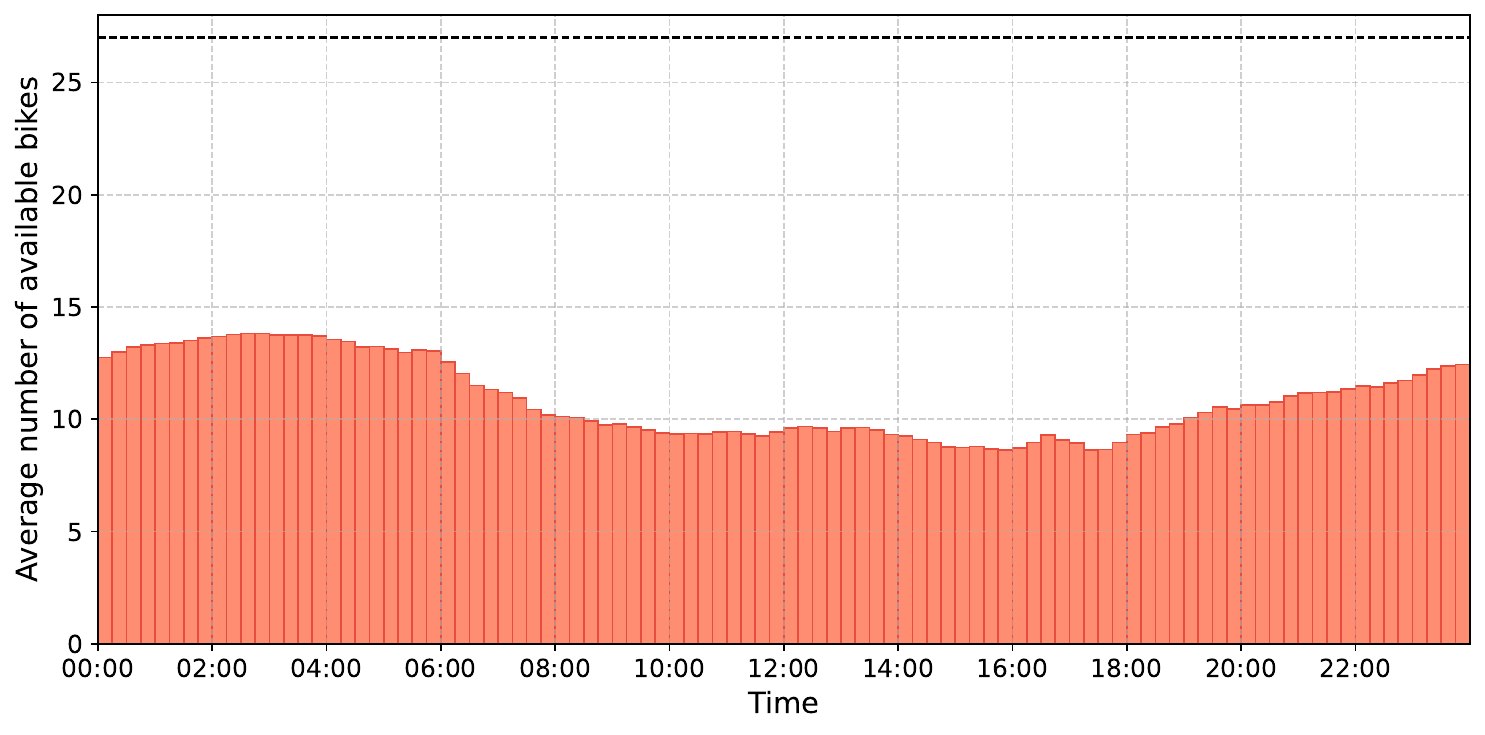}
      \caption{}
      \label{fig:velov-9002-distribution}
    \end{subfigure}
    \hfill
    \begin{subfigure}{\QuadPlotWidth}
      \centering
      \includegraphics[width=\linewidth]{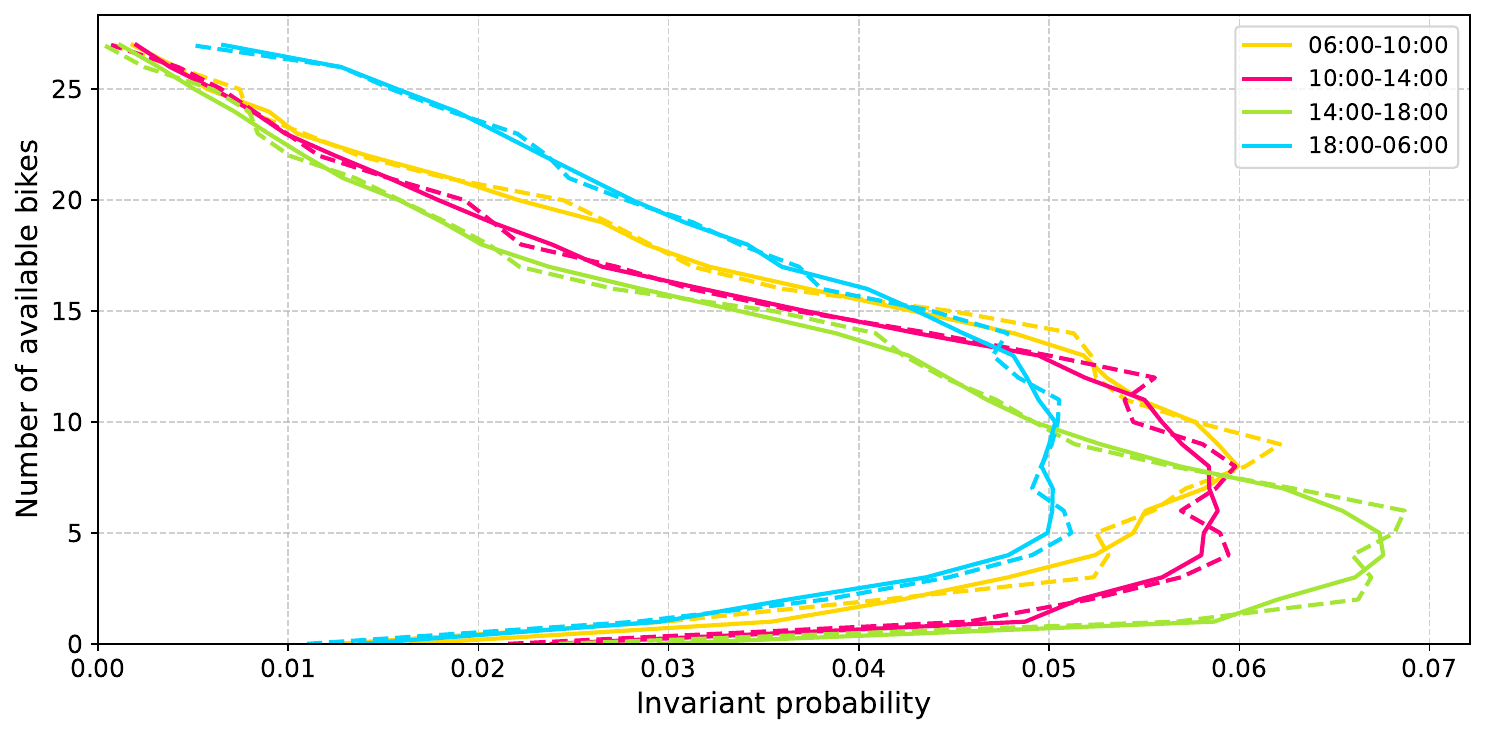}
      \caption{}
      \label{fig:velov-9002-invariant}
    \end{subfigure}
    \caption{Station $9002$, which has $27$ docks: same panels as in Figure~\ref{fig:velov-9022}.
      \textbf{A}~Number of available bikes on Thursday 2 April 2026.
      \textbf{B}~Cumulative time spent at each level of availability between January and June 2026.
      \textbf{C}~Average number of available bikes over the day.
      \textbf{D}~Invariant distribution estimated on each time slot; dashed lines correspond to the empirical estimator, solid lines to the plug-in look-ahead estimator.}
    \label{fig:velov-9002}
  \end{figure}
\end{landscape}
}

Beyond the invariant distribution itself, a quantity of interest for this problem is the availability of a station, that is, the probability $1-\pi(0)-\pi(L)$ that it offers at the same time an available bike and a free dock.
It can be evaluated from the empirical estimator and from the plug-in look-ahead estimator.
Figure~\ref{fig:velov-boxplots} compares the two with the value obtained by solving the invariance equation $\pi = \pi \KernelMatrixParam{\ParamEstimator[n]}$ for the estimated kernel, which provides no ground truth but a reference that assumes the model to be exact.
Relative to that reference, the error of the plug-in look-ahead estimator is the smaller of the two, as expected from the theoretical and numerical results of Section~\ref{sec:theory-parametric-kernel} and of Subsection~\ref{subsec:birth-death-simulations}.
A second, less immediate phenomenon also appears.
Consider for instance the probability that no bike is available.
The plug-in look-ahead estimator tends to over-estimate it, relative to the empirical estimator, when $\theta$ is estimated below $1/2$, and to under-estimate it when $\theta$ is estimated above $1/2$.
This is what one should expect: a station with $\theta > 1/2$ tends to receive more bikes than it gives, and that information is carried by the transition kernel, hence available to the plug-in look-ahead estimator but not to the empirical one, as Figure~\ref{fig:velov-estim-vs-theta} shows.
The same phenomenon, reversed, is observed for the probability that the station is full.
Seeing the dynamics through the transition kernel therefore improves the estimation of the quantity of interest, always under the assumption that the model is correct.
The estimates obtained for every time slot and every station of the 9th district are reported in Figure~\ref{fig:velov-lyon9-results} and Table~\ref{tab:velov-availability}.

\begin{figure}[htbp]
  \centering
  \includegraphics[width=\PlotWidth]{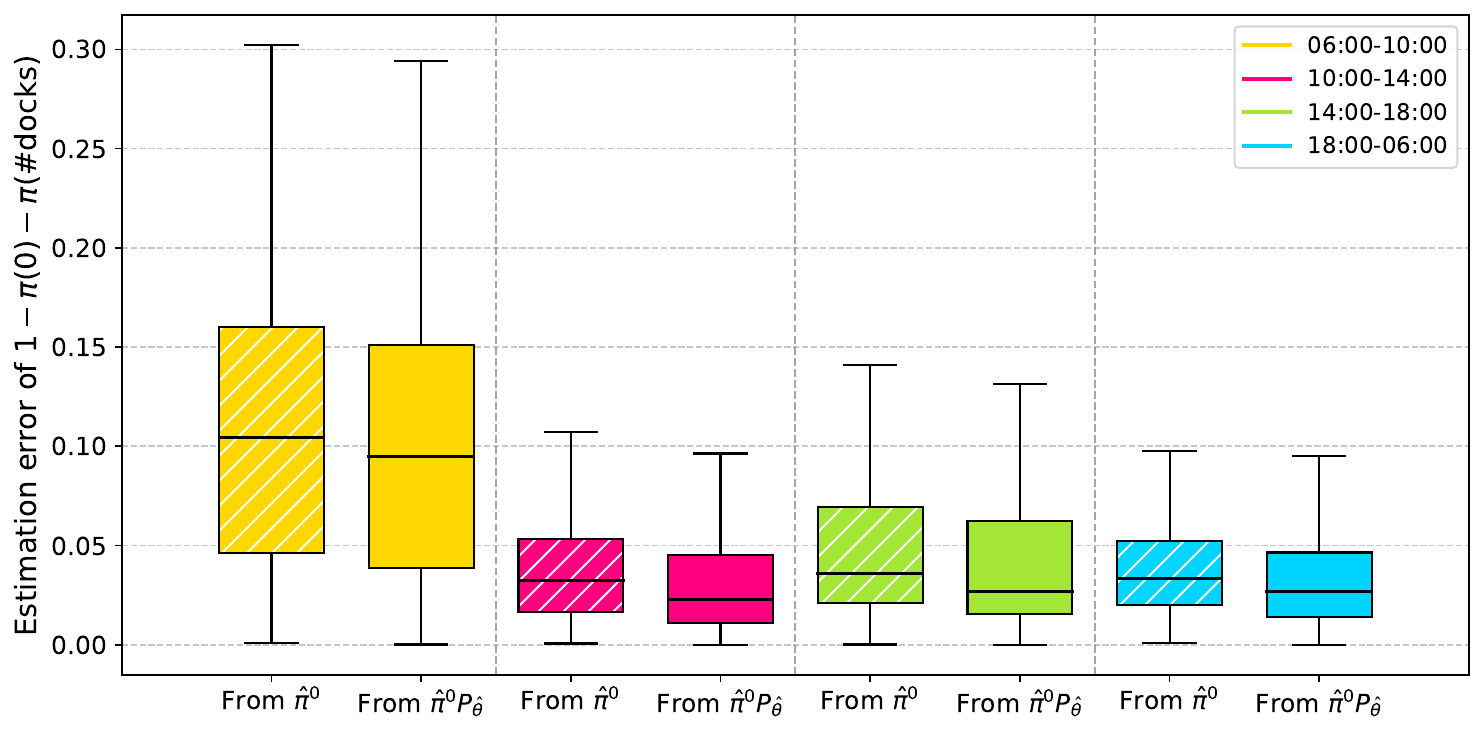}
  \caption{Distribution, over the $368$ stations of the network and for each time slot, of the error made on the availability $1-\pi(0)-\pi(L)$ by the empirical estimator and by the plug-in look-ahead estimator.
    The reference is the value obtained by solving the invariance equation for the estimated kernel.}
  \label{fig:velov-boxplots}
\end{figure}

\begin{figure}[htbp]
  \centering
  \includegraphics[width=\PlotWidth]{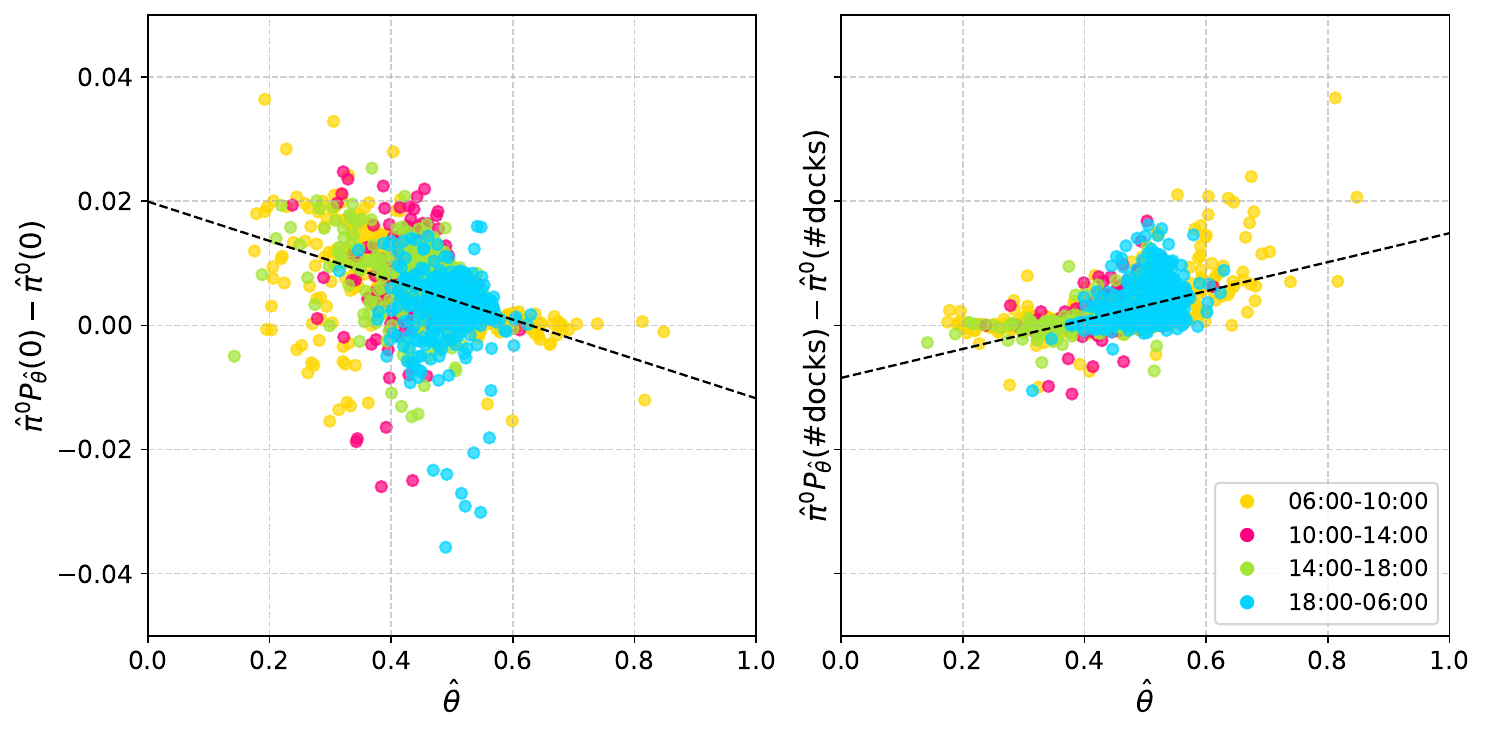}
  \caption{Difference between the plug-in look-ahead estimator and the empirical estimator, over the $368$ stations of the network, as a function of the estimated parameter, for the probability that a station is empty (left) and for the probability that it is full (right).}
  \label{fig:velov-estim-vs-theta}
\end{figure}

\begin{figure}[htbp]
  \centering
  \includegraphics[trim=122 0 122 0, clip, width=\MapWidth]%
    {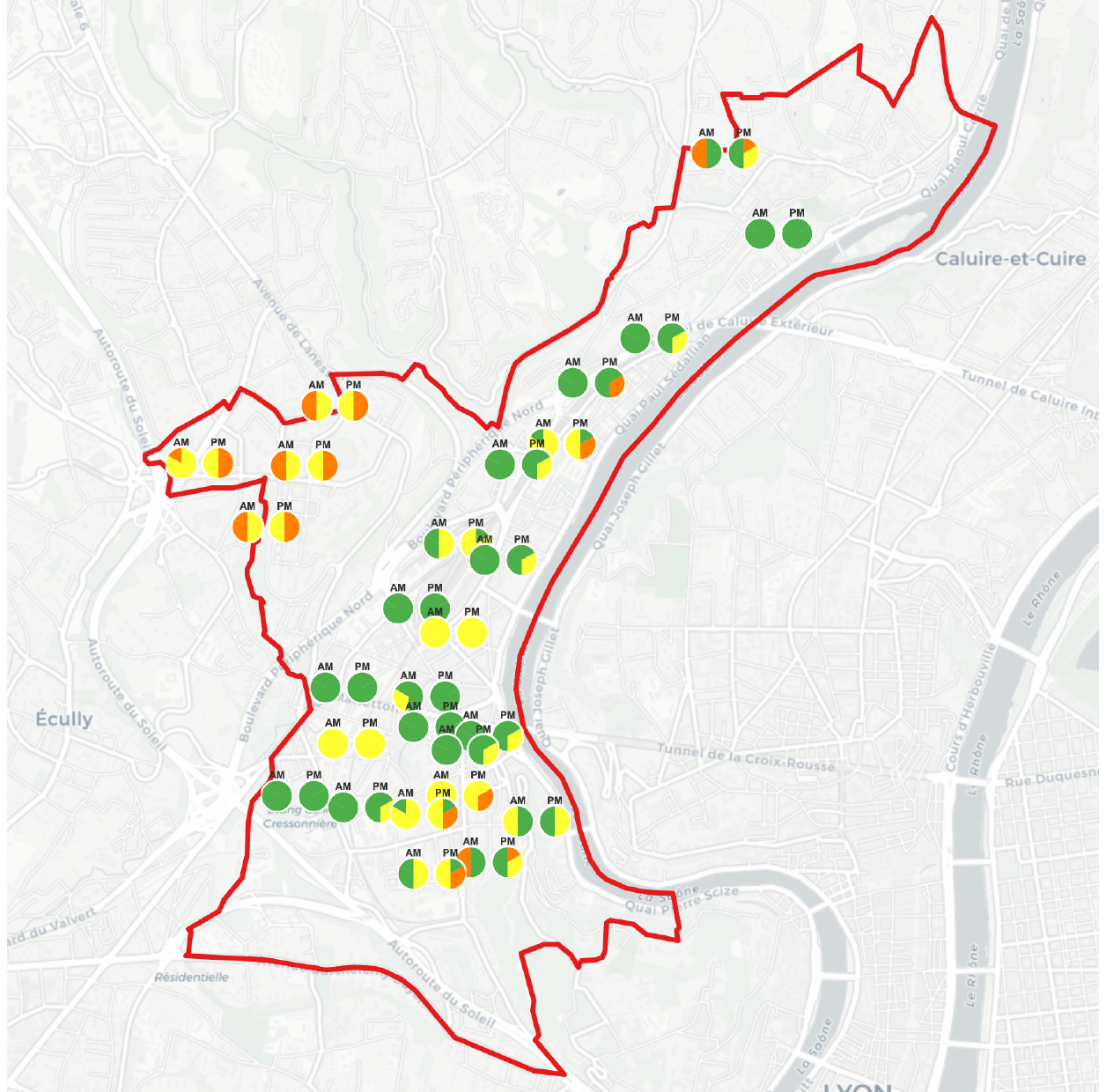}
  \caption{Plug-in look-ahead estimate of the probability that a station is available, that is, neither empty nor full, for the stations of the 9th district of Lyon and for each time slot.
    Green stands for a probability above $95\%$, yellow for a probability between $90$ and $95\%$, orange between $80$ and $90\%$, and red below $80\%$.
    The corresponding results over the whole metropolitan area are summarised in Table~\ref{tab:velov-availability}.}
  \label{fig:velov-lyon9-results}
\end{figure}

\begin{table}[htbp]
  \centering
  \caption{Plug-in look-ahead estimates of the probability that a station is available, that is, neither empty nor full, over the $368$ stations of the metropolitan area and for each time slot.
    For each availability class, the table gives the number of stations falling in the class and the share of the network it represents.}
  \label{tab:velov-availability}
  \small
  \begin{tabular*}{\textwidth}{@{\extracolsep{\fill}}lrrrrrrrr}
    \toprule
    & \multicolumn{8}{c}{Estimated availability (\%)} \\
    \cmidrule{2-9}
    & \multicolumn{2}{c}{\textcolor{AvailRed}{$[0,80)$}}
    & \multicolumn{2}{c}{\textcolor{AvailOrange}{$[80,90)$}}
    & \multicolumn{2}{c}{\textcolor{AvailYellow}{$[90,95)$}}
    & \multicolumn{2}{c}{\textcolor{AvailGreen}{$[95,100]$}} \\
    \cmidrule(lr){2-3} \cmidrule(lr){4-5} \cmidrule(lr){6-7} \cmidrule(lr){8-9}
    Time slot & Stations & Share & Stations & Share
              & Stations & Share & Stations & Share \\
    \midrule
    06:00--10:00 & $12$ & $3\%$       & $89$ & $24\%$ & $123$ & $33\%$ & $144$ & $39\%$ \\
    10:00--14:00 & $9$  & $2\%$       & $52$ & $14\%$ & $116$ & $32\%$ & $191$ & $52\%$ \\
    14:00--18:00 & $2$  & $<\!1\%$    & $58$ & $16\%$ & $155$ & $42\%$ & $153$ & $42\%$ \\
    18:00--06:00 & $0$  & $0\%$       & $24$ & $7\%$  & $167$ & $45\%$ & $177$ & $48\%$ \\
    \bottomrule
  \end{tabular*}
\end{table}

\bibliographystyle{plain}
\bibliography{references}

\newpage
\appendix

\section{Proofs of Section~\ref{sec:theory-known-kernel}}
\label{app:known-kernel}

\subsection{Proof of Proposition~\ref{prop:variance-iid-functional}}
\label{app:proof-variance-monotone}

For $\IndexBetween{i}{1}{n}$, we have
\[
  \Expect\left[\Contribution{k}{i}{f}^{2}\right]
  = \SumOver{y} \pi(y)\left(\SumOver{z} f(z)\Kernel[k]{z}{y}\right)^{2}.
\]
By Jensen's inequality,
\[
\begin{aligned}
  \Expect\left[\Contribution{k}{i}{f}^{2}\right]
  &\leq \SumOver{y} \pi(y) \SumOver{z} f(z)^{2} \Kernel[k]{z}{y} \\
  &= \SumOver{z} f(z)^{2} \SumOver{y} \pi(y) \Kernel[k]{z}{y}
   = \SumOver{z} f(z)^{2} \pi(z)
   = \pi\!\left(f^{2}\right).
\end{aligned}
\]
Hence we immediately deduce that $\PiHatOf{k}{f} \in \Ltwopi$ and that
\[
  \Variance\left(\PiHatOf{k}{f}\right)
  = \frac{1}{n} \left(\SumOver{y} \pi(y)\left(\SumOver{z} f(z)\Kernel[k]{z}{y}\right)^{2}
    - \pi(f)^2\right).
\]
We now prove that $\Variance(\PiHatOf{k}{f})$ is decreasing in $k$.
We have
\[
\begin{aligned}
  \Expect\left[\Contribution{k+1}{i}{f}^{2}\right]
  &= \SumOver{y} \pi(y)\left(\SumOver{z} f(z)\Kernel[k+1]{z}{y}\right)^{2} \\
  &= \SumOver{y} \pi(y)\left(\SumOver{z} f(z) \SumOver{x} \Kernel{x}{y}\Kernel[k]{z}{x}\right)^{2} \\
  &= \SumOver{y} \pi(y)\left(\SumOver{x} \Kernel{x}{y}
     \left\{\SumOver{z} f(z)\Kernel[k]{z}{x}\right\}\right)^{2}.
\end{aligned}
\]
Using Jensen's inequality again,
\[
\begin{aligned}
  \Expect\left[\Contribution{k+1}{i}{f}^{2}\right]
  &\leq \SumOver{y} \pi(y) \SumOver{x} \Kernel{x}{y}
     \left\{\SumOver{z} f(z)\Kernel[k]{z}{x}\right\}^{2} \\
  &= \SumOver{x} \pi(x)\left\{\SumOver{z} f(z)\Kernel[k]{z}{x}\right\}^{2}
   = \Expect\left[\Contribution{k}{i}{f}^{2}\right].
\end{aligned}
\]
The result directly follows from this inequality.
Proposition~\ref{prop:variance-iid-pointwise} follows by taking $f = \Dirac{x}$ for $x \in \StateSpace$.

\subsection{Proof of Proposition~\ref{prop:clt-functional}}
\label{app:proof-clt}

We begin with the following lemma.

\begin{lemma}
\label{lem:covariance-contributions}
Assume that the data $(X_i)_{\IndexBetween{i}{1}{n}}$ form a stationary Markov chain with kernel $\KernelMatrix$.
For $f \in \Ltwopi$, $k,l \in \Naturals$ and $i,j \in \Naturals$,
\[
  \Covariance\left(\Contribution{k}{i}{f}, \Contribution{l}{j}{f}\right)
  = \begin{cases}
      \SumOver{y,z} \pi(y)\Kernel[j-i]{z}{y}
      \KernelAverage{k}{y}{f}\KernelAverage{l}{z}{f} - \pi(f)^2
      & \text{if } i \leq j, \\[1ex]
      \SumOver{y,z} \pi(y)\Kernel[i-j]{z}{y}
      \KernelAverage{l}{y}{f}\KernelAverage{k}{z}{f} - \pi(f)^2
      & \text{if } j \leq i,
    \end{cases}
\]
where $\KernelAverage{k}{y}{f} = \SumOver{x} f(x)\Kernel[k]{x}{y}$.
\end{lemma}

\begin{proof}
To prove this result, we need to evaluate the expectation $\Expect[\Contribution{k}{i}{f}\Contribution{l}{j}{f}]$.
A direct calculation yields
\[
\begin{aligned}
  \Expect\left[\Contribution{k}{i}{f}\Contribution{l}{j}{f}\right]
  &= \SumOver{y,z} \Prob\left(X_i = y, X_j = z\right)
     \left(\SumOver{x} f(x)\Kernel[k]{x}{y}\right)
     \left(\SumOver{x} f(x)\Kernel[l]{x}{z}\right) \\
  &= \begin{cases}
       \SumOver{y,z} \pi(y)\Kernel[j-i]{z}{y}
       \KernelAverage{k}{y}{f}\KernelAverage{l}{z}{f} & \text{if } i \leq j, \\[1ex]
       \SumOver{y,z} \pi(z)\Kernel[i-j]{y}{z}
       \KernelAverage{k}{y}{f}\KernelAverage{l}{z}{f} & \text{if } j \leq i,
     \end{cases} \\
  &= \begin{cases}
       \SumOver{y,z} \pi(y)\Kernel[j-i]{z}{y}
       \KernelAverage{k}{y}{f}\KernelAverage{l}{z}{f} & \text{if } i \leq j, \\[1ex]
       \SumOver{y,z} \pi(y)\Kernel[i-j]{z}{y}
       \KernelAverage{l}{y}{f}\KernelAverage{k}{z}{f} & \text{if } j \leq i,
     \end{cases}
\end{aligned}
\]
which gives the expected formula.
\end{proof}

Proposition~\ref{prop:clt-functional} stems from the central limit theorem for Markov chains, the limiting variance being given, for $k \in \Naturals$ and $f \in \Ltwopi$, by
\[
  \AsymptoticVariance[k]{f}
  = \Variance\left(\Contribution{k}{1}{f}\right)
  + 2 \SumFromTo{i}{2}{+\infty}
    \Covariance\left(\Contribution{k}{1}{f}, \Contribution{k}{i}{f}\right),
\]
where the expression of $\Variance(\Contribution{k}{1}{f})$ was established in Appendix~\ref{app:proof-variance-monotone} and the expression of $\Covariance(\Contribution{k}{1}{f}, \Contribution{k}{i}{f})$ is directly derived from Lemma~\ref{lem:covariance-contributions}.

Proposition~\ref{prop:clt-pointwise} is a direct application of Proposition~\ref{prop:clt-functional} with $f = \Dirac{x}$, $x \in \StateSpace$, for which some computations allow us to write
\[
\begin{aligned}
  \AsymptoticVariance[k]{x}
  &= \SumOver{y} \pi(y)\left(\SumOver{z} \Dirac{x}(z)\Kernel[k]{z}{y}\right)^{2}
     - \pi\left(\Dirac{x}\right)^{2} \\
  &\quad + 2 \SumFromTo{i}{1}{+\infty}\left(\SumOver{y,z} \pi(y)\Kernel[i]{z}{y}
     \KernelAverage{k}{y}{\Dirac{x}}\KernelAverage{k}{z}{\Dirac{x}}
     - \pi\left(\Dirac{x}\right)^{2}\right) \\
  &= \SumOver{y} \pi(y)\Kernel[k]{x}{y}^{2} - \pi(x)^{2}
     + 2 \SumFromTo{i}{1}{+\infty}\left(\SumOver{y,z} \pi(y)\Kernel[i]{z}{y}
     \Kernel[k]{x}{y}\Kernel[k]{x}{z} - \pi(x)^{2}\right) \\
  &= \SumOver{y} \pi(y)\Kernel[k]{x}{y}^{2} - \pi(x)^{2}
     + 2 \SumFromTo{i}{1}{+\infty}\left(\SumOver{y} \pi(y)\Kernel[k]{x}{y}
     \SumOver{z} \Kernel[i]{z}{y}\Kernel[k]{x}{z} - \pi(x)^{2}\right) \\
  &= \SumOver{y} \pi(y)\Kernel[k]{x}{y}^{2} - \pi(x)^{2}
     + 2 \SumFromTo{i}{1}{+\infty}\left(\SumOver{y} \pi(y)\Kernel[k]{x}{y}
     \Kernel[i+k]{x}{y} - \pi(x)^{2}\right) \\
  &= \SumOver{y} \pi(y)\Kernel[k]{x}{y}^{2} - \pi(x)^{2}
     + 2 \SumFromTo{i}{1}{+\infty}\left(\SumOver{y} \pi(y)\Kernel[k]{x}{y}
     \left\{\Kernel[i+k]{x}{y} - \pi(x)\right\}\right),
\end{aligned}
\]
which ends the proof.

\subsection{Proof of Lemma~\ref{lem:summability}}
\label{app:proof-summability}

For $i \in \Naturals$ and $x \in \StateSpace$, we define
\[
  F_i(x) = \SumOver{y} \pi(y)\Kernel[k]{x}{y}
  \left\{\Kernel[i+k]{x}{y} - \pi(x)\right\}.
\]
Thus,
\[
  \left|F_i(x)\right|
  \leq \SumOver{y} \pi(y)\Kernel[k]{x}{y}
     \left|\Kernel[i+k]{x}{y} - \pi(x)\right|.
\]
Since $\KernelMatrix$ is uniformly ergodic, there exist $C \geq 0$ and $\rho \in [0,1[$ such that
\[
  \forall y \in \StateSpace,\ \forall i \in \Naturals, \quad
  \NormTV{\Kernel[i]{\cdot}{y} - \pi} = \SumOver{x} \left|\Kernel[i]{x}{y} - \pi(x)\right| \leq C \rho^{i}.
\]
This leads to
\[
\begin{aligned}
  \SumOver{x} \left|F_i(x)\right|
  &\leq \SumOver{x} \SumOver{y} \pi(y)\Kernel[k]{x}{y}
     \left|\Kernel[i+k]{x}{y} - \pi(x)\right| \\
  &= \SumOver{y} \SumOver{x} \pi(y)\Kernel[k]{x}{y}
     \left|\Kernel[i+k]{x}{y} - \pi(x)\right| \\
  &\leq \SumOver{y} \pi(y) \SumOver{x}
     \left|\Kernel[i+k]{x}{y} - \pi(x)\right|
   \leq C \rho^{i+k}.
\end{aligned}
\]
Hence,
\[
\begin{aligned}
  \SumOver{x} \AsymptoticVariance[k]{x}
  &= \SumOver{x}\left(\SumOver{y} \pi(y)\Kernel[k]{x}{y}^{2} - \pi(x)^{2}
     + 2 \SumFromTo{i}{1}{+\infty} F_i(x)\right) \\
  &\leq \SumOver{x} \SumOver{y} \pi(y)\Kernel[k]{x}{y}^{2}
     + 2 \SumOver{x} \SumFromTo{i}{1}{+\infty}\left|F_i(x)\right| \\
  &\leq \SumOver{x} \SumOver{y} \pi(y)\Kernel[k]{x}{y}
     + 2 \SumFromTo{i}{1}{+\infty} \SumOver{x}\left|F_i(x)\right| \\
  &\leq 2 \SumOver{x} \pi(x) + 2 \SumFromTo{i}{1}{+\infty} C \rho^{i+k}
   \leq 2 + \frac{2C \rho^{k+1}}{1-\rho} < +\infty.
\end{aligned}
\]
This proves the first part of Lemma~\ref{lem:summability}.
For the second part, Proposition~\ref{prop:covariance-markov} gives
\[
\begin{aligned}
  \SumOver{x} \Variance\left(\PiHatAt{k}{x}\right)
  &= \SumOver{x}\left\{\frac{1}{n}\left(\SumOver{y} \pi(y)\Kernel[k]{x}{y}^{2} - \pi(x)^{2}\right)
     + \frac{2}{n^{2}} \SumFromTo{i}{1}{n}(n-i) F_i(x)\right\} \\
  &= \frac{1}{n} \SumOver{x} \AsymptoticVariance[k]{x}
    - \frac{2}{n} \SumOver{x} \SumFromTo{i}{1}{+\infty} F_i(x)
    + \frac{2}{n^{2}} \SumOver{x} \SumFromTo{i}{1}{n}(n-i) F_i(x) \\
  &= \frac{1}{n} \SumOver{x} \AsymptoticVariance[k]{x}
     - \frac{2}{n} \SumOver{x}\left(\frac{1}{n} \SumFromTo{i}{1}{n} i F_i(x)
     + \SumFromTo{i}{n+1}{+\infty} F_i(x)\right).
\end{aligned}
\]
We define, for $x \in \StateSpace$,
\[
  \mathcal{E}_n(x) = \frac{1}{n} \SumFromTo{i}{1}{n} i F_i(x)
  + \SumFromTo{i}{n+1}{+\infty} F_i(x),
\]
and we have
\[
\begin{aligned}
  \SumOver{x} \left|\mathcal{E}_n(x)\right|
  &= \SumOver{x}\left|\frac{1}{n} \SumFromTo{i}{1}{n} i F_i(x)
     + \SumFromTo{i}{n+1}{+\infty} F_i(x)\right| \\
  &\leq \SumOver{x}\left(\SumFromTo{i}{1}{n} \frac{i}{n}\left|F_i(x)\right|
     + \SumFromTo{i}{n+1}{+\infty}\left|F_i(x)\right|\right) \\
  &= \SumFromTo{i}{1}{n} \frac{i}{n} \SumOver{x}\left|F_i(x)\right|
     + \SumFromTo{i}{n+1}{+\infty} \SumOver{x}\left|F_i(x)\right| \\
  &\leq \SumFromTo{i}{1}{n} \frac{i}{n} C \rho^{i+k}
     + \SumFromTo{i}{n+1}{+\infty} C \rho^{i+k} \\
  &\leq \frac{C}{n} \frac{\rho^{k+1}}{(1-\rho)^{2}}
     + C \frac{\rho^{k}}{1-\rho} \rho^{n+1}.
\end{aligned}
\]
This proves that $\SumOver{x}|\mathcal{E}_n(x)|$ goes to $0$ when $n$ goes to infinity, and hence the second part of Lemma~\ref{lem:summability}.

\subsection{Proof of Theorem~\ref{thm:monotone-amse}}\label{proof:thm:monotone-amse}

\subsubsection{One-step comparison in the general case}

We begin by establishing the comparison between $\PiHat{1}$ and $\Empirical$ in the general non-reversible setting.
For any $x \in \StateSpace$, we can write
\[
  \AsymptoticVariance[0]{x} - \AsymptoticVariance[1]{x} = T_1(x) + 2 T_2(x),
\]
where
\[
\begin{aligned}
  T_1(x) &= \pi(x) - \SumOver{y} \pi(y) \Kernel{x}{y}^{2}, \\
  T_2(x) &= \SumFromTo{i}{1}{+\infty}
    \left\{\pi(x) \Kernel[i]{x}{x}
    - \SumOver{y} \pi(y) \Kernel{x}{y} \Kernel[i+1]{x}{y}\right\}.
\end{aligned}
\]
Through expansion and rearrangement, we obtain
\[
\begin{aligned}
  T_2(x)
  &= \SumFromTo{i}{1}{+\infty}
     \left\{\SumOver{y} \pi(y) \Kernel{x}{y} \Kernel[i]{x}{x}
     - \SumOver{y} \pi(y) \Kernel{x}{y} \Kernel[i+1]{x}{y}\right\} \\
  &= \SumFromTo{i}{1}{+\infty} \SumOver{y} \pi(y) \Kernel{x}{y}
     \left\{\Kernel[i]{x}{x} - \Kernel[i+1]{x}{y}\right\} \\
  &= \SumOver{y} \pi(y) \Kernel{x}{y}
     \SumFromTo{i}{1}{+\infty}
     \left\{\Kernel[i]{x}{x} - \Kernel[i+1]{x}{y}\right\}.
\end{aligned}
\]
For $y \in \StateSpace$, since $\Kernel[i+1]{x}{y}$ converges at an exponential rate to $\pi(x)$ when $i$ goes to infinity, we may write
\[
  \SumFromTo{i}{1}{+\infty}\left\{\Kernel[i]{x}{x} - \Kernel[i+1]{x}{y}\right\}
  = \SumFromTo{i}{1}{+\infty}\left\{\Kernel[i]{x}{x} - \Kernel[i]{x}{y}\right\}
  + \Big(\Kernel{x}{y} - \pi(x)\Big).
\]
It follows that
\[
\begin{aligned}
  T_2(x)
  &= \SumOver{y} \pi(y) \Kernel{x}{y}
     \left(\SumFromTo{i}{1}{+\infty}\left\{\Kernel[i]{x}{x} - \Kernel[i]{x}{y}\right\}
     + \Kernel{x}{y} - \pi(x)\right) \\
  &= \SumOver{y} \pi(y) \Kernel{x}{y}
     \SumFromTo{i}{1}{+\infty}\left\{\Kernel[i]{x}{x} - \Kernel[i]{x}{y}\right\}
     + \SumOver{y} \pi(y) \Kernel{x}{y}^{2} - \pi(x)^{2} \\
  &= \SumOver{y} \pi(y) \Kernel{x}{y}
     \left(\SumFromTo{i}{0}{+\infty}\left\{\Kernel[i]{x}{x} - \Kernel[i]{x}{y}\right\}
     - \Big(\Kernel[0]{x}{x} - \Kernel[0]{x}{y}\Big)\right) \\
  &\quad + \SumOver{y} \pi(y) \Kernel{x}{y}^{2} - \pi(x)^{2} \\
  &= \SumOver{y} \pi(y) \Kernel{x}{y}
     \left(\SumFromTo{i}{0}{+\infty}\left\{\Kernel[i]{x}{x} - \Kernel[i]{x}{y}\right\}\right) \\
  &\quad - \pi(x) + \pi(x)\Kernel{x}{x}
     + \SumOver{y} \pi(y) \Kernel{x}{y}^{2} - \pi(x)^{2}.
\end{aligned}
\]
Let $\GreenFunction{x}{y} = \SumFromTo{i}{0}{+\infty}\{\Kernel[i]{x}{y} - \pi(x)\}$ denote the centred Green function.
We can write
\[
  \SumFromTo{i}{0}{+\infty}\left\{\Kernel[i]{x}{x} - \Kernel[i]{x}{y}\right\}
  = \GreenFunction{x}{x} - \GreenFunction{x}{y}.
\]
We also introduce $\HittingTime{y}{x}$, the expected hitting time (or mean first passage time) from $y$ to $x$,
\[
  \HittingTime{y}{x} = \frac{\GreenFunction{x}{x} - \GreenFunction{x}{y}}{\pi(x)}.
\]
This identity implies that
\[
  \SumFromTo{i}{0}{+\infty}\left\{\Kernel[i]{x}{x} - \Kernel[i]{x}{y}\right\}
  = \pi(x)\, \HittingTime{y}{x}.
\]
By recombining $T_1(x)$ and $T_2(x)$, we have
\[
\begin{aligned}
  \AsymptoticVariance[0]{x} - \AsymptoticVariance[1]{x}
  &= \SumOver{y} \pi(y) \Kernel{x}{y}^{2} - \pi(x) + 2\pi(x)\Kernel{x}{x}
     - 2\pi(x)^{2} \\
  &\quad + 2 \SumOver{y} \pi(y)\pi(x) \Kernel{x}{y} \HittingTime{y}{x}.
\end{aligned}
\]
For a fixed $x \in \StateSpace$, we define the vector $h = (\HittingTime{y}{x})_{y \in \StateSpace}$.
Given the standard identity for hitting times,
\[
  \forall y \neq x, \quad
  \HittingTime{y}{x} = 1 + \SumOver{z} \Kernel{z}{y} \HittingTime{z}{x},
\]
and
\[
  1 + \SumOver{z} \Kernel{z}{x} \HittingTime{z}{x} = \frac{1}{\pi(x)},
\]
we derive that $h$ satisfies the following operator equation,
\[
  h = \Ones + \KernelMatrix h - \frac{\Dirac{x}}{\pi(x)}.
\]
Manipulating the inner product $\ScalarProduct{u}{v} = \SumOver{y} u(y)v(y)\pi(y)$ on $\Ltwopi$, we obtain
\[
  \ScalarProduct{\Kernel{x}{\cdot}}{(\Identity - \KernelMatrix)h}
  = \SumOver{y} \Kernel{x}{y}\left(1 - \frac{\Dirac{x}(y)}{\pi(x)}\right)\pi(y)
  = \pi(x) - \Kernel{x}{x}.
\]
On the other hand, since the adjoint of $\KernelMatrix$ for this inner product is
\[
  \Kernel[*]{y}{x} = \frac{\pi(y)\Kernel{x}{y}}{\pi(x)},
\]
we also have
\[
\begin{aligned}
  \ScalarProduct{\Kernel{x}{\cdot}}{(\Identity - \KernelMatrix)h}
  &= \ScalarProduct{(\Identity - \KernelMatrix[*])\Kernel{x}{\cdot}}{h} \\
  &= \SumOver{y,z} (\Identity - \KernelMatrix[*])_{y,z} \Kernel{x}{z} h(y) \pi(y) \\
  &= \SumOver{y,z}
     \left(\Dirac{y}(z) - \frac{\pi(z)\Kernel{y}{z}}{\pi(y)}\right)
     \Kernel{x}{z} h(y) \pi(y) \\
  &= \SumOver{y} \Kernel{x}{y} h(y) \pi(y)
     - \SumOver{y,z} \pi(z) \Kernel{y}{z} \Kernel{x}{z} h(y).
\end{aligned}
\]
By multiplying by $\pi(x)$ and then summing over $x$, we obtain
\[
\begin{aligned}
  \SumOver{x} \pi(x)^{2} - \SumOver{x} \pi(x)\Kernel{x}{x}
  &= \SumOver{x,y} \pi(x)\pi(y)\Kernel{x}{y} \HittingTime{y}{x} \\
  &\quad - \SumOver{x,y,z} \pi(x)\pi(z)\Kernel{y}{z}\Kernel{x}{z} \HittingTime{y}{x}.
\end{aligned}
\]
Our quantity of interest for comparing the two estimators is
\[
\begin{aligned}
  I &= \SumOver{x} \AsymptoticVariance[0]{x} - \SumOver{x} \AsymptoticVariance[1]{x} \\
    &= \SumOver{x,y} \pi(y)\Kernel{x}{y}^{2} - 1
       + 2 \SumOver{x,y,z} \pi(x)\pi(z)\Kernel{y}{z}\Kernel{x}{z} \HittingTime{y}{x}.
\end{aligned}
\]
Noticing that
\[
\begin{aligned}
  \SumOver{y,z} \pi(z)\Kernel{y}{z}^{2}
  + \SumOverDistinctXY \pi(z)\Kernel{x}{z}\Kernel{y}{z}
  &= \SumOver{x,y,z} \pi(z)\Kernel{x}{z}\Kernel{y}{z} \\
  &= \SumOver{z} \pi(z)\left(\SumOver{x} \Kernel{x}{z}\right)^{2} = 1,
\end{aligned}
\]
and substituting into $I$, we get
\[
\begin{aligned}
  I &= 2 \SumOver{x,y,z} \pi(x)\pi(z)\Kernel{y}{z}\Kernel{x}{z} \HittingTime{y}{x}
       - \SumOverDistinctXY \pi(z)\Kernel{x}{z}\Kernel{y}{z} \\
    &= \SumOverDistinctXY \pi(z)\Kernel{x}{z}\Kernel{y}{z}
       \Big(2\pi(x)\HittingTime{y}{x} - 1\Big).
\end{aligned}
\]
Moreover, by symmetry in $x$ and $y$,
\[
  I = \SumOverDistinctXY \pi(z)\Kernel{x}{z}\Kernel{y}{z}
      \Big(\pi(x)\HittingTime{y}{x} + \pi(y)\HittingTime{x}{y} - 1\Big).
\]
For $x \neq y$, let $\CommuteTime{y}{x} = \inf\{t \geq T_y \,:\, X_t = x\}$ be the commute time from $y$ to $x$, where $T_y = \inf\{t \geq 0 \,:\, X_t = y\}$ denotes the hitting time of $y$.
Then
\[
  \HittingTime{y}{x} + \HittingTime{x}{y}
  = \Expect_x\left[\CommuteTime{y}{x}\right]
  \geq \frac{1}{\pi(x)},
\]
which translates into
\[
  1 \leq \pi(x)\HittingTime{y}{x} + \pi(x)\HittingTime{x}{y}.
\]
If $\pi(x) \leq \pi(y)$, we finally get
\[
  1 \leq \pi(x)\HittingTime{y}{x} + \pi(y)\HittingTime{x}{y},
\]
and the symmetry in $x$ and $y$ also ensures the result when $\pi(y) \leq \pi(x)$.
Hence $I$ is non-negative as a sum of non-negative terms, which proves that
\[
  \AsymptoticMSE\left(\PiHat{1}\right) \leq \AsymptoticMSE\left(\Empirical\right).
\]

\subsubsection{The reversible case}

We now assume that $\KernelMatrix$ is reversible, meaning that $\KernelMatrix$ is a self-adjoint bounded operator on $\Ltwopi$.
Let $\StationaryProjector$ be the bounded operator defined on $\Ltwopi$ by
\[
  \StationaryProjector h = \langle \pi, h \rangle \Ones,
\]
where $\langle \cdot, \cdot \rangle$ denotes the classical Euclidean scalar product.
Remark that $\KernelMatrix \StationaryProjector = \StationaryProjector \KernelMatrix = \StationaryProjector^{2} = \StationaryProjector$ and thus that $(\KernelMatrix - \StationaryProjector)^{j} = \KernelMatrix[j] - \StationaryProjector$ for $j \geq 1$.
For $x \in \StateSpace$, we have, by~\cite[Lemma~21.2.9]{DMPS2018},
\[
  \left|(\KernelMatrix - \StationaryProjector)^{j} f(x)\right|
  = \left|\KernelMatrix[j] f(x) - \pi(f)\right|
  \leq \NormTV{\Kernel[j]{\cdot}{x} - \pi}^{1/2}
       \left(\KernelMatrix[j] f^{2}(x) + \pi(f^{2})\right)^{1/2}.
\]
Hence, since $\KernelMatrix$ is uniformly ergodic, there exist $C > 0$ and $\rho\in[0,1[$ such that
\[
  \left|(\KernelMatrix - \StationaryProjector)^{j} f(x)\right|
  \leq \sqrt{C}\, \rho^{j}
       \left(\KernelMatrix[j] f^{2}(x) + \pi(f^{2})\right)^{1/2},
\]
and
\[
  \NormPi{(\KernelMatrix - \StationaryProjector)^{j} f}^{2}
  = \pi\!\left(\left|(\KernelMatrix - \StationaryProjector)^{j} f\right|^{2}\right)
  \leq 2C\, \rho^{2j} \pi(f^{2})
  = 2C\, \rho^{2j} \NormPi{f}^{2}.
\]
From this it follows that
\[
  \NormOperator{(\KernelMatrix - \StationaryProjector)^{j}}_{\pi} \leq \sqrt{2C}\, \rho^{j}.
\]
This implies that the Neumann series
\[
  \FundamentalMatrix = \sum_{j \geq 0} (\KernelMatrix - \StationaryProjector)^{j}
\]
converges in operator norm and that $\FundamentalMatrix = (\Identity - \KernelMatrix + \StationaryProjector)^{-1}$.
Because taking the adjoint is continuous for this topology \cite[p.~34]{Murphy1990}, and because $\KernelMatrix$ and $\StationaryProjector$ are self-adjoint as operators on $\Ltwopi$, the operator $\FundamentalMatrix$ is self-adjoint as well.

Now, according to \cite[Theorem~21.2.10]{DMPS2018}, for $f \in \Ltwopi$ centred, $\sqrt{n}\, \EmpiricalOf{f}$ converges in distribution to a centred Gaussian random variable with variance
\[
  \AsymptoticVariance[0]{f}
  = \pi(f^{2}) + 2 \sum_{j \geq 1} \pi\!\left(f \KernelMatrix[j] f\right)
  = \ScalarProduct{f}{f} + 2 \sum_{j \geq 1} \ScalarProduct{f}{\KernelMatrix[j] f}.
\]
Because $f$ is centred, we have $\StationaryProjector f = 0$, so we can rewrite
\begin{equation}\label{eq:variancef:from:clt}
\begin{aligned}
  \AsymptoticVariance[0]{f}
  &= \ScalarProduct{f}{(\KernelMatrix - \StationaryProjector)^{0} f}
     + 2 \sum_{j \geq 1} \ScalarProduct{f}{(\KernelMatrix - \StationaryProjector)^{j} f} \\
  &= 2 \sum_{j \geq 0} \ScalarProduct{f}{(\KernelMatrix - \StationaryProjector)^{j} f}
     - \ScalarProduct{f}{f}
   = 2 \ScalarProduct{f}{\FundamentalMatrix f} - \ScalarProduct{f}{f}.
\end{aligned}
\end{equation}
Moreover, for $f,g \in \Ltwopi$ centred and $a,b \in \Reals$, one can deduce that $af + bg$ converges in distribution to a Gaussian random variable and thus, by the characterisation of Gaussian vectors,
\[
  \Big(\sqrt{n}\, \EmpiricalOf{f},\ \sqrt{n}\, \EmpiricalOf{g}\Big)
\]
converges in law to a centred two-dimensional Gaussian random variable whose covariance matrix is denoted by $\Sigma(f,g)$.
For $a,b \in \Reals$,
\[
  \AsymptoticVariance[0]{af + bg}
  = a^{2} \AsymptoticVariance[0]{f} + b^{2} \AsymptoticVariance[0]{g}
  + 2ab\, \Sigma(f,g),
\]
and
\[
\begin{aligned}
  \AsymptoticVariance[0]{af + bg}
  &= 2 \ScalarProduct{af + bg}{\FundamentalMatrix(af + bg)}
     - \ScalarProduct{af + bg}{af + bg} \\
  &= 2a^{2} \ScalarProduct{f}{\FundamentalMatrix f}
     + 2b^{2} \ScalarProduct{g}{\FundamentalMatrix g}
     + 2ab \ScalarProduct{f}{\FundamentalMatrix g}
     + 2ab \ScalarProduct{g}{\FundamentalMatrix f} \\
  &\hspace{4mm}
     - \Big(a^{2} \ScalarProduct{f}{f} + b^{2} \ScalarProduct{g}{g}
     + 2ab \ScalarProduct{f}{g}\Big) \\
  &= a^{2} \AsymptoticVariance[0]{f} + b^{2} \AsymptoticVariance[0]{g}
     + 2ab \Big(\ScalarProduct{f}{\FundamentalMatrix g}
     + \ScalarProduct{g}{\FundamentalMatrix f} - \ScalarProduct{f}{g}\Big).
\end{aligned}
\]
Hence, by identification,
\[
  \Sigma(f,g) = \ScalarProduct{f}{\FundamentalMatrix g}
  + \ScalarProduct{g}{\FundamentalMatrix f} - \ScalarProduct{f}{g}.
\]
Finally, for any $f,g \in \Ltwopi$ (no longer assumed to be centred), the vector
\[
  \Big(\sqrt{n}\left(\EmpiricalOf{f} - \pi(f)\right),\
       \sqrt{n}\left(\EmpiricalOf{g} - \pi(g)\right)\Big)
\]
converges in law to a centred two-dimensional Gaussian random variable with covariance
\[
\begin{aligned}
  \Sigma(f,g)
  &= \ScalarProduct{f - \pi(f)}{\FundamentalMatrix(g - \pi(g))}
   + \ScalarProduct{g - \pi(g)}{\FundamentalMatrix(f - \pi(f))} \\
  &\hspace{0.4cm}
   - \ScalarProduct{f - \pi(f)}{g - \pi(g)} \\[1ex]
  &= \ScalarProduct{f}{\FundamentalMatrix g}
   - \ScalarProduct{\pi(f)}{\FundamentalMatrix g}
   - \ScalarProduct{f}{\FundamentalMatrix \pi(g)} \\
  &\hspace{0.4cm}
   + \ScalarProduct{\pi(f)}{\FundamentalMatrix \pi(g)}
   + \ScalarProduct{g}{\FundamentalMatrix f}
   - \ScalarProduct{\pi(g)}{\FundamentalMatrix f} \\
  &\hspace{0.4cm}
   - \ScalarProduct{g}{\FundamentalMatrix \pi(f)}
   + \ScalarProduct{\pi(g)}{\FundamentalMatrix \pi(f)} \\
  &\hspace{0.4cm}
   - \Big(\ScalarProduct{f}{g} - \ScalarProduct{f}{\pi(g)}
   - \ScalarProduct{\pi(f)}{g} + \pi(f)\pi(g)\Big).
\end{aligned}
\]
We remark that $\FundamentalMatrix \pi(f) = \pi(f)$ and $\FundamentalMatrix \pi(g) = \pi(g)$, hence
\[
\begin{aligned}
  \Sigma(f,g)
  &= \ScalarProduct{f}{\FundamentalMatrix g} - \pi(f)\pi(\FundamentalMatrix g)
   - \ScalarProduct{f}{\pi(g)} + \pi(f)\pi(g) \\
  &\hspace{0.4cm}
   + \ScalarProduct{g}{\FundamentalMatrix f} - \pi(g)\pi(\FundamentalMatrix f)
   - \ScalarProduct{g}{\pi(f)} + \pi(g)\pi(f) \\
  &\hspace{0.4cm}
   - \ScalarProduct{f}{g} + \pi(f)\pi(g) + \pi(f)\pi(g) - \pi(f)\pi(g),
\end{aligned}
\]
and since $\pi(\FundamentalMatrix f) = \pi(f)$ and $\pi(\FundamentalMatrix g) = \pi(g)$,
\[
  \Sigma(f,g) = \ScalarProduct{f}{\FundamentalMatrix g}
  + \ScalarProduct{g}{\FundamentalMatrix f} - \ScalarProduct{f}{g}
  - \pi(f)\pi(g).
\]
As a consequence, for any $f,g \in \Ltwopi$, the vector
\[
  \Big(\sqrt{n}\left(\EmpiricalOf{\KernelMatrix[k] f} - \pi(f)\right),\
       \sqrt{n}\left(\EmpiricalOf{\KernelMatrix[k] g} - \pi(g)\right)\Big)
\]
converges in law to a centred two-dimensional Gaussian random variable with covariance
\[
  \Sigma\!\left(\KernelMatrix[k] f, \KernelMatrix[k] g\right)
  = \ScalarProduct{\KernelMatrix[k] f}{\FundamentalMatrix \KernelMatrix[k] g}
  + \ScalarProduct{\KernelMatrix[k] g}{\FundamentalMatrix \KernelMatrix[k] f}
  - \ScalarProduct{\KernelMatrix[k] f}{\KernelMatrix[k] g} - \pi(f)\pi(g).
\]
Since $\KernelMatrix[k]$ and $\FundamentalMatrix$ are self-adjoint in $\Ltwopi$,
\[
\begin{aligned}
  \Sigma\!\left(\KernelMatrix[k] f, \KernelMatrix[k] g\right)
  &= \ScalarProduct{f}{\KernelMatrix[k] \FundamentalMatrix \KernelMatrix[k] g}
   + \ScalarProduct{\KernelMatrix[k] \FundamentalMatrix \KernelMatrix[k] g}{f}
   - \ScalarProduct{f}{\KernelMatrix[2k] g} - \pi(f)\pi(g) \\
  &= 2 \ScalarProduct{f}{\KernelMatrix[k] \FundamentalMatrix \KernelMatrix[k] g}
   - \ScalarProduct{f}{\KernelMatrix[2k] g} - \pi(f)\pi(g) \\
  &= \ScalarProduct{f}{\KernelMatrix[k](2\FundamentalMatrix - \Identity)\KernelMatrix[k] g}
   - \pi(f)\pi(g).
\end{aligned}
\]
In particular,
\[
  0 \leq \Sigma\!\left(\KernelMatrix[k] f, \KernelMatrix[k] f\right)
  = \ScalarProduct{f}{\KernelMatrix[k](2\FundamentalMatrix - \Identity)\KernelMatrix[k] f}
  - \pi(f)^{2},
\]
which implies that $\KernelMatrix[k](2\FundamentalMatrix - \Identity)\KernelMatrix[k]$ is a self-adjoint positive operator on $\Ltwopi$.

Let $(\HilbertBasis{x})_{x \in \StateSpace} = (\Dirac{x}/\sqrt{\pi(x)})_{x \in \StateSpace}$ be a Hilbert basis of $\Ltwopi$, and set $b_x = \Dirac{x}$.
Denoting by $\DiagonalStationary$ the diagonal matrix whose diagonal is given by $\pi$, we have
\[
\begin{aligned}
  \SumOver{x}
  \ScalarProduct{\HilbertBasis{x}}{(2\FundamentalMatrix - \Identity)\DiagonalStationary \HilbertBasis{x}}
  &= \SumOver{x}
     \ScalarProduct{\HilbertBasis{x}}{(2\FundamentalMatrix - \Identity)\pi(x) \HilbertBasis{x}} \\
  &= \SumOver{x} \pi(x)
     \ScalarProduct{\frac{b_x}{\sqrt{\pi(x)}}}
     {(2\FundamentalMatrix - \Identity) \frac{b_x}{\sqrt{\pi(x)}}} \\
  &= \SumOver{x} \ScalarProduct{b_x}{(2\FundamentalMatrix - \Identity) b_x} \\
  &= \SumOver{x} \left(\Sigma(b_x,b_x) + \pi(x)^{2}\right).
\end{aligned}
\]
However, for $x \in \StateSpace$, $\Sigma(b_x,b_x) = \AsymptoticVariance[0]{x}$, and since Lemma~\ref{lem:summability} ensures that $\SumOver{x} \AsymptoticVariance[0]{x} < +\infty$, we deduce that
\[
  \SumOver{x}
  \ScalarProduct{\HilbertBasis{x}}{(2\FundamentalMatrix - \Identity)\DiagonalStationary \HilbertBasis{x}}
  < +\infty.
\]
As a consequence, $(2\FundamentalMatrix - \Identity)\DiagonalStationary$ is a trace class operator.
Since the space of trace class operators is an ideal, it follows that $\KernelMatrix[k](2\FundamentalMatrix - \Identity)\DiagonalStationary\KernelMatrix[k]$ is trace class as well \cite[Section~3.6]{Murphy1990}, and that for $x \in \StateSpace$,
\[
  \AsymptoticVariance[k]{x}
  = \ScalarProduct{\HilbertBasis{x}}
    {\KernelMatrix[k](2\FundamentalMatrix - \Identity)\KernelMatrix[k] \DiagonalStationary \HilbertBasis{x}}
  - \pi(x)^{2}.
\]
Thus,
\begin{align}
  \TotalVariance{k} = \SumOver{x} \AsymptoticVariance[k]{x}
  &= \SumOver{x}
     \ScalarProduct{\HilbertBasis{x}}
     {\KernelMatrix[k](2\FundamentalMatrix - \Identity)\KernelMatrix[k] \DiagonalStationary \HilbertBasis{x}}
     - \SumOver{x} \pi(x)^{2} \nonumber \\
  &= \Trace\left(\KernelMatrix[k](2\FundamentalMatrix - \Identity)\KernelMatrix[k] \DiagonalStationary\right)
     - \SumOver{x} \pi(x)^{2}.
  \label{eq:totalvariance:from:clt}
\end{align}
Thus,
\[
\begin{aligned}
  \TotalVariance{k} - \TotalVariance{k+1}
  &= \Trace\left(\KernelMatrix[k](2\FundamentalMatrix - \Identity)\KernelMatrix[k] \DiagonalStationary\right)
   - \Trace\left(\KernelMatrix[k+1](2\FundamentalMatrix - \Identity)\KernelMatrix[k+1] \DiagonalStationary\right) \\
  &= \Trace\left((2\FundamentalMatrix - \Identity)\KernelMatrix[k]
     \left(\DiagonalStationary - \KernelMatrix \DiagonalStationary \KernelMatrix\right)
     \KernelMatrix[k]\right) \\
  &= \Trace\left((2\FundamentalMatrix - \Identity)\KernelMatrix[2k]
     \left(\Identity - \KernelMatrix[2]\right) \DiagonalStationary\right).
\end{aligned}
\]
Now, let $u \in \Ltwopi$.
We evaluate the associated quadratic form,
\[
\begin{aligned}
  \ScalarProduct{u}{\KernelMatrix[2] u}
  &= \SumOver{x} \pi(x) (\KernelMatrix u)^{2}(x) \\
  &= \SumOver{x} \pi(x)\left(\SumOver{y} \Kernel{y}{x} u(y)\right)^{2} \\
  &\leq \SumOver{x,y} \pi(x) \Kernel{y}{x} u(y)^{2} \\
  &= \SumOver{y} \pi(y) u(y)^{2}
   = \ScalarProduct{u}{u},
\end{aligned}
\]
using Jensen's inequality.
Thus, $\ScalarProduct{u}{(\Identity - \KernelMatrix[2])u} \geq 0$, and we deduce that $\Identity - \KernelMatrix[2]$ is positive semidefinite.
Now, remark that if $A$ and $B$ are two commuting positive semidefinite self-adjoint bounded operators, then $AB$ remains positive semidefinite.
Indeed, according to \cite[Theorem~2.4.4]{Murphy1990}, if $A^{1/2}$ denotes the square root of $A$, then $B$ commutes with $A^{1/2}$ and
\[
  \ScalarProduct{u}{ABu} = \ScalarProduct{u}{A^{1/2} B A^{1/2} u}
  = \ScalarProduct{A^{1/2} u}{B A^{1/2} u} \geq 0,
\]
since $A^{1/2}$ is self-adjoint and $B$ is positive semidefinite.

It is then easy to see that $\KernelMatrix[2k]$ commutes with $\Identity - \KernelMatrix[2]$, so that $\KernelMatrix[2k](\Identity - \KernelMatrix[2])$ is positive semidefinite.
Then, it follows from $\KernelMatrix \FundamentalMatrix = \FundamentalMatrix \KernelMatrix$ that $(2\FundamentalMatrix - \Identity)$ and $\KernelMatrix[2k](\Identity - \KernelMatrix[2])$ commute, so that $(2\FundamentalMatrix - \Identity)\KernelMatrix[2k](\Identity - \KernelMatrix[2])$ is positive semidefinite.
Ultimately,
\[
  \Trace\left((2\FundamentalMatrix - \Identity)\KernelMatrix[2k](\Identity - \KernelMatrix[2])\DiagonalStationary\right)
\]
is the trace of the product of two positive semidefinite self-adjoint operators, which is therefore non-negative \cite{Simon2005}.
Hence $\TotalVariance{k} - \TotalVariance{k+1} \geq 0$, which completes the proof.

\subsubsection{A non-reversible counterexample}
\label{sssec:counterexample}
The stochastic matrix
\[
\KernelMatrix = \begin{pmatrix}
    0 & 1 & 0 \\
    0 & \frac{1}{10} & \frac{9}{10} \\
    1 & 0 & 0
\end{pmatrix}\!,
\]
whose stationary distribution is $\pi = (9/28, 10/28, 9/28)$, provides an example of a Markov kernel that is irreducible and aperiodic, yet non-reversible, and for which
\[
\AsymptoticMSE\left(\PiHat{1}\right) < \AsymptoticMSE\left(\PiHat{2}\right).
\]

\subsection{Proof of Proposition~\ref{prop:bias-decreasing-iid}}
\label{app:proof-bias-decreasing-iid}

\subsubsection{Independent data}

Let $x \in \StateSpace$.
For $k \in \Naturals$, we have
\[
  \Expect\left[\PiHatAt{k}{x}\right] = \Expect\left[\Contribution{k}{1}{x}\right]
  = \SumOver{y} \mu(y) \Kernel[k]{x}{y}
\]
and
\[
  \Bias\left(\PiHatAt{k}{x}\right) = \SumOver{y} \left(\mu(y)-\pi(y)\right) \Kernel[k]{x}{y}.
\]
Thus,
\[
  \SumOver{x} \Bias\left(\PiHatAt{k}{x}\right)^{2}
  = \SumOver{x} \left(\SumOver{y} \left(\mu(y)-\pi(y)\right) \Kernel[k]{x}{y}\right)^{2}.
\]
We now prove that this quantity is non-increasing in $k$.
We have
\[
\begin{aligned}
  \SumOver{x} \Bias\left(\PiHatAt{k+1}{x}\right)^{2}
  &= \SumOver{x} \left(\SumOver{y} \left(\mu(y)-\pi(y)\right) \Kernel[k+1]{x}{y}\right)^{2} \\
  &= \SumOver{x} \left(\SumOver{y} \left(\mu(y)-\pi(y)\right)
     \left[\SumOver{z} \Kernel[k]{z}{y} \Kernel{x}{z}\right]\right)^{2} \\
  &= \SumOver{x} \left(\SumOver{z} \left[\SumOver{y} \left(\mu(y)-\pi(y)\right)
     \Kernel[k]{z}{y}\right] \Kernel{x}{z}\right)^{2},
\end{aligned}
\]
so that, letting $D_k(z) = \SumOver{y} (\mu(y)-\pi(y)) \Kernel[k]{z}{y}$ for $z \in \StateSpace$,
\[
  \SumOver{x} \Bias\left(\PiHatAt{k+1}{x}\right)^{2}
  = \SumOver{x} \left(\SumOver{z} D_k(z) \Kernel{x}{z}\right)^{2}.
\]
Since $\KernelMatrix$ is doubly stochastic, $\SumOver{z} \Kernel{x}{z} = 1$, hence we can apply Jensen's inequality, which yields
\[
\begin{aligned}
  \SumOver{x} \Bias\left(\PiHatAt{k+1}{x}\right)^{2}
  &\leq \SumOver{x} \SumOver{z} D_k(z)^{2} \Kernel{x}{z}
   = \SumOver{z} D_k(z)^{2} \\
  &= \SumOver{z} \left(\SumOver{y} \left(\mu(y)-\pi(y)\right) \Kernel[k]{z}{y}\right)^{2}
   = \SumOver{x} \Bias\left(\PiHatAt{k}{x}\right)^{2},
\end{aligned}
\]
which proves the first claim.
The reversible kernel
\[
  \KernelMatrix =
  \begin{pmatrix}
    \frac{1}{100} & \frac{1}{100} & \frac{98}{100} \\[0.5ex]
    \frac{1}{100} & \frac{98}{100} & \frac{1}{100} \\[0.5ex]
    \frac{98}{9800} & \frac{1}{9800} & \frac{9701}{9800}
  \end{pmatrix},
\]
whose invariant distribution is $\pi = (1/100, 1/100, 98/100)$, provides the counterexample.
Indeed, for $\mu = (5/1000, 28/1000, 967/1000)$,
\[
  \SumOver{x} \Bias^{2}\!\left(\PiHatAt{1}{x}\right)
  > \SumOver{x} \Bias^{2}\!\left(\PiHatAt{0}{x}\right).
\]

\subsubsection{Markovian data}

For $\IndexBetween{i}{1}{n}$, let $\mu_i = \mu \KernelMatrix[i-1]$, so that $X_i$ is distributed according to $\mu_i$.
Thus, for $x \in \StateSpace$,
\[
  \Expect\left[\PiHatAt{k}{x}\right]
  = \frac{1}{n} \SumFromTo{i}{1}{n} \Expect\left[\Contribution{k}{i}{x}\right]
  = \frac{1}{n} \SumFromTo{i}{1}{n} \SumOver{y} \mu_i(y) \Kernel[k]{x}{y}
\]
and
\[
  \Bias\left(\PiHatAt{k}{x}\right)
  = \frac{1}{n} \SumFromTo{i}{1}{n} \SumOver{y} \left(\mu_i(y)-\pi(y)\right) \Kernel[k]{x}{y}.
\]
Hence,
\[
\begin{aligned}
  \SumOver{x} \Bias\left(\PiHatAt{k}{x}\right)^{2}
  &= \frac{1}{n^{2}} \SumOver{x} \left(\SumFromTo{i}{1}{n}
     \SumOver{y} \left(\mu_i(y)-\pi(y)\right) \Kernel[k]{x}{y}\right)^{2} \\
  &= \frac{1}{n^{2}} \SumOver{x} \left(\SumOver{y} \left(\SumFromTo{i}{1}{n}
     \left(\mu_i(y)-\pi(y)\right)\right) \Kernel[k]{x}{y}\right)^{2}.
\end{aligned}
\]
We now prove that this quantity is non-increasing in $k$.
We have
\[
\begin{aligned}
  \SumOver{x} \Bias\left(\PiHatAt{k+1}{x}\right)^{2}
  &= \frac{1}{n^{2}} \SumOver{x} \left(\SumOver{y} \left(\SumFromTo{i}{1}{n}
     \left(\mu_i(y)-\pi(y)\right)\right) \Kernel[k+1]{x}{y}\right)^{2} \\
  &= \frac{1}{n^{2}} \SumOver{x} \left(\SumOver{y} \left(\SumFromTo{i}{1}{n}
     \left(\mu_i(y)-\pi(y)\right)\right)
     \left[\SumOver{z} \Kernel[k]{z}{y} \Kernel{x}{z}\right]\right)^{2} \\
  &= \frac{1}{n^{2}} \SumOver{x} \left(\SumOver{z} \left[\SumOver{y}
     \left(\SumFromTo{i}{1}{n} \left(\mu_i(y)-\pi(y)\right)\right)
     \Kernel[k]{z}{y}\right] \Kernel{x}{z}\right)^{2},
\end{aligned}
\]
so that, letting $D_k^{n}(z) = \SumOver{y} \SumFromTo{i}{1}{n} (\mu_i(y)-\pi(y)) \Kernel[k]{z}{y}$ for $z \in \StateSpace$,
\[
  \SumOver{x} \Bias\left(\PiHatAt{k+1}{x}\right)^{2}
  = \frac{1}{n^{2}} \SumOver{x} \left(\SumOver{z} D_k^{n}(z) \Kernel{x}{z}\right)^{2}.
\]
Since $\KernelMatrix$ is doubly stochastic, $\SumOver{z} \Kernel{x}{z} = 1$, hence we can apply Jensen's inequality,
\[
\begin{aligned}
  \SumOver{x} \Bias\left(\PiHatAt{k+1}{x}\right)^{2}
  &\leq \frac{1}{n^{2}} \SumOver{x} \SumOver{z} D_k^{n}(z)^{2} \Kernel{x}{z}
   = \frac{1}{n^{2}} \SumOver{z} D_k^{n}(z)^{2} \\
  &= \frac{1}{n^{2}} \SumOver{z} \left(\SumOver{y} \SumFromTo{i}{1}{n}
     \left(\mu_i(y)-\pi(y)\right) \Kernel[k]{z}{y}\right)^{2} \\
  &= \SumOver{x} \Bias\left(\PiHatAt{k}{x}\right)^{2},
\end{aligned}
\]
which proves the first claim.
The counterexample is the same as for independent and identically distributed data.

\subsection{Proof of Proposition~\ref{prop:variance-nonstationary-iid}}
\label{app:proof-variance-nonstationary-iid}

First, we remark that
\[
  \Variance\left(\EmpiricalAt{x}\right) = \frac{1}{n}\Big(\mu(x) - \mu(x)^{2}\Big).
\]
In addition,
\[
\begin{aligned}
  \Variance\left(\PiHatAt{k}{x}\right)
  &= \frac{1}{n}\left(\SumOver{y} \mu(y)\Kernel[k]{x}{y}^{2}
     - \left(\SumOver{y} \mu(y)\Kernel[k]{x}{y}\right)^{2}\right) \\
  &= \frac{1}{n}\left(\SumOver{y,z} \mu(y)\mu(z)\Kernel[k]{x}{y}^{2}
     - \SumOver{y,z} \mu(y)\mu(z)\Kernel[k]{x}{y}\Kernel[k]{x}{z}\right) \\
  &= \frac{1}{n}\left(\frac{1}{2} \SumOver{y,z} \mu(y)\mu(z)
     \Big(\Kernel[k]{x}{y} - \Kernel[k]{x}{z}\Big)^{2}\right).
\end{aligned}
\]
Thus we have
\[
  \SumOver{x} \Variance\left(\PiHatAt{k}{x}\right)
  = \frac{1}{n}\left(\frac{1}{2} \SumOver{y,z} \mu(y)\mu(z)
    \SumOver{x}\Big(\Kernel[k]{x}{y} - \Kernel[k]{x}{z}\Big)^{2}\right).
\]
For $y,z \in \StateSpace$, we observe that
\[
\begin{aligned}
  \SumOver{x}\Big(\Kernel[k]{x}{y} - \Kernel[k]{x}{z}\Big)^{2}
  &\leq \SumOver{x}\left|\Kernel[k]{x}{y} - \Kernel[k]{x}{z}\right| \\
  &\leq \SumOver{x}\Kernel[k]{x}{y} + \SumOver{x}\Kernel[k]{x}{z} = 2,
\end{aligned}
\]
and trivially that $\SumOver{x}(\Kernel[k]{x}{y} - \Kernel[k]{x}{z})^{2} = 0$, if $y = z$, which leads to
\[
\begin{aligned}
  \SumOver{x} \Variance\left(\PiHatAt{k}{x}\right)
  &\leq \frac{1}{n}\left(\frac{1}{2} \SumOver{y,z} \mu(y)\mu(z) \times 2 \times
     \Indicator{y \neq z}\right) \\
  &\leq \frac{1}{n}\left(\SumOver{y,z} \mu(y)\mu(z) - \SumOver{y} \mu(y)^{2}\right) \\
  &\leq \frac{1}{n}\left(1 - \SumOver{y} \mu(y)^{2}\right)
   \leq \SumOver{x} \Variance\left(\EmpiricalAt{x}\right),
\end{aligned}
\]
which states the expected result.

\subsection{Proof of Proposition~\ref{prop:bias-negligible}}
\label{app:proof-bias-negligible}

For $i \in \Naturals^{*}$, let $\mu_i = \mu \KernelMatrix[i-1]$.
For $x \in \StateSpace$ and $k \in \Naturals$,
\[
\begin{aligned}
  \Expect\left[\PiHatAt{k}{x}\right]
  &= \frac{1}{n} \SumFromTo{i}{1}{n} \Expect\left[\Contribution{k}{i}{x}\right] \\
  &= \frac{1}{n} \SumFromTo{i}{1}{n}
     \left(\SumOver{y} \mu_i(y) \Kernel[k]{x}{y}\right)
   = \frac{1}{n} \SumFromTo{i}{1}{n} \mu_{i+k}(x).
\end{aligned}
\]
Hence,
\begin{equation}\label{eq:formula:bias:markov}
  \SumOver{x} \Bias\left(\PiHatAt{k}{x}\right)^{2}
  = \frac{1}{n^{2}} \SumOver{x}\left(\SumFromTo{i}{1}{n}
     \Big(\mu_{i+k}(x) - \pi(x)\Big)\right)^{2},
\end{equation}
and therefore
\[
\begin{aligned}
  \SumOver{x} \Bias\left(\PiHatAt{k}{x}\right)^{2}
  &\leq \frac{1}{n^{2}} \SumOver{x}\left(\SumFromTo{i}{1}{n}
     \Big|\mu_{i+k}(x) - \pi(x)\Big|\right)^{2} \\
  &\leq \frac{1}{n^{2}}\left(\SumOver{x} \SumFromTo{i}{1}{n}
     \Big|\mu_{i+k}(x) - \pi(x)\Big|\right)^{2} \\
  &= \frac{1}{n^{2}}\left(\SumFromTo{i}{1}{n}
     \NormTV{\mu_{i+k} - \pi}\right)^{2}.
\end{aligned}
\]
Since $\KernelMatrix$ is uniformly ergodic, there exist $C > 0$ and $\rho \in [0,1[$ such that, for any probability distribution $\mu$ on $\StateSpace$ and any $i \in \Naturals$, $\NormTV{\mu \KernelMatrix[i] - \pi} \leq C \rho^{i}$.
Hence,
\[
  \SumOver{x} \Bias\left(\PiHatAt{k}{x}\right)^{2}
  \leq \frac{1}{n^{2}}\left(\SumFromTo{i}{1}{n} C \rho^{i+k-1}\right)^{2}
  = \frac{C^{2}}{n^{2}} \rho^{2k} \frac{(1-\rho^{n})^{2}}{(1-\rho)^{2}}
  = \BigO\!\left(\frac{1}{n^{2}}\right),
\]
and the result is uniform in $\mu$, as noticed in Remark~\ref{rem:uniformity-initial-distribution}.

\subsection{Proof of Proposition~\ref{prop:variance-nonstationary-markov}}
\label{app:proof-variance-nonstationary-markov}

For $i \in \Naturals^{*}$, let $\mu_i = \mu \KernelMatrix[i-1]$.
We start by noting that
\[
  \Variance\left(\PiHatAt{k}{x}\right)
  = \frac{1}{n^{2}} \SumFromTo{i}{1}{n} \Variance\left(\Contribution{k}{i}{x}\right)
  + \frac{2}{n^{2}} \sum_{1 \leq i < j \leq n}
    \Covariance\left(\Contribution{k}{i}{x}, \Contribution{k}{j}{x}\right).
\]
Hence, we compute all the relevant quantities.
For $x \in \StateSpace$ and $i,j \in \Naturals$ with $1 \leq i < j$,
\[
  \Expect\left[\Contribution{k}{i}{x}\right]
  = \SumOver{y} \mu_i(y) \Kernel[k]{x}{y}
  = \mu_{i+k}(x),
\]
\[
  \Expect\left[\Contribution{k}{i}{x}^{2}\right]
  = \SumOver{y} \mu_i(y) \Kernel[k]{x}{y}^{2},
\]
and
\[
\begin{aligned}
  \Expect\left[\Contribution{k}{i}{x}\Contribution{k}{j}{x}\right]
  &= \SumOver{y,z} \Prob(X_i = y, X_j = z) \Kernel[k]{x}{y}\Kernel[k]{x}{z} \\
  &= \SumOver{y,z} \mu_i(y) \Kernel[j-i]{z}{y}
     \Kernel[k]{x}{y}\Kernel[k]{x}{z} \\
  &= \SumOver{y} \mu_i(y)
     \Kernel[j-i+k]{x}{y} \Kernel[k]{x}{y}.
\end{aligned}
\]
Now, we compare these values to the ones obtained in the stationary framework.
Since $\KernelMatrix$ is uniformly ergodic, we first introduce $C > 0$ and $\rho \in [0,1[$ such that, for any probability distribution $\mu$ on $\StateSpace$ and any $i \in \Naturals$, $\NormTV{\mu \KernelMatrix[i] - \pi} \leq C \rho^{i}$.
We compute
\[
\begin{aligned}
  \SumOver{x}\left|\frac{1}{n} \SumFromTo{i}{1}{n}
    \mu_{i+k}(x)^{2} - \pi(x)^{2}\right|
  &\leq \SumOver{x} \frac{1}{n} \SumFromTo{i}{1}{n}
     \left|\mu_{i+k}(x)^{2} - \pi(x)^{2}\right| \\
  &\leq 2 \SumOver{x} \frac{1}{n} \SumFromTo{i}{1}{n}
     \left|\mu_{i+k}(x) - \pi(x)\right| \\
  &= \frac{2}{n} \SumFromTo{i}{1}{n} \NormTV{\mu_{i+k} - \pi}
   \leq \frac{2 C \rho^{k}}{1-\rho}\, \frac{1}{n},
\end{aligned}
\]
and
\[
\begin{aligned}
  &\SumOver{x}\left|\frac{1}{n} \SumFromTo{i}{1}{n} \SumOver{y}
    \mu_i(y)\Kernel[k]{x}{y}^{2}
    - \SumOver{y} \pi(y)\Kernel[k]{x}{y}^{2}\right| \\
  &\qquad \leq \SumOver{x} \frac{1}{n} \SumFromTo{i}{1}{n} \SumOver{y}
     \left|\mu_i(y) - \pi(y)\right| \Kernel[k]{x}{y}^{2} \\
  &\qquad \leq \SumOver{x} \frac{1}{n} \SumFromTo{i}{1}{n} \SumOver{y}
     \left|\mu_i(y) - \pi(y)\right| \Kernel[k]{x}{y} \\
  &\qquad = \frac{1}{n} \SumFromTo{i}{1}{n} \SumOver{y}
     \left|\mu_i(y) - \pi(y)\right|
   \leq \frac{C}{n(1-\rho)}.
\end{aligned}
\]
This guarantees that
\[
  \frac{1}{n^{2}} \SumOver{x} \SumFromTo{i}{1}{n}
  \Variance\left(\Contribution{k}{i}{x}\right)
  = \frac{1}{n}\left(\SumOver{x}\left\{\SumOver{y} \pi(y)\Kernel[k]{x}{y}^{2}
    - \pi(x)^{2}\right\}\right) + \LittleO\!\left(\frac{1}{n}\right).
\]
For the covariance contribution, rather than directly comparing it with its asymptotic value, we compare it with itself under the assumption $\mu = \pi$.
In other words, for $1 \leq i < j \leq n$ and $x \in \StateSpace$, we want to bound
\[
  \left|\Covariance\left(\Contribution{k}{i}{x}, \Contribution{k}{j}{x}\right)
  - \SumOver{y} \pi(y)\Kernel[j-i+k]{x}{y}\Kernel[k]{x}{y} + \pi(x)^{2}\right| = \left| \Delta_1(i,j,x) - \Delta_2(i,j,x)\right|,
\]
where
\[
\begin{aligned}
  \Delta_1(i,j,x) &= \SumOver{y}\left\{\mu_i(y) - \pi(y)\right\}
    \Kernel[k]{x}{y}\left\{\Kernel[j-i+k]{x}{y} - \pi(x)\right\} ,\\
  \Delta_2(i,j,x)&=\SumOver{y} \mu_i(y) \Kernel[k]{x}{y}
    \left\{\mu_{j+k}(x) - \pi(x)\right\}.
\end{aligned}
\]
For the first error term, we have
\[
  \left|\Delta_1(i,j,x)\right|
  \leq \SumOver{y}\left|\mu_i(y) - \pi(y)\right|
    \times \left|\Kernel[j-i+k]{x}{y} - \pi(x)\right|,
\]
then, summing over $x \in \StateSpace$,
\[
\begin{aligned}
  \SumOver{x}\left|\Delta_1(i,j,x)\right|
  &\leq \SumOver{y}\left|\mu_i(y) - \pi(y)\right|
     \times \SumOver{x}\left|\Kernel[j-i+k]{x}{y} - \pi(x)\right| \\
  &= \SumOver{y}\left|\mu_i(y) - \pi(y)\right|
     \times \NormTV{\Dirac{y}\KernelMatrix[j-i+k] - \pi} \\
  &\leq C \rho^{j-i+k} \NormTV{\mu_i - \pi}
   \leq C^{2} \rho^{j+k-1}.
\end{aligned}
\]
Finally, summing over $1 \leq i < j \leq n$,
\[
  \sum_{1 \leq i < j \leq n} \SumOver{x}\left|\Delta_1(i,j,x)\right|
  \leq C^{2} \rho^{k} \sum_{1 \leq i < j \leq n} \rho^{j-1}
  = C^{2} \rho^{k} \SumFromTo{j}{1}{n} (j-1)\rho^{j-1}
  \leq C^{2} \frac{\rho^{k+1}}{(1-\rho)^{2}}.
\]
For the second error term, we have
\[
  \left|\Delta_2(i,j,x)\right| \leq \SumOver{y} \mu_i(y) \Kernel[k]{x}{y}
     \left|\mu_{j+k}(x) - \pi(x)\right| \leq \left|\mu_{j+k}(x) - \pi(x)\right|,
\]
then, summing over $x \in \StateSpace$,
\[
  \SumOver{x}\left|\Delta_2(i,j,x)\right|
  \leq \NormTV{\mu_{j+k} - \pi} \leq C \rho^{j+k-1},
\]
and, summing over $1 \leq i < j \leq n$,
\[
  \sum_{1 \leq i < j \leq n} \SumOver{x}\left|\Delta_2(i,j,x)\right|
  \leq C \frac{\rho^{k+1}}{(1-\rho)^{2}}.
\]
From these bounds on the error, we deduce that
\[
\begin{aligned}
  &\frac{2}{n^{2}} \SumOver{x} \sum_{1 \leq i < j \leq n}
   \Covariance\left(\Contribution{k}{i}{x}, \Contribution{k}{j}{x}\right) \\
  &\qquad = \frac{2}{n^{2}} \SumOver{x} \sum_{1 \leq i < j \leq n}
   \left(\SumOver{y} \pi(y)\Kernel[j-i+k]{x}{y}\Kernel[k]{x}{y} - \pi(x)^{2}\right)
   + \LittleO\!\left(\frac{1}{n}\right),
\end{aligned}
\]
and, from Lemma~\ref{lem:summability}, that
\[
\begin{aligned}
  &\frac{2}{n^{2}} \SumOver{x} \sum_{1 \leq i < j \leq n}
   \left(\SumOver{y} \pi(y)\Kernel[j-i+k]{x}{y}\Kernel[k]{x}{y} - \pi(x)^{2}\right) \\
  &\qquad = \frac{2}{n} \SumOver{x} \SumFromTo{i}{1}{+\infty}
   \left(\SumOver{y} \pi(y)\Kernel[i+k]{x}{y}\Kernel[k]{x}{y} - \pi(x)^{2}\right)
  + \LittleO\!\left(\frac{1}{n}\right).
\end{aligned}
\]
Finally, recombining everything together yields the expected result, which is uniform in the initial distribution as noticed in Remark~\ref{rem:uniformity-initial-distribution}.

\subsection{Proof of Proposition~\ref{prop:invertibility-sigma}}
\label{app:proof-inversibility}

To prove this result, we first establish the following lemma.

\begin{lemma}
\label{lem:span-empirical}
Let $F$ be the vector subspace of~$\Reals^{\StateSpace}$ composed of the finite support vectors of~$\Reals^{\StateSpace}$, and for $n \in \Naturals^{*}$, let
\[
  E_n = \left\{ \EmpiricalAt[n]{A} \,:\, A \in \StateSpace^{n},\
  \Prob\left((X_1, \dots, X_n) = A\right) > 0 \right\},
\]
where, for $A = (a_1, \dots, a_n) \in \StateSpace^{n}$, $\EmpiricalAt[n]{A} = n^{-1} \SumFromTo{i}{1}{n} \Dirac{a_i}$ denotes the empirical measure of the sample $A$.
Then $\Span(E_n) = F$.
\end{lemma}

\begin{proof}
For $x \in \StateSpace$, let $e_x = (\Dirac{x}(y))_{y \in \StateSpace}$.
Since every empirical measure has finite support, $E_n \subset F$, hence $\Span(E_n) \subset F$, and it suffices to show that $e_x \in \Span(E_n)$ for every $x \in \StateSpace$.
\begin{itemize}
\item \textit{Independent and identically distributed observations.}\\
  Let $x \in \StateSpace$.
  We have
\[
  \Prob\left((X_1, \dots, X_n) = (x, \dots, x)\right) = \pi(x)^{n} > 0,
\]
and $\EmpiricalAt[n]{(x, \dots, x)} = e_x$.
Therefore, $\Span(E_n) = F$.
\item \textit{Stationary Markov chain.}\\
  Since $\KernelMatrix$ is irreducible and positive recurrent, $\pi(x) > 0$ for every $x \in \StateSpace$, so that, the chain being stationary, any sequence $(x_1, \dots, x_n) \in \StateSpace^{n}$ such that $\Kernel{x_{i+1}}{x_i} > 0$ for all $\IndexBetween{i}{1}{n-1}$ satisfies $\Prob((X_1, \dots, X_n) = (x_1, \dots, x_n)) > 0$.
  Let $(x_1, \dots, x_{n+1}) \in \StateSpace^{n+1}$ be a chain of length $n+1$ such that for all $\IndexBetween{i}{1}{n}$, $\Kernel{x_{i+1}}{x_i} > 0$.
  We define $C = (x_1, \dots, x_n)$ and $C' = (x_2, \dots, x_{n+1})$, which both belong to the set of samples defining $E_n$.
  We have
\[
  \EmpiricalAt[n]{C} = \frac{1}{n} \SumFromTo{i}{1}{n} e_{x_i}
  \quad \text{and} \quad
  \EmpiricalAt[n]{C'} = \frac{1}{n} \SumFromTo{i}{2}{n+1} e_{x_i},
\]
hence
\[
  n \EmpiricalAt[n]{C} - n \EmpiricalAt[n]{C'} = e_{x_1} - e_{x_{n+1}}.
\]
Since $\KernelMatrix$ is irreducible and aperiodic, $\KernelMatrix[n+1]$ is irreducible.
Hence, for all $x, y \in \StateSpace$ with $x \neq y$, there exists a path $x = w_1, \dots, w_q = y$ of length $q \geq 2$ such that for all $\IndexBetween{i}{1}{q-1}$, $\Kernel[n+1]{w_{i+1}}{w_i} > 0$.
Thus, for each $\IndexBetween{i}{1}{q-1}$, there exists a chain $(z_1^i, \dots, z_{n+1}^i) \in \StateSpace^{n+1}$ of length $n+1$ such that $z_1^i = w_i$, $z_{n+1}^i = w_{i+1}$, and for all $\IndexBetween{j}{1}{n}$, $\Kernel{z_{j+1}^i}{z_j^i} > 0$.
Setting $C_i = (z_1^i, \dots, z_n^i)$ and $C'_i = (z_2^i, \dots, z_{n+1}^i)$, the previous identity yields
\[
  n \EmpiricalAt[n]{C_i} - n \EmpiricalAt[n]{C'_i}
  = e_{z_1^i} - e_{z_{n+1}^i} = e_{w_i} - e_{w_{i+1}}.
\]
Summing these differences gives
\[
  \SumFromTo{i}{1}{q-1}\left[ n \EmpiricalAt[n]{C_i} - n \EmpiricalAt[n]{C'_i} \right]
  = \SumFromTo{i}{1}{q-1} \left(e_{w_i} - e_{w_{i+1}}\right)
  = e_{w_1} - e_{w_q}
  = \Dirac{x} - \Dirac{y}.
\]
Therefore, for all $x, y \in \StateSpace$ with $x \neq y$, $e_x - e_y \in \Span(E_n)$.

On the other hand, for any $x \in \StateSpace$, consider a path $(x, x_2, \dots, x_n) \in \StateSpace^{n}$ such that $\Prob((X_1, \dots, X_n) = (x, x_2, \dots, x_n)) > 0$.
We write
\[
\begin{aligned}
  n \EmpiricalAt[n]{(x, x_2, \dots, x_n)}
  &= e_x + \SumFromTo{i}{2}{n} e_{x_i}
   = e_x + \SumFromTo{i}{2}{n} \left[ e_x + (e_{x_i} - e_x) \right] \\
  &= n e_x + \SumFromTo{i}{2}{n} (e_{x_i} - e_x).
\end{aligned}
\]
Since $n \EmpiricalAt[n]{(x, x_2, \dots, x_n)} \in \Span(E_n)$ and $\SumFromTo{i}{2}{n} (e_{x_i} - e_x) \in \Span(E_n)$, it follows that $e_x \in \Span(E_n)$ for all $x \in \StateSpace$.
The vector space spanned by $(e_x)_{x \in \StateSpace}$ is precisely $F$, consequently, $F = \Span(E_n)$.
\end{itemize}
\end{proof}

Coming back to the proof of Proposition~\ref{prop:invertibility-sigma}, $\Sigma(f)$ is degenerate if and only if there exists $v \in \Reals^{I}$, $v \neq 0$, such that $v\Transpose \Sigma(f) v = 0$.
However,
\[
  v\Transpose \Sigma(f) v = \Variance\left(v\Transpose \PiHatOf{I}{f}\right),
\]
where $\PiHatOf{I}{f} = (\PiHatOf{k}{f})_{k \in I}$.
Hence, $\Sigma(f)$ is degenerate if and only if $v\Transpose \PiHatOf{I}{f}$ is almost surely constant for some $v \in \Reals^{I}$, $v \neq 0$.
Assume that this is the case, and let $C \in \Reals$ denote this constant, so that $C = v\Transpose \PiHatOf{I}{f}$ almost surely.
Then, on one hand
\[
  v\Transpose \PiHatOf{I}{f}
  = \sum_{k \in I} v_k \PiHatOf{k}{f}
  = \sum_{k \in I} v_k \EmpiricalOf{\KernelMatrix[k] f}
  = \EmpiricalOf{\sum_{k \in I} v_k \KernelMatrix[k] f},
\]
and on the other hand,
\[
  C = \EmpiricalOf{C \Ones}.
\]
Hence, almost surely,
\[
  \EmpiricalOf{\sum_{k \in I} v_k \KernelMatrix[k] f - C \Ones} = 0,
\]
that is, $\EmpiricalOf{g} = \SumOver{x} \EmpiricalAt{x} g(x) = 0$ almost surely, where $g = \sum_{k \in I} v_k \KernelMatrix[k] f - C \Ones$.
In other words, $\langle \mu, g \rangle = \SumOver{x} \mu(x) g(x) = 0$ for every $\mu \in E_n$, hence, by linearity, for every $\mu \in \Span(E_n)$.
Using Lemma~\ref{lem:span-empirical}, we conclude that $g \in F^{\perp}$, the orthogonal of $F$ for $\langle \cdot, \cdot \rangle$.
However, $F^{\perp} = \{\mathbf{0}_{\Reals^{\StateSpace}}\}$, since $\langle e_x, g \rangle = g(x)$ for every $x \in \StateSpace$, thus
\[
  \sum_{k \in I} v_k \KernelMatrix[k] f - C \Ones = 0.
\]
Since $v \neq 0$, the family $\{\Ones\} \cup \{\KernelMatrix[k] f\}_{k \in I}$ is linearly dependent in $\Reals^{\StateSpace}$.

Conversely, assume that this family is linearly dependent, that is, $\sum_{k \in I} v_k \KernelMatrix[k] f - C \Ones = 0$ for some $(v, C) \in \Reals^{I} \times \Reals$, $(v, C) \neq 0$.
Then $v \neq 0$, since otherwise $C \Ones = 0$, hence $C = 0$.
Moreover, by the computation above,
\[
  v\Transpose \PiHatOf{I}{f}
  = \EmpiricalOf{\sum_{k \in I} v_k \KernelMatrix[k] f}
  = \EmpiricalOf{C \Ones} = C,
\]
so that $v\Transpose \Sigma(f) v = \Variance(v\Transpose \PiHatOf{I}{f}) = 0$, and $\Sigma(f)$ is degenerate, which concludes the proof.

\subsection{Proof of Proposition~\ref{prop:covariance-iid}}
\label{app:proof-covariance-iid}

Independently of the data framework, we have
\begin{align}
  \Covariance\left(\PiHatOf{k}{f}, \PiHatOf{l}{f}\right)
  &= \Expect\left[\PiHatOf{k}{f}\PiHatOf{l}{f}\right]
   - \Expect\left[\PiHatOf{k}{f}\right]\Expect\left[\PiHatOf{l}{f}\right] \nonumber \\
  &= \frac{1}{n^{2}} \SumFromTo{i}{1}{n} \SumFromTo{j}{1}{n}
     \Expect\left[\Contribution{k}{i}{f}\Contribution{l}{j}{f}\right]
   - \pi(f)^2.
  \label{eq:covariance-double-sum}
\end{align}
Consequently, we need to compute $\Expect[\Contribution{k}{i}{f}\Contribution{l}{j}{f}]$ for $\IndexBetween{i,j}{1}{n}$.
In the independent setting,
\[
  \Expect\left[\Contribution{k}{i}{f}\Contribution{l}{j}{f}\right]
  = \begin{cases}
      \SumOver{y} \pi(y)\KernelAverage{k}{y}{f}\KernelAverage{l}{y}{f}
      & \text{if } i = j, \\[1ex]
      \pi(f)^2 & \text{otherwise.}
    \end{cases}
\]
Hence,
\[
\begin{aligned}
  \Covariance\left(\PiHatOf{k}{f}, \PiHatOf{l}{f}\right)
  &= \frac{1}{n^{2}}\left\{n \SumOver{y} \pi(y)\KernelAverage{k}{y}{f}\KernelAverage{l}{y}{f}
     + n(n-1)\pi(f)^2\right\} - \pi(f)^2 \\
  &= \frac{1}{n}\left\{\SumOver{y} \pi(y)\KernelAverage{k}{y}{f}\KernelAverage{l}{y}{f}
     - \pi(f)^2\right\}.
\end{aligned}
\]

\subsection{Proof of Proposition~\ref{prop:covariance-markov}}
\label{app:proof-covariance-markov}

In the Markov setting, reusing~\eqref{eq:covariance-double-sum} and grouping the pairs $(i,j)$ according to the lag $|j-i|$, together with Lemma~\ref{lem:covariance-contributions}, we have
\[
\begin{aligned}
  &n^2\Covariance\left(\PiHatOf{k}{f}, \PiHatOf{l}{f}\right)
  = \SumFromTo{i}{1}{n} \SumFromTo{j}{1}{n}
     \Covariance\left(\Contribution{k}{i}{f}, \Contribution{l}{j}{f}\right) \\
  &\enspace = \SumFromTo{i}{1}{n}
     \Covariance\left(\Contribution{k}{i}{f}, \Contribution{l}{i}{f}\right)
     + \SumFromTo{i}{1}{n-1} \SumFromTo{j}{i+1}{n}
     \Covariance\left(\Contribution{k}{i}{f}, \Contribution{l}{j}{f}\right)
     + \SumFromTo{j}{1}{n-1} \SumFromTo{i}{j+1}{n}
     \Covariance\left(\Contribution{k}{i}{f}, \Contribution{l}{j}{f}\right) \\
  &\enspace = n \left\{\SumOver{y} \pi(y)\KernelAverage{k}{y}{f}\KernelAverage{l}{y}{f}
     - \pi(f)^{2}\right\} \\
  &\qquad + \SumFromTo{i}{1}{n-1}(n-i) \left\{\SumOver{y,z} \pi(y)\Kernel[i]{z}{y}
     \KernelAverage{k}{y}{f}\KernelAverage{l}{z}{f} - \pi(f)^{2}\right\} \\
  &\qquad + \SumFromTo{i}{1}{n-1}(n-i) \left\{\SumOver{y,z} \pi(y)\Kernel[i]{z}{y}
     \KernelAverage{l}{y}{f}\KernelAverage{k}{z}{f} - \pi(f)^{2}\right\}.
\end{aligned}
\]
This yields,
\[
\begin{aligned}
  &\Covariance\left(\PiHatOf{k}{f}, \PiHatOf{l}{f}\right)
  = \frac{1}{n} \left\{\SumOver{y} \pi(y)\KernelAverage{k}{y}{f}\KernelAverage{l}{y}{f}
     - \pi(f)^{2}\right\} \\
  &\qquad + \frac{1}{n^{2}} \SumFromTo{i}{1}{n-1}(n-i)
     \left(\SumOver{y,z} \pi(y)\Kernel[i]{z}{y}
     \left\{\KernelAverage{k}{y}{f}\KernelAverage{l}{z}{f}
     + \KernelAverage{l}{y}{f}\KernelAverage{k}{z}{f}\right\}
     - 2\pi(f)^{2}\right),
\end{aligned}
\]
stating the expected formula.

\section{Proofs of Section~\ref{sec:theory-parametric-kernel}}
\label{app:parametric-kernel}

\subsection{Proof of Proposition~\ref{prop:regularity-finite}}
\label{app:proof-regularity-finite}

\begin{itemize}
\item \textit{Condition~\ref{ass:regularity1}.}\\
Since $\theta \mapsto \NormOperator{\KernelMatrixParam[k]{\theta}}_{2}$ is continuous and $\ParamSpace$ is compact, it is uniformly bounded and, since
\[
  \NormOperator{\cdot}_{2,\text{res}} \leq \NormOperator{\cdot}_{2},
\]
the mapping $\theta \mapsto \NormOperator{\KernelMatrixParam[k]{\theta}}_{2,\text{res}}$ is also uniformly bounded.
Moreover, as $\theta \mapsto \KernelMatrixParam[k]{\theta}$ is of class $\mathcal{C}^{2}$, it is of class $\mathcal{C}^{1}$.
Hence, since $\ParamSpace$ is convex, the mean value inequality yields
\[
  \NormOperator{\KernelMatrixParam[k]{\theta}-\KernelMatrixTrue[k]}_{2,\text{res}}
  \leq \NormOperator{\KernelMatrixParam[k]{\theta}-\KernelMatrixTrue[k]}_{2}
  \leq [\KernelMatrix[k]]_{\text{Lip}} \NormTwo{\theta-\ParamTrue},
\]
where $[\KernelMatrix[k]]_{\text{Lip}} = \sup_{t \in \ParamSpace} \NormTwo{\Gradient \KernelMatrixParam[k]{t}} < +\infty$, since $\theta \mapsto \Gradient \KernelMatrixParam[k]{\theta}$ is continuous and $\ParamSpace$ is compact.

\item \textit{Condition~\ref{ass:regularity2}.}\\
First, for any $x \in \StateSpace$,
\[
  \SumOver{y} \pi(y) \NormTwo{\Gradient \KernelParam[k]{\ParamTrue}{x}{y}} < +\infty
\]
as a finite sum of finite terms.
Then, we write
\[
  \NormTwo{\pi\left(\KernelMatrixParam[k]{\theta}-\KernelMatrixTrue[k]\right)}^{2}
  = \SumOver{x} \left(\SumOver{y} \pi(y)
    \Big[\KernelParam[k]{\theta}{x}{y}-\KernelParam[k]{\ParamTrue}{x}{y}\Big]\right)^{2}.
\]
Thanks to a second-order Taylor expansion, we have
\[
\begin{aligned}
  &\NormTwo{\pi\left(\KernelMatrixParam[k]{\theta}-\KernelMatrixTrue[k]\right)}^{2} \\
  &\qquad = \SumOver{x} \left(\SumOver{y} \pi(y)\Big[\Gradient \KernelParam[k]{\ParamTrue}{x}{y}\Transpose
     (\theta-\ParamTrue) + \BigO\!\left(\NormTwo{\theta-\ParamTrue}^{2}\right)\Big]\right)^{2} \\
  &\qquad = \SumOver{x} \left(\Gradient\left(\pi \KernelMatrixTrue[k]\right)\!\left(x\right)\Transpose
     (\theta-\ParamTrue) + \BigO\!\left(\NormTwo{\theta-\ParamTrue}^{2}\right)\right)^{2} \\
  &\qquad = \SumOver{x} \Gradient\left(\pi \KernelMatrixTrue[k]\right)\!\left(x\right)\Transpose
     (\theta-\ParamTrue)(\theta-\ParamTrue)\Transpose
     \Gradient\left(\pi \KernelMatrixTrue[k]\right)\!\left(x\right)
     + \BigO\!\left(\NormTwo{\theta-\ParamTrue}^{3}\right),
\end{aligned}
\]
which is equal to the expected formula.
\end{itemize}

\subsection{Proof of Theorem~\ref{thm:mse-expansion}}
\label{app:proof-mse-expansion}

For any $k \in \Naturals$, we decompose the estimation error by introducing the true transition kernel $\KernelMatrixTrue[k]$ and its plug-in estimator $\KernelMatrixPlugIn[k]$,
\[
\begin{aligned}
  \PiHatParam[n-m]{\ParamEstimator}{k} - \pi
  &= \Empirical[n-m] \KernelMatrixPlugIn[k] - \pi \\
  &= \Empirical[n-m] \KernelMatrixPlugIn[k] - \pi \KernelMatrixPlugIn[k]
   + \pi \KernelMatrixPlugIn[k] - \pi \KernelMatrixTrue[k].
\end{aligned}
\]
Taking the squared norm on both sides yields the expansion
\[
\begin{aligned}
  \NormTwo{\PiHatParam[n-m]{\ParamEstimator}{k} - \pi}^{2}
  &= \NormTwo{\left(\Empirical[n-m] - \pi\right)\KernelMatrixPlugIn[k]}^{2}
     + \NormTwo{\pi\left(\KernelMatrixPlugIn[k] - \KernelMatrixTrue[k]\right)}^{2} \\
  &\quad + 2 \left\langle \left(\Empirical[n-m] - \pi\right)\KernelMatrixPlugIn[k],
     \pi\left(\KernelMatrixPlugIn[k] - \KernelMatrixTrue[k]\right)\right\rangle.
\end{aligned}
\]
To evaluate the MSE, we first control the expectation of the cross-product term.
We condition on the first part of the trajectory $X_1,\dots,X_m$ and note that $\ParamEstimator$ is measurable with respect to the associated $\sigma$-algebra.
We obtain
\[
\begin{aligned}
  &\Expect\left[\left\langle \left(\Empirical[n-m] - \pi\right)\KernelMatrixPlugIn[k],
    \pi\left(\KernelMatrixPlugIn[k] - \KernelMatrixTrue[k]\right)\right\rangle\right] \\
  &\qquad = \Expect\left[\Expect\left[\left\langle
    \left(\Empirical[n-m] - \pi\right)\KernelMatrixPlugIn[k],
    \pi\left(\KernelMatrixPlugIn[k] - \KernelMatrixTrue[k]\right)\right\rangle
    \,\Big|\, X_1,\dots,X_m\right]\right] \\
  &\qquad = \Expect\left[\left\langle
    \Expect\left[\Empirical[n-m] - \pi \,\Big|\, X_1,\dots,X_m\right]\KernelMatrixPlugIn[k],
    \pi\left(\KernelMatrixPlugIn[k] - \KernelMatrixTrue[k]\right)\right\rangle\right].
\end{aligned}
\]
Invoking the Markov property, the conditional expectation simplifies to
\[
  \Expect\left[\Empirical[n-m] - \pi \mid X_1,\dots,X_m\right]
  = \Expect\left[\Empirical[n-m] - \pi \mid X_m\right].
\]
Then, applying the Cauchy--Schwarz inequality, we have
\[
\begin{aligned}
  &\left|\left\langle \Expect\left[\Empirical[n-m] - \pi \,\Big|\, X_m\right]\KernelMatrixPlugIn[k],
   \pi\left(\KernelMatrixPlugIn[k] - \KernelMatrixTrue[k]\right)\right\rangle \right|\\
  &\qquad \leq \NormTwo{\Expect\left[\Empirical[n-m] - \pi \mid X_m\right]\KernelMatrixPlugIn[k]}
   \times \NormTwo{\pi\left(\KernelMatrixPlugIn[k] - \KernelMatrixTrue[k]\right)}.
\end{aligned}
\]
Condition \ref{ass:regularity1} ensures that there exists $L > 0$ such that $\NormOperator{\KernelMatrixPlugIn[k]}_{2,\text{res}} \leq L$.
Hence, this upper bound becomes
\[
  L \, \NormTwo{\Expect\left[\Empirical[n-m] - \pi \mid X_m\right]}
  \times \NormTwo{\pi\left(\KernelMatrixPlugIn[k] - \KernelMatrixTrue[k]\right)}.
\]
By rewriting the conditional expectation term through the integrated conditional bias, we have
\[
  \NormTwo{\Expect\left[\Empirical[n-m] - \pi \mid X_m\right]}
  = \sqrt{\SumOver{x} \Bias^{2}\!\left(\Empirical[n-m] \mid X_m\right)}.
\]
Conditionally on $X_m$, the data $(X_i)_{\IndexBetween{i}{m+1}{n}}$ form a Markov chain with kernel $\KernelMatrixTrue$ and initial distribution $\KernelParam{\ParamTrue}{\cdot}{X_m}$.
Hence, Proposition~\ref{prop:bias-negligible} (with $k=0$) applies to $\Empirical[n-m]$ conditionally on $X_m$.
Since this result holds uniformly in the initial distribution (Remark~\ref{rem:uniformity-initial-distribution}), there exists $C > 0$, which depends neither on $X_m$ nor on $n-m$, such that the right-hand side above is bounded by $C/(n-m)$.
Therefore,
\[
  \left|\left\langle \Expect\left[\Empirical[n-m] - \pi \,\Big|\, X_m\right]\KernelMatrixPlugIn[k],
  \pi\left(\KernelMatrixPlugIn[k] - \KernelMatrixTrue[k]\right)\right\rangle \right|
  \leq L C \times \frac{\NormTwo{\pi\left(\KernelMatrixPlugIn[k] - \KernelMatrixTrue[k]\right)}}{n-m}.
\]
Taking the total expectation leads to the following bound for the first cross-term,
\[
  \Expect\left[\left\langle \left(\Empirical[n-m] - \pi\right)\KernelMatrixPlugIn[k],
  \pi\left(\KernelMatrixPlugIn[k] - \KernelMatrixTrue[k]\right)\right\rangle\right]
  = \BigO\!\left(\frac{\Expect\left[\NormTwo{\pi\left(\KernelMatrixPlugIn[k] - \KernelMatrixTrue[k]\right)}\right]}{n-m}\right).
\]
Under Condition~\ref{ass:regularity2}, we directly derive that
\[
\begin{aligned}
  \MSE\left(\PiHatParam[n-m]{\ParamEstimator}{k}\right)
  &= \Expect\left[\NormTwo{\left(\Empirical[n-m] - \pi\right)\KernelMatrixPlugIn[k]}^{2}\right] \\
  &\quad + \Trace\left(\Gradient\left(\pi \KernelMatrixTrue[k]\right)\Transpose
     \Expect\left[\left(\EstimationError\right)\left(\EstimationError\right)\Transpose\right]
     \Gradient\left(\pi \KernelMatrixTrue[k]\right)\right) \\
  &+ \BigO\!\left(\Expect\left[\NormTwo{\EstimationError}^{3}\right]\right)
   + \BigO\!\left(\frac{\Expect\left[\NormTwo{\EstimationError}\right]}{n-m}\right).
\end{aligned}
\]
We then refine the first term by expanding around the true transition kernel $\KernelMatrixTrue[k]$,
\[
  \left(\Empirical[n-m] - \pi\right)\KernelMatrixPlugIn[k]
  = \left(\Empirical[n-m] - \pi\right)\KernelMatrixTrue[k]
  + \left(\Empirical[n-m] - \pi\right)\left(\KernelMatrixPlugIn[k] - \KernelMatrixTrue[k]\right),
\]
which leads to the identity
\[
\begin{aligned}
  \NormTwo{\left(\Empirical[n-m] - \pi\right)\KernelMatrixPlugIn[k]}^{2}
  &= \NormTwo{\PiHatParam[n-m]{\ParamTrue}{k} - \pi}^{2}
   + \NormTwo{\left(\Empirical[n-m] - \pi\right)\left(\KernelMatrixPlugIn[k] - \KernelMatrixTrue[k]\right)}^{2} \\
  &+ 2 \left\langle \PiHatParam[n-m]{\ParamTrue}{k} - \pi,
     \left(\Empirical[n-m] - \pi\right)\left(\KernelMatrixPlugIn[k] - \KernelMatrixTrue[k]\right)\right\rangle.
\end{aligned}
\]
Invoking the Lipschitz continuity of the mapping $\theta \mapsto \KernelMatrixParam[k]{\theta}$ in Condition \ref{ass:regularity1}, the second term is bounded by
\[
\begin{aligned}
  \NormTwo{\left(\Empirical[n-m] - \pi\right)\left(\KernelMatrixPlugIn[k] - \KernelMatrixTrue[k]\right)}
  &\leq \NormTwo{\Empirical[n-m] - \pi}
        \times \NormOperator{\KernelMatrixPlugIn[k] - \KernelMatrixTrue[k]}_{2,\text{res}} \\
  &\leq [\KernelMatrix[k]]_{\text{Lip}} \times \NormTwo{\Empirical[n-m] - \pi} \times \NormTwo{\EstimationError}.
\end{aligned}
\]
Taking expectations, it directly follows that
\[
  \Expect\left[\NormTwo{\left(\Empirical[n-m] - \pi\right)\left(\KernelMatrixPlugIn[k] - \KernelMatrixTrue[k]\right)}^{2}\right]
  = \BigO\!\left(\frac{\Expect\left[\NormTwo{\EstimationError}^{2}\right]}{n-m}\right).
\]
For the remaining cross-product term, we apply the Cauchy--Schwarz inequality,
\[
\begin{aligned}
  &\Expect\left[\left|\left\langle \PiHatParam[n-m]{\ParamTrue}{k} - \pi,
   \left(\Empirical[n-m] - \pi\right)\left(\KernelMatrixPlugIn[k] - \KernelMatrixTrue[k]\right)\right\rangle\right|\right] \\
  &\qquad \leq \sqrt{\Expect\left[\NormTwo{\PiHatParam[n-m]{\ParamTrue}{k} - \pi}^{2}\right]
     \times \Expect\left[\NormTwo{\left(\Empirical[n-m] - \pi\right)\left(\KernelMatrixPlugIn[k] - \KernelMatrixTrue[k]\right)}^{2}\right]} \\
  &\qquad = \sqrt{\BigO\!\left(\frac{1}{n-m}\right)
     \times \BigO\!\left(\frac{\Expect\left[\NormTwo{\EstimationError}^{2}\right]}{n-m}\right)} \\
  &\qquad = \BigO\!\left(\frac{\Expect\left[\NormTwo{\EstimationError}^{2}\right]^{1/2}}{n-m}\right).
\end{aligned}
\]
Finally, gathering all pieces together, we arrive at the full expression for the MSE,
\[
\begin{aligned}
  &\MSE\left(\PiHatParam[n-m]{\ParamEstimator}{k}\right)
  = \MSE\left(\PiHatParam[n-m]{\ParamTrue}{k}\right) \\
  &\enspace + \Trace\left(\Gradient\left(\pi \KernelMatrixTrue[k]\right)\Transpose
     \Expect\left[\left(\EstimationError\right)\left(\EstimationError\right)\Transpose\right]
     \Gradient\left(\pi \KernelMatrixTrue[k]\right)\right) \\
  &\enspace + \BigO\!\left(\Expect\left[\NormTwo{\EstimationError}^{3}\right]\right)
   + \BigO\!\left(\frac{\Expect\left[\NormTwo{\EstimationError}^{2}\right]^{1/2}}{n-m}\right)
   + \BigO\!\left(\frac{\Expect\left[\NormTwo{\EstimationError}^{2}\right]}{n-m}\right).
\end{aligned}
\]

\subsection{Proof of Corollary~\ref{cor:data-partition1}}
\label{app:proof-data-partition1}

We assume that $m^{-1} = \LittleO(g(m))$.
We want to show that there exists no split $(m, n-m)$ such that $\PiHatParam[n-m]{\ParamEstimator}{k}$ is asymptotically a better estimator than $\Empirical$.
We proceed by analysing the possible asymptotic regimes for the split $m$ relative to $n$.
\begin{itemize}
  \item \textit{$n-m = \LittleO(n)$.}\\
  In this regime, $n/(n-m)$ goes to infinity with $n$.
  Hence, when $n$ goes to infinity, we obtain
    \[
      \frac{\frac{1}{n-m} \SumOver{x} \AsymptoticVariance[k]{x}}{\MSE\left(\Empirical\right)}
      \EquivalentTo \frac{\SumOver{x} \AsymptoticVariance[k]{x}}
      {\SumOver{x} \AsymptoticVariance[0]{x}} \times \frac{n}{n-m}
      \to +\infty.
    \]
    Since the remaining terms in the expansion of $\MSE(\PiHatParam[n-m]{\ParamEstimator}{k})$ are non-negative or asymptotically negligible, it directly follows that $\MSE(\PiHatParam[n-m]{\ParamEstimator}{k}) / \MSE(\Empirical)$ goes to infinity with $n$.

    \newpage

  \item \textit{$(n-m)^{-1} = \BigO(n^{-1})$.}\\
    In this regime, $n g(m)$ goes to infinity with $n$.
    Hence, when $n$ goes to infinity,
    \[
    \begin{aligned}
      \frac{\MSE\left(\PiHatParam[n-m]{\ParamEstimator}{k}\right)}{\MSE\left(\Empirical\right)}
      &\geq \frac{g(m) \Trace\left(\Gradient\left(\pi \KernelMatrixTrue[k]\right)\Transpose
         \CovParam \Gradient\left(\pi \KernelMatrixTrue[k]\right)\right)
         + \LittleO\big(g(m)\big)}
         {\frac{1}{n} \SumOver{x} \AsymptoticVariance[0]{x} + \LittleO\left(\frac{1}{n}\right)} \\
      &\EquivalentTo n g(m)
         \frac{\Trace\left(\Gradient\left(\pi \KernelMatrixTrue[k]\right)\Transpose
         \CovParam \Gradient\left(\pi \KernelMatrixTrue[k]\right)\right)}
         {\SumOver{x} \AsymptoticVariance[0]{x}}
       \to +\infty,
    \end{aligned}
    \]
    since $\Trace(\Gradient(\pi \KernelMatrixTrue[k])\Transpose \CovParam \Gradient(\pi \KernelMatrixTrue[k])) \neq 0$ under the general assumptions on $\ParamEstimator$.
\end{itemize}
In all cases, the asymptotic MSE of $\PiHatParam[n-m]{\ParamEstimator}{k}$ is strictly greater than that of $\Empirical$.
Therefore, there exists no split such that $\PiHatParam[n-m]{\ParamEstimator}{k}$ is asymptotically a better estimator than $\Empirical$.

\subsection{Proof of Corollary~\ref{cor:data-partition2}}
\label{app:proof-data-partition2}

We assume that $g(m) = \LittleO(m^{-1})$, and we consider any split satisfying $m = \LittleO(n)$, $n = \LittleO(g(m)^{-1})$ and $n = \LittleO(m^{3/2})$.
We analyse each term of the expansion given in Theorem~\ref{thm:mse-expansion}.
First, we have
\[
  \MSE\left(\PiHatParam[n-m]{\ParamTrue}{k}\right)
  = \frac{1}{n-m} \SumOver{x} \AsymptoticVariance[k]{x}
    + \LittleO\left(\frac{1}{n-m}\right)
  = \frac{1}{n} \SumOver{x} \AsymptoticVariance[k]{x}
    + \LittleO\left(\frac{1}{n}\right).
\]
In addition,
\[
\begin{aligned}
  &\Trace\left(\Gradient\left(\pi \KernelMatrixTrue[k]\right)\Transpose
   \Expect\left[\left(\EstimationError\right)\left(\EstimationError\right)\Transpose\right]
   \Gradient\left(\pi \KernelMatrixTrue[k]\right)\right) \\
  &\qquad \EquivalentTo g(m) \Trace\left(\Gradient\left(\pi \KernelMatrixTrue[k]\right)\Transpose
   \CovParam \Gradient\left(\pi \KernelMatrixTrue[k]\right)\right)
   = \LittleO\left(\frac{1}{n}\right).
\end{aligned}
\]
Furthermore,
\[
\begin{aligned}
  \BigO\!\left(\frac{\Expect\left[\NormTwo{\EstimationError}^{2}\right]^{1/2}}{n-m}\right)
  &= \BigO\!\left(\frac{g(m)^{1/2}}{n}\right)
  = \LittleO\left(\frac{1}{n}\right), \\
  \BigO\!\left(\frac{\Expect\left[\NormTwo{\EstimationError}^{2}\right]}{n-m}\right)
  &= \BigO\!\left(\frac{g(m)}{n}\right)
  = \LittleO\left(\frac{1}{n}\right), \\
  \BigO\!\left(\Expect\left[\NormTwo{\EstimationError}^{3}\right]\right)
  &= \BigO\left(m^{-3/2}\right)
  = \LittleO\left(\frac{1}{n}\right).
\end{aligned}
\]
Combining these five estimates with the expansion of Theorem~\ref{thm:mse-expansion}, we obtain
\[
  \MSE\left(\PiHatParam[n-m]{\ParamEstimator}{k}\right)
  = \frac{1}{n} \SumOver{x} \AsymptoticVariance[k]{x} + \LittleO\left(\frac{1}{n}\right),
\]
and thus $\AsymptoticMSE(\PiHatParam[n-m]{\ParamEstimator}{k}) = \AsymptoticMSE(\PiHatParam{\ParamTrue}{k})$, which proves the result.

\subsection{Proof of Corollary~\ref{cor:data-partition3}}
\label{app:proof-data-partition3}

We assume that $g(m) \EquivalentTo m^{-1}$ as $m$ goes to infinity.
Thus,
\[
\begin{aligned}
  &\Trace\left(\Gradient\left(\pi \KernelMatrixTrue[k]\right)\Transpose
   \Expect\left[\left(\EstimationError\right)\left(\EstimationError\right)\Transpose\right]
   \Gradient\left(\pi \KernelMatrixTrue[k]\right)\right) \\
  &\qquad \EquivalentTo \frac{1}{m} \,
   \Trace\left(\Gradient\left(\pi \KernelMatrixTrue[k]\right)\Transpose
   \CovParam \Gradient\left(\pi \KernelMatrixTrue[k]\right)\right),
\end{aligned}
\]
since $\Trace(\Gradient(\pi \KernelMatrixTrue[k])\Transpose \CovParam \Gradient(\pi \KernelMatrixTrue[k])) \neq 0$ under the assumptions of the corollary.
We therefore define
\[
  c^{2} = \Trace\left(\Gradient\left(\pi \KernelMatrixTrue[k]\right)\Transpose
  \CovParam \Gradient\left(\pi \KernelMatrixTrue[k]\right)\right),
\]
so that $c^{2} / \SumOver{x} \AsymptoticVariance[0]{x} = \RatioParameter$, in accordance with Corollary~\ref{cor:data-partition3}.

Then, let $\SplitFraction \in \, ]0,1[$ and consider a split such that $m \EquivalentTo \SplitFraction n$, so that $n-m \EquivalentTo (1-\SplitFraction) n$.
In this regime, since $\Expect[\NormTwo{\EstimationError}^{2}] \EquivalentTo g(m) \Trace(\CovParam) = \BigO(m^{-1})$ and $\Expect[\NormTwo{\EstimationError}^{3}] = \BigO(m^{-3/2})$, the last three terms of the expansion of Theorem~\ref{thm:mse-expansion} satisfy
\[
\begin{aligned}
  &\BigO\!\left(\Expect\left[\NormTwo{\EstimationError}^{3}\right]\right)
  + \BigO\!\left(\frac{\Expect\left[\NormTwo{\EstimationError}^{2}\right]^{1/2}}{n-m}\right)
  + \BigO\!\left(\frac{\Expect\left[\NormTwo{\EstimationError}^{2}\right]}{n-m}\right) \\
  &\qquad = \BigO\left(n^{-3/2}\right)
  = \LittleO\left(\frac{1}{n}\right).
\end{aligned}
\]
Moreover, the data $(X_i)_{\IndexBetween{i}{m+1}{n}}$ form a stationary Markov chain, so that, by~\eqref{eq:mseinfty:stationaryMC},
\[
\begin{aligned}
  \MSE\left(\PiHatParam[n-m]{\ParamTrue}{k}\right)
  &= \frac{1}{n-m} \SumOver{x} \AsymptoticVariance[k]{x} + \LittleO\left(\frac{1}{n}\right), \\
  \MSE\left(\Empirical\right)
  &= \frac{1}{n} \SumOver{x} \AsymptoticVariance[0]{x} + \LittleO\left(\frac{1}{n}\right).
\end{aligned}
\]
Dividing the expansion of Theorem~\ref{thm:mse-expansion} by $\MSE(\Empirical)$, we obtain
\[
\begin{aligned}
  \frac{\MSE\left(\PiHatParam[n-m]{\ParamEstimator}{k}\right)}{\MSE\left(\Empirical\right)}
  &= \frac{\MSE\left(\PiHatParam[n-m]{\ParamTrue}{k}\right)}{\MSE\left(\Empirical\right)} \\
  &\quad + \frac{\Trace\left(\Gradient\left(\pi \KernelMatrixTrue[k]\right)\Transpose
     \Expect\left[\left(\EstimationError\right)\left(\EstimationError\right)\Transpose\right]
     \Gradient\left(\pi \KernelMatrixTrue[k]\right)\right)}{\MSE\left(\Empirical\right)}
     + \LittleO(1) \\
  &= \frac{\frac{1}{n(1-\SplitFraction)} \SumOver{x} \AsymptoticVariance[k]{x}}
     {\frac{1}{n} \SumOver{x} \AsymptoticVariance[0]{x}}
   + \frac{\frac{c^{2}}{\SplitFraction n}}
     {\frac{1}{n} \SumOver{x} \AsymptoticVariance[0]{x}} + \LittleO(1) \\
  &= \frac{\RatioLookAhead}{1-\SplitFraction}
   + \frac{\RatioParameter}{\SplitFraction} + \LittleO(1).
\end{aligned}
\]
It directly follows that
\[
  \frac{\AsymptoticMSE\left(\PiHatParam[n-m]{\ParamEstimator}{k}\right)}
  {\AsymptoticMSE\left(\Empirical\right)}
  = \frac{\RatioLookAhead}{1-\SplitFraction} + \frac{\RatioParameter}{\SplitFraction}.
\]
Hence, for $\SplitFraction \in \, ]0,1[$, we define $h(\SplitFraction) = \RatioLookAhead/(1-\SplitFraction) + \RatioParameter/\SplitFraction$.
The ratio $\RatioParameter$ is positive, since $\CovParam$ is positive semidefinite and the trace in its numerator is non-zero by assumption.
The ratio $\RatioLookAhead$ is positive according to Remark~\ref{rem:mse:zero}.
The function $h$ is then twice differentiable on $]0,1[$, with
\[
  h'(\SplitFraction) = \frac{\RatioLookAhead}{(1-\SplitFraction)^{2}} - \frac{\RatioParameter}{\SplitFraction^{2}}
  \qquad \text{and} \qquad
  h''(\SplitFraction) = \frac{2\RatioLookAhead}{(1-\SplitFraction)^{3}} + \frac{2\RatioParameter}{\SplitFraction^{3}} > 0.
\]
$h$ is thus strictly convex on $]0,1[$, and its derivative vanishes if and only if $\SplitFraction/(1-\SplitFraction) = \sqrt{\RatioParameter}/\sqrt{\RatioLookAhead}$.
Consequently, $h$ reaches its minimum at the single point
\[
  \SplitFractionOptimal = \frac{\sqrt{\RatioParameter}}
  {\sqrt{\RatioLookAhead}+\sqrt{\RatioParameter}} \in \, ]0,1[,
\]
and
\[
  h\left(\SplitFractionOptimal\right)
  = \frac{\sqrt{\RatioLookAhead}+\sqrt{\RatioParameter}}{\sqrt{\RatioLookAhead}} \RatioLookAhead
  + \frac{\sqrt{\RatioLookAhead}+\sqrt{\RatioParameter}}{\sqrt{\RatioParameter}} \RatioParameter
  = \left(\sqrt{\RatioLookAhead}+\sqrt{\RatioParameter}\right)^{2}.
\]
Therefore, for any split of the form $m \EquivalentTo \SplitFractionOptimal n$, the asymptotic MSE of $\PiHatParam[n-m]{\ParamEstimator}{k}$ is strictly smaller than that of $\Empirical$ if and only if $h(\SplitFractionOptimal) < 1$, that is, if and only if $\sqrt{\RatioLookAhead} + \sqrt{\RatioParameter} < 1$, which proves the result.

We finally determine the set of split fractions $\SplitFraction$ that lead to an improvement, as explained in Remark~\ref{rem:additional-knowledge-corollary}.
Assume that $\sqrt{\RatioLookAhead} + \sqrt{\RatioParameter} < 1$.
For $\SplitFraction \in \, ]0,1[$, multiplying by $\SplitFraction(1-\SplitFraction) > 0$, $h(\SplitFraction) < 1$ if and only if
\[
  q(\SplitFraction) = \SplitFraction^{2}
    + \SplitFraction\left(\RatioLookAhead-\RatioParameter-1\right)
    + \RatioParameter < 0.
\]
The polynomial $q$ satisfies $q(0) = \RatioParameter > 0$, $q(1) = \RatioLookAhead > 0$ and $q(\SplitFractionOptimal) = \SplitFractionOptimal(1-\SplitFractionOptimal)(h(\SplitFractionOptimal)-1) < 0$.
Hence, its discriminant $\Delta = (\RatioLookAhead-\RatioParameter-1)^{2}-4\RatioParameter$ is positive, and its two roots
\begin{equation}\label{eq:def:interval:alpha}
  \psi^{\pm}(\RatioLookAhead,\RatioParameter)
  = \frac{1-\RatioLookAhead+\RatioParameter \pm \sqrt{\Delta}}{2}
\end{equation}
both lie in $]0,1[$.
Therefore, any split of the form $m \EquivalentTo \SplitFraction n$ with $\SplitFraction \in \, ]\psi^{-},\psi^{+}[$ ensures that $\PiHatParam[n-m]{\ParamEstimator}{k}$ has a strictly smaller asymptotic MSE than $\Empirical$, while no other $\SplitFraction \in \, ]0,1[$ does.

\section{Proofs of Section~\ref{sec:birth-death}}
\label{app:birth-death}

Throughout this appendix, $\KernelMatrix = \KernelMatrixParam{\theta,\varepsilon}$ denotes the transition kernel of the birth-and-death process of Section~\ref{sec:birth-death}, and $\KernelMatrixParam{\theta,0}$ the kernel obtained for $\varepsilon = 0$.

\subsection{Proof of Proposition~\ref{prop:birth-death-law}}
\label{app:proof-birth-death-law}

It is immediate that $\KernelMatrix$ is irreducible and aperiodic.
It is also reversible, as the transition kernel of a birth-and-death process.
Moreover, by definition,
\[
  \KernelMatrix = (1-\varepsilon) \KernelMatrixParam{\theta,0} + \varepsilon \Identity,
\]
hence the invariant distribution of $\KernelMatrix$ is the invariant distribution of $\KernelMatrixParam{\theta,0}$, denoted by $\pi$.
We then have
\[
\begin{aligned}
  \KernelParam{\theta,0}{1}{0} \pi(0) = \KernelParam{\theta,0}{0}{1} \pi(1)
  &\Leftrightarrow \pi(0) = (1-\theta) \pi(1) \\
  &\Leftrightarrow \pi(1) = (1+\rho) \pi(0).
\end{aligned}
\]
Then, for $\IndexBetween{x}{1}{L-2}$,
\[
\begin{aligned}
  \KernelParam{\theta,0}{x+1}{x} \pi(x) = \KernelParam{\theta,0}{x}{x+1} \pi(x+1)
  &\Leftrightarrow \theta \pi(x) = (1-\theta) \pi(x+1) \\
  &\Leftrightarrow \pi(x+1) = \frac{\theta}{1-\theta} \pi(x) \\
  &\Leftrightarrow \pi(x+1) = \rho \pi(x),
\end{aligned}
\]
hence, by induction, we obtain the result for $\IndexBetween{x}{1}{L-1}$.
Finally,
\[
\begin{aligned}
  \KernelParam{\theta,0}{L}{L-1} \pi(L-1) = \KernelParam{\theta,0}{L-1}{L} \pi(L)
  &\Leftrightarrow \theta \pi(L-1) = \pi(L) \\
  &\Leftrightarrow \pi(L) = \frac{\rho}{1+\rho} \pi(0) (1+\rho) \rho^{L-2} \\
  &\Leftrightarrow \pi(L) = \pi(0) \rho^{L-1}.
\end{aligned}
\]
We then obtain $\pi(0)$ by normalisation, since $\SumFromTo{x}{0}{L} \pi(x) = 1$.

\subsection{Proof of Lemma~\ref{prop:birth-death-variance}}
\label{app:proof-birth-death-variance}

In the proof of Theorem~\ref{thm:monotone-amse}, we have established equation~\eqref{eq:variancef:from:clt}, namely
\[
  \AsymptoticVariance[0]{f}
  = 2 \ScalarProduct{\bar{f}}{\FundamentalMatrix \bar{f}}
  - \ScalarProduct{\bar{f}}{\bar{f}}.
\]
First, remark that $\ScalarProduct{\bar{f}}{\bar{f}}$ is simply $\pi(\bar{f}^{2})$.
Hence, we only need to analyse the term $\ScalarProduct{\bar{f}}{\FundamentalMatrix \bar{f}}$.
Let $g = \FundamentalMatrix \bar{f}$, so that
\[
  \ScalarProduct{\bar{f}}{\FundamentalMatrix \bar{f}}
  = \ScalarProduct{(\Identity - \KernelMatrix)g}{g}.
\]
We have
\[
\begin{aligned}
  \ScalarProduct{(\Identity - \KernelMatrix)g}{g}
  &= \SumFromTo{i}{0}{L} \pi(i) \left[(\Identity - \KernelMatrix)g\right](i)\, g(i) \\
  &= \SumFromTo{i}{0}{L} \pi(i)
     \left[g(i) - \SumFromTo{j}{0}{L} \Kernel{j}{i} g(j)\right] g(i) \\
  &= \SumFromTo{i}{0}{L} \pi(i) g(i)^{2}
     - \sum_{i,j = 0}^{L} \pi(i) \Kernel{j}{i} g(j) g(i).
\end{aligned}
\]
However,
\[
\begin{aligned}
  &\sum_{i,j = 0}^{L} \pi(i) \Kernel{j}{i} \left[g(i)-g(j)\right]^{2} \\
  &\qquad = \sum_{i,j = 0}^{L} \pi(i) \Kernel{j}{i} g(i)^{2}
     - 2 \sum_{i,j = 0}^{L} \pi(i) \Kernel{j}{i} g(i) g(j)
     + \sum_{i,j = 0}^{L} \pi(i) \Kernel{j}{i} g(j)^{2} \\
  &\qquad = \sum_{i,j = 0}^{L} \pi(j) \Kernel{i}{j} g(i)^{2}
     - 2 \sum_{i,j = 0}^{L} \pi(i) \Kernel{j}{i} g(i) g(j)
     + \SumFromTo{j}{0}{L} \pi(j) g(j)^{2} \\
  &\qquad = \SumFromTo{i}{0}{L} \pi(i) g(i)^{2}
     - 2 \sum_{i,j = 0}^{L} \pi(i) \Kernel{j}{i} g(i) g(j)
     + \SumFromTo{j}{0}{L} \pi(j) g(j)^{2} \\
  &\qquad = 2 \ScalarProduct{(\Identity - \KernelMatrix)g}{g},
\end{aligned}
\]
thanks to the detailed balance equation $\pi(i)\Kernel{j}{i} = \pi(j)\Kernel{i}{j}$ of reversible Markov chains.
Hence, we can rewrite
\[
  \ScalarProduct{(\Identity - \KernelMatrix)g}{g}
  = \frac{1}{2} \sum_{i,j=0}^{L} \pi(i) \Kernel{j}{i} \left[g(i)-g(j)\right]^{2}.
\]
Then, since $\KernelMatrix$ is the transition kernel of a birth-and-death process, this expression simplifies to
\[
\begin{aligned}
  \ScalarProduct{(\Identity - \KernelMatrix)g}{g}
  &= \frac{1}{2} \Bigg[\SumFromTo{j}{1}{L} \pi(j-1) \Kernel{j}{j-1}
     \left[g(j-1)-g(j)\right]^{2} \\
  &\hspace{1.2cm} + \SumFromTo{j}{0}{L-1} \pi(j+1) \Kernel{j}{j+1}
     \left[g(j+1)-g(j)\right]^{2}\Bigg] \\
  &= \SumFromTo{j}{0}{L-1} \pi(j) \Kernel{j+1}{j} \left[g(j+1)-g(j)\right]^{2}.
\end{aligned}
\]
Now, to complete the proof, we need to solve for $g$.
To this end, we use the fact that, by definition,
\[
  (\Identity - \KernelMatrix)g = \bar{f},
\]
and that, for $\IndexBetween{i}{1}{L-1}$,
\[
\begin{aligned}
  \left[(\Identity - \KernelMatrix)g\right](i)
  &= g(i) - \left[\Kernel{i}{i}g(i) + \Kernel{i-1}{i}g(i-1)
     + \Kernel{i+1}{i}g(i+1)\right] \\
  &= \left[\Kernel{i-1}{i} + \Kernel{i+1}{i}\right]g(i)
     - \Kernel{i-1}{i}g(i-1) - \Kernel{i+1}{i}g(i+1) \\
  &= \Kernel{i-1}{i} \left[g(i)-g(i-1)\right]
     + \Kernel{i+1}{i} \left[g(i)-g(i+1)\right].
\end{aligned}
\]
Multiplying by $\pi(i)$ and summing from $i = 1$ to $i = j$, for $\IndexBetween{j}{1}{L-1}$, this yields
\[
\begin{aligned}
  &\SumFromTo{i}{1}{j} \pi(i)\left[(\Identity - \KernelMatrix)g\right](i) \\
  &\qquad = \SumFromTo{i}{1}{j} \pi(i)\Big(\Kernel{i-1}{i} \left[g(i)-g(i-1)\right]
     + \Kernel{i+1}{i} \left[g(i)-g(i+1)\right]\Big) \\
  &\qquad = \SumFromTo{i}{1}{j} \pi(i-1) \Kernel{i}{i-1} \left[g(i)-g(i-1)\right]
     + \SumFromTo{i}{1}{j} \pi(i) \Kernel{i+1}{i} \left[g(i)-g(i+1)\right] \\
  &\qquad = \SumFromTo{i}{0}{j-1} \pi(i) \Kernel{i+1}{i} \left[g(i+1)-g(i)\right]
     + \SumFromTo{i}{1}{j} \pi(i) \Kernel{i+1}{i} \left[g(i)-g(i+1)\right] \\
  &\qquad = \pi(j) \Kernel{j+1}{j} \left[g(j)-g(j+1)\right]
     + \pi(0) \Kernel{1}{0}\left[g(1)-g(0)\right].
\end{aligned}
\]
On the boundary $i = 0$, we have
\[
\begin{aligned}
  \left[(\Identity - \KernelMatrix)g\right](0)
  &= g(0) - \left(\Kernel{0}{0}g(0) + \Kernel{1}{0}g(1)\right) \\
  &= \left(1-\Kernel{0}{0}\right)g(0) - \Kernel{1}{0}g(1) \\
  &= \Kernel{1}{0} \left(g(0)-g(1)\right),
\end{aligned}
\]
thus,
\[
  \pi(j)\Kernel{j+1}{j} \left[g(j)-g(j+1)\right]
  = \SumFromTo{i}{0}{j} \pi(i) \left[(\Identity - \KernelMatrix)g\right](i)
  = \SumFromTo{i}{0}{j} \pi(i) \bar{f}(i),
\]
from which we deduce that
\[
  g(j)-g(j+1) = \frac{1}{\pi(j)\Kernel{j+1}{j}} \SumFromTo{i}{0}{j} \pi(i)\bar{f}(i).
\]
We finally obtain
\[
\begin{aligned}
  \ScalarProduct{\bar{f}}{\FundamentalMatrix \bar{f}}
  &= \SumFromTo{j}{0}{L-1} \pi(j)\Kernel{j+1}{j}
     \left(\frac{1}{\pi(j)\Kernel{j+1}{j}} \SumFromTo{i}{0}{j} \pi(i)\bar{f}(i)\right)^{2} \\
  &= \SumFromTo{j}{0}{L-1} \frac{1}{\pi(j)\Kernel{j+1}{j}}
     \left(\SumFromTo{i}{0}{j} \pi(i)\bar{f}(i)\right)^{2}.
\end{aligned}
\]

\subsection{Proof of Proposition~\ref{prop:birth-death-MSE-pi0}}
\label{app:proof-birth-death-MSE-pi0}

We know that
\[
\begin{aligned}
  \TotalVariance{0}
  &= \SumFromTo{x}{0}{L} \AsymptoticVariance[0]{x} \\
  &= \SumFromTo{x}{0}{L} \Bigg(\frac{2}{1-\varepsilon} \SumFromTo{j}{0}{L-1}
     \frac{1}{\pi(j)\KernelParam{\theta,0}{j+1}{j}}
     \left(\SumFromTo{i}{0}{j} \pi(i) \Big(\Indicator{x=i}-\pi(x)\Big)\right)^{2}
     - \pi(x)+\pi(x)^{2}\Bigg) \\
  &= \frac{2}{1-\varepsilon} \SumFromTo{j}{0}{L-1}
     \frac{1}{\pi(j)\KernelParam{\theta,0}{j+1}{j}}
     \SumFromTo{x}{0}{L}
     \left(\SumFromTo{i}{0}{j} \pi(i) \Big(\Indicator{x=i}-\pi(x)\Big)\right)^{2}
     - 1 + \SumFromTo{x}{0}{L} \pi(x)^{2} \\
  &= \frac{2}{1-\varepsilon} \SumFromTo{j}{0}{L-1}
     \frac{1}{\pi(j)\KernelParam{\theta,0}{j+1}{j}} \SumFromTo{x}{0}{L} S_x(j)^{2}
     - 1 + \SumFromTo{x}{0}{L} \pi(x)^{2},
\end{aligned}
\]
where, for $\IndexBetween{x}{0}{L}$ and $\IndexBetween{j}{0}{L-1}$,
\[
  S_x(j) = \SumFromTo{i}{0}{j} \pi(i) \Big(\Indicator{x=i}-\pi(x)\Big).
\]
However,
\[
  S_x(j) =
  \begin{cases}
    -\pi(x) \Pi(j) & \text{if } j < x, \\
    \pi(x)\big(1-\Pi(j)\big) & \text{if } j \geq x,
  \end{cases}
\]
with $\Pi(j) = \SumFromTo{i}{0}{j} \pi(i)$ the cumulative distribution function of $\pi$.
We then obtain
\[
\begin{aligned}
  \SumFromTo{x}{0}{L} S_x(j)^{2}
  &= \SumFromTo{x}{0}{j} \pi(x)^{2} \big(1-\Pi(j)\big)^{2}
     + \SumFromTo{x}{j+1}{L} \pi(x)^{2}\Pi(j)^{2} \\
  &= \SumFromTo{x}{0}{j} \pi(x)^{2} \big(1-\Pi(j)\big)^{2}
     + \SumFromTo{x}{j+1}{L} \pi(x)^{2} \big(1-\Pi(j)\big)^{2}
     - \big(1-2\Pi(j)\big)\SumFromTo{x}{j+1}{L} \pi(x)^{2} \\
  &= \big(1-\Pi(j)\big)^{2} \SumFromTo{x}{0}{L} \pi(x)^{2}
     - \big(1-2\Pi(j)\big)
     \left(\SumFromTo{x}{0}{L} \pi(x)^{2}-\SumFromTo{x}{0}{j} \pi(x)^{2}\right).
\end{aligned}
\]
For $\IndexBetween{j}{0}{L}$, we denote $\Pi^{2}(j) = \SumFromTo{x}{0}{j} \pi(x)^{2}$.
Thus, we can rewrite
\[
\begin{aligned}
  \TotalVariance{0}
  &= \frac{2}{1-\varepsilon} \SumFromTo{j}{0}{L-1}
     \frac{\big(1-\Pi(j)\big)^{2} \Pi^{2}(L)
     - \big(1-2\Pi(j)\big) \big(\Pi^{2}(L)-\Pi^{2}(j)\big)}
     {\pi(j)\KernelParam{\theta,0}{j+1}{j}}
     - 1 + \Pi^{2}(L) \\
  &= \frac{2}{1-\varepsilon} \SumFromTo{j}{0}{L-1}
     \frac{\big(1-\Pi(j)\big)^{2} \Pi^{2}(L)}{\pi(j)\KernelParam{\theta,0}{j+1}{j}} \\
  &\quad - \frac{2}{1-\varepsilon} \SumFromTo{j}{0}{L-1}
     \frac{\big(1-2\Pi(j)\big) \big(\Pi^{2}(L)-\Pi^{2}(j)\big)}
     {\pi(j)\KernelParam{\theta,0}{j+1}{j}}
     - 1 + \Pi^{2}(L).
\end{aligned}
\]
For $\IndexBetween{j}{0}{L-1}$, we have
\[
\begin{aligned}
  1-\Pi(j)
  &= 1-\left(\pi(0) + \pi(0)(1+\rho) \SumFromTo{y}{1}{j} \rho^{y-1}\right) \\
  &= 1-\pi(0) \left(1+(1+\rho)\frac{1-\rho^{j}}{1-\rho}\right) \\
  &= 1-\frac{1-\rho}{2(1-\rho^{L})}
     \left(\frac{2}{1-\rho}-\frac{1+\rho}{1-\rho}\rho^{j}\right) \\
  &= \frac{1}{1-\rho^{L}} \times \frac{1+\rho}{2}\rho^{j}
     - \frac{\rho^{L}}{1-\rho^{L}}.
\end{aligned}
\]
Therefore,
\[
\begin{aligned}
  \SumFromTo{j}{0}{L-1}
  \frac{\big(1-\Pi(j)\big)^{2}}{\pi(j)\KernelParam{\theta,0}{j+1}{j}}
  &= \frac{\big(1-\pi(0)\big)^{2}}{\pi(0)}
     + \SumFromTo{j}{1}{L-1} \frac{(1+\rho)^{2}}{4(1-\rho^{L})^{2}} \rho^{2j}
     \frac{1}{\pi(0)(1+\rho)\rho^{j-1} \theta} + E_L \\
  &= \frac{\big(1-\pi(0)\big)^{2}}{\pi(0)}
     + \frac{(1+\rho)^{2}}{4 \pi(0)(1-\rho^{L})^{2}} \SumFromTo{j}{1}{L-1} \rho^{j}
     + E_L \\
  &= \frac{\big(1-\pi(0)\big)^{2}}{\pi(0)}
     + \frac{(1+\rho)^{2}}{4 \pi(0)(1-\rho^{L})^{2}} \rho \frac{1-\rho^{L-1}}{1-\rho}
     + E_L,
\end{aligned}
\]
with
\[
  E_L = \SumFromTo{j}{1}{L-1}
  \frac{1}{\pi(j)\KernelParam{\theta,0}{j+1}{j}}
  \left[\left(\frac{\rho^{L}}{1-\rho^{L}}\right)^{2}
  - \frac{(1+\rho) \rho^{j} \rho^{L}}{(1-\rho^{L})^{2}}\right]
  = \LittleO_{L}(1).
\]
Since $\pi(0) = (1-\rho)/(2(1-\rho^{L})) = (1-\rho)/2 + \LittleO_{L}(1)$, we deduce that
\[
\begin{aligned}
  \SumFromTo{j}{0}{L-1}
  \frac{\big(1-\Pi(j)\big)^{2}}{\pi(j)\KernelParam{\theta,0}{j+1}{j}}
  &= \frac{(1+\rho)^{2}}{2(1-\rho)}
     + \frac{(1+\rho)^{2}}{2(1-\rho)} \times \frac{\rho}{1-\rho} + \LittleO_{L}(1) \\
  &= \frac{(1+\rho)^{2}}{2(1-\rho)^{2}} + \LittleO_{L}(1).
\end{aligned}
\]
On the other hand, we compute
\[
\begin{aligned}
  \Pi^{2}(L)
  &= \pi(0)^{2} + \pi(0)^{2}(1+\rho)^{2} \SumFromTo{x}{1}{L-1} \rho^{2(x-1)}
     + \pi(0)^{2}\rho^{2(L-1)} \\
  &= \pi(0)^{2} \left(1 + \rho^{2(L-1)}
     + (1+\rho)^{2} \frac{1-\rho^{2(L-1)}}{1-\rho^{2}} \right) \\
  &  = \frac{1-\rho}{2} \times \frac{1-\rho^{2L-1}}{\left(1-\rho^{L}\right)^{2}}.
\end{aligned}
\]
For $\IndexBetween{j}{0}{L-1}$, we have
\[
\begin{aligned}
  \Pi^{2}(L) - \Pi^{2}(j)
  &= \pi(0)^{2} (1+\rho)^{2} \SumFromTo{x}{j+1}{L-1} \rho^{2(x-1)}
     + \pi(0)^{2} \rho^{2(L-1)} \\
  &= \pi(0)^{2} \left( (1+\rho)^{2} \rho^{2j} \frac{1-\rho^{2(L-1-j)}}{1-\rho^{2}}
     + \rho^{2(L-1)} \right) \\
   &= \left(\frac{1-\rho}{2(1-\rho^{L})}\right)^{2}
     \frac{(1+\rho)\rho^{2j} - 2\rho^{2L-1}}{1-\rho} \\
  &= \frac{(1-\rho^{2})\rho^{2j} - 2(1-\rho)\rho^{2L-1}}{4(1-\rho^{L})^{2}}.
\end{aligned}
\]
Furthermore, using the expression of $1-\Pi(j)$, we have
\[
  1-2\Pi(j) = 2\big(1-\Pi(j)\big) - 1
  = \frac{(1+\rho)\rho^{j} - (1+\rho^{L})}{1-\rho^{L}}.
\]
We finally obtain
\[
\begin{aligned}
  &\SumFromTo{j}{0}{L-1}
   \frac{\big(1-2\Pi(j)\big) \big(\Pi^{2}(L)-\Pi^{2}(j)\big)}
   {\pi(j)\KernelParam{\theta,0}{j+1}{j}} \\
  &\qquad = \frac{1}{4\pi(0)(1-\rho^{L})^{3}} \SumFromTo{j}{0}{L-1} \frac{1}{\rho^{j}}
     \Big( (1+\rho)\rho^{j} - (1+\rho^{L}) \Big)
     \Big( (1-\rho^{2})\rho^{2j} - 2(1-\rho)\rho^{2L-1} \Big) \\
  &\qquad = \frac{1}{4\pi(0)(1-\rho^{L})^{3}} \SumFromTo{j}{0}{L-1}
     \bigg( (1+\rho)(1-\rho^{2})\rho^{2j} - 2(1-\rho^{2})\rho^{2L-1} \\
  &\qquad \hspace{3.5cm} - (1+\rho^{L})(1-\rho^{2})\rho^{j}
     + 2(1+\rho^{L})(1-\rho)\rho^{2L-1-j} \bigg) \\
  &\qquad = \frac{1}{4\pi(0)(1-\rho^{L})^{3}}
     \Big( (1+\rho)(1-\rho^{2L}) - 2L(1-\rho^{2})\rho^{2L-1} \\
  &\qquad \hspace{3.5cm} - (1+\rho)(1-\rho^{2L}) + 2\rho^{L}(1-\rho^{2L}) \Big) \\
  &\qquad = \frac{2\rho^{L}(1-\rho^{2L}) - 2L(1-\rho^{2})\rho^{2L-1}}
     {4\pi(0)(1-\rho^{L})^{3}}
   = \LittleO_{L}(1).
\end{aligned}
\]
Returning to our initial formula, we finally conclude that
\[
\begin{aligned}
  \TotalVariance{0}
  &= \frac{2}{1-\varepsilon} \times \frac{(1+\rho)^{2}}{2(1-\rho)^{2}}
     \times \frac{1-\rho}{2} - 1 + \frac{1-\rho}{2} + \LittleO_{L}(1) \\
  &= \frac{2}{1-\varepsilon} \times \frac{(1+\rho)^{2}}{4(1-\rho)}
     - \frac{1+\rho}{2} + \LittleO_{L}(1).
\end{aligned}
\]

\subsection{Proof of Proposition~\ref{prop:birth-death-MSE-pi1}}
\label{app:proof-birth-death-MSE-pi1}

We know that
\[
\begin{aligned}
  \SumFromTo{x}{0}{L} \AsymptoticVariance[1]{x}
  &= \SumFromTo{x}{0}{L} \Bigg(2 \SumFromTo{j}{0}{L-1}
     \frac{1}{\pi(j)\Kernel{j+1}{j}}
     \left( \SumFromTo{i}{0}{j} \pi(i) \Big(\Kernel{x}{i}-\pi(x)\Big) \right)^{2} \\
  &\hspace{2cm} - \SumFromTo{y}{0}{L}\pi(y)\Kernel{x}{y}^{2} + \pi(x)^{2}\Bigg) \\
  &= \frac{2}{1-\varepsilon} \SumFromTo{j}{0}{L-1}
     \frac{1}{\pi(j)\KernelParam{\theta,0}{j+1}{j}} \SumFromTo{x}{0}{L} S'_x(j)^{2} \\
  &\quad - \SumFromTo{y}{0}{L} \pi(y) \SumFromTo{x}{0}{L} \Kernel{x}{y}^{2}
     + \SumFromTo{x}{0}{L} \pi(x)^{2},
\end{aligned}
\]
where, for $\IndexBetween{x}{0}{L}$ and $\IndexBetween{j}{0}{L-1}$,
\[
  S'_x(j) = \SumFromTo{i}{0}{j} \pi(i) \Big(\Kernel{x}{i}-\pi(x)\Big).
\]
For $\IndexBetween{x}{1}{L-1}$, we have
\[
\begin{aligned}
  S'_x(j) &=
  \begin{cases}
    -\pi(x) \Pi(j) & \text{if } j < x-1, \\
    -\pi(x) \Pi(j) + \pi(x-1) \Kernel{x}{x-1} & \text{if } j = x-1, \\
    -\pi(x) \Pi(j) + \pi(x-1) \Kernel{x}{x-1} + \pi(x) \Kernel{x}{x} & \text{if } j = x, \\
    \pi(x) \big(1-\Pi(j)\big) & \text{if } j > x,
  \end{cases} \\
  &=
  \begin{cases}
    S_x(j) & \text{if } j < x-1, \\
    S_x(j)+\pi(x-1)\Kernel{x}{x-1} & \text{if } j = x-1, \\
    S_x(j)-\pi(x+1)\Kernel{x}{x+1} & \text{if } j = x, \\
    S_x(j) & \text{if } j > x,
  \end{cases}
\end{aligned}
\]
and, using the reversibility of the chain,
\[
  S'_x(j) =
  \begin{cases}
    S_x(j) & \text{if } j < x-1, \\
    S_x(j)+\pi(j)\Kernel{j+1}{j} & \text{if } j = x-1, \\
    S_x(j)-\pi(j)\Kernel{j+1}{j} & \text{if } j = x, \\
    S_x(j) & \text{if } j > x,
  \end{cases}
\]
which allows us to simply write
\[
  S'_x(j) = S_x(j) + \pi(j)\Kernel{j+1}{j} \Big(\Dirac{x}(j+1)-\Dirac{x}(j)\Big).
\]
Similarly, we have
\[
  S'_x(0) =
  \begin{cases}
    S_x(0)-\pi(0)\Kernel{1}{0} & \text{if } x = 0, \\
    S_x(0)+\pi(0)\Kernel{1}{0} & \text{if } x = 1, \\
    S_x(0) & \text{if } x > 1,
  \end{cases}
\]
which therefore allows us to conclude that, for all $\IndexBetween{j}{0}{L-1}$,
\[
  S'_x(j) = S_x(j) + \pi(j)\Kernel{j+1}{j} \Big(\Dirac{x}(j+1)-\Dirac{x}(j)\Big).
\]
Squaring and summing over $\IndexBetween{x}{0}{L}$, we obtain
\[
\begin{aligned}
  \SumFromTo{x}{0}{L} S'_x(j)^{2}
  &= \SumFromTo{x}{0}{L} S_x(j)^{2}
     + 2 \pi(j)\Kernel{j+1}{j} \SumFromTo{x}{0}{L} S_x(j)
     \Bigl(\Dirac{x}(j+1)-\Dirac{x}(j)\Bigr) \\
  &\quad + \pi(j)^{2} \Kernel{j+1}{j}^{2}
     \SumFromTo{x}{0}{L} \Bigl(\Dirac{x}(j+1)-\Dirac{x}(j)\Bigr)^{2} \\
  &= \SumFromTo{x}{0}{L} S_x(j)^{2}
     + 2\pi(j)\Kernel{j+1}{j} \Bigl(S_{j+1}(j)-S_j(j)\Bigr)
     + 2\pi(j)^{2} \Kernel{j+1}{j}^{2}.
\end{aligned}
\]
Plugging this result back into the overall expression for $\SumFromTo{x}{0}{L} \AsymptoticVariance[1]{x}$, we get
\[
\begin{aligned}
  \SumFromTo{x}{0}{L} \AsymptoticVariance[1]{x}
  &= \SumFromTo{x}{0}{L} \AsymptoticVariance[0]{x} + 1 - \Pi^{2}(L) \\
  &\quad + \frac{2}{1-\varepsilon} \SumFromTo{j}{0}{L-1}
     \frac{2 \pi(j)^{2} \Kernel{j+1}{j}^{2}
     + 2\pi(j)\Kernel{j+1}{j} \Big( S_{j+1}(j) - S_j(j) \Big)}
     {\pi(j)\KernelParam{\theta,0}{j+1}{j}} \\
  &\quad - \SumFromTo{y}{0}{L} \pi(y) \SumFromTo{x}{0}{L} \Kernel{x}{y}^{2}
     + \Pi^{2}(L),
\end{aligned}
\]
and, since $\Kernel{j+1}{j} = (1-\varepsilon) \KernelParam{\theta,0}{j+1}{j}$,
\[
\begin{aligned}
  \SumFromTo{x}{0}{L} \AsymptoticVariance[1]{x}
  &= \SumFromTo{x}{0}{L} \AsymptoticVariance[0]{x} + 1
     - \SumFromTo{y}{0}{L} \pi(y) \SumFromTo{x}{0}{L} \Kernel{x}{y}^{2} \\
  &\quad + 4(1-\varepsilon) \SumFromTo{j}{0}{L-1} \pi(j)\KernelParam{\theta,0}{j+1}{j}
     + 4 \SumFromTo{j}{0}{L-1} \big( S_{j+1}(j) - S_j(j) \big).
\end{aligned}
\]
We then evaluate
\[
\begin{aligned}
  \SumFromTo{j}{0}{L-1} \pi(j) \KernelParam{\theta,0}{j+1}{j}
  &= \pi(0) + \pi(0) (1+\rho) \SumFromTo{j}{1}{L-1} \rho^{j-1} \theta
   = \pi(0) \left(1+\rho\frac{1-\rho^{L-1}}{1-\rho}\right) \\
  &= \pi(0) \frac{1-\rho^{L}}{1-\rho}
   = \frac{1}{2}.
\end{aligned}
\]
In addition,
\[
\begin{aligned}
  \SumFromTo{j}{0}{L-1} \big(S_{j+1}(j)-S_j(j)\big)
  &= \SumFromTo{j}{0}{L-1} \Big[-\pi(j+1)\Pi(j)-\pi(j)\big(1-\Pi(j)\big)\Big] \\
  &= \SumFromTo{j}{0}{L-1} \Big(1-\Pi(j)\Big) \big(\pi(j+1)-\pi(j)\big)
     - \SumFromTo{j}{0}{L-1} \pi(j+1) \\
  &= \frac{1+\rho}{2} \SumFromTo{j}{0}{L-1} \rho^{j} \big(\pi(j+1)-\pi(j) \big)
     - 1 + \pi(0) + \LittleO_{L}(1) \\
  &= \frac{1+\rho}{2} \pi(0)\left(\rho + (1+\rho) \SumFromTo{j}{1}{L-1}
     \rho^{j} \Big(\rho^{j} - \rho^{j-1} \Big) \right) - 1 + \pi(0) + \LittleO_{L}(1) \\
  &= \frac{1+\rho}{2} \pi(0) \Big(\rho - (1+\rho) \frac{\rho}{1+\rho}\Big)
     - 1 + \pi(0) + \LittleO_{L}(1) \\
  &= - 1 + \pi(0) + \LittleO_{L}(1)
   = -\frac{1+\rho}{2} + \LittleO_{L}(1),
\end{aligned}
\]
and
\[
\begin{aligned}
  \SumFromTo{y}{0}{L} \pi(y) \SumFromTo{x}{0}{L} \Kernel{x}{y}^{2}
  &= \Big(\pi(0)+\pi(L)\Big) \Big[\varepsilon^{2}+(1-\varepsilon)^{2}\Big] \\
  &\quad + \SumFromTo{y}{1}{L-1} \pi(y)
     \Big[\varepsilon^{2}+(1-\varepsilon)^{2}\big(\theta^{2}+(1-\theta)^{2}\big)\Big] \\
  &= \pi(0) \Big[\varepsilon^{2}+(1-\varepsilon)^{2}\Big]
     + \SumFromTo{y}{1}{L-1} \pi(y)
     \left[\varepsilon^{2}+(1-\varepsilon)^{2} \frac{1+\rho^{2}}{(1+\rho)^{2}} \right]
     + \LittleO_{L}(1) \\
  &= \varepsilon^{2}+(1-\varepsilon)^{2} \frac{1+\rho^{2}}{(1+\rho)^{2}}
     + \pi(0) (1-\varepsilon)^{2} \frac{2\rho}{(1+\rho)^{2}} + \LittleO_{L}(1) \\
  &= \varepsilon^{2}+(1-\varepsilon)^{2}
     \left[\frac{1+\rho^{2}}{(1+\rho)^{2}}+\frac{(1-\rho)\rho}{(1+\rho)^{2}} \right]
     + \LittleO_{L}(1) \\
  &= \varepsilon^{2}+\frac{(1-\varepsilon)^{2}}{1+\rho} + \LittleO_{L}(1).
\end{aligned}
\]
Gathering these contributions with the expression of $\SumFromTo{x}{0}{L} \AsymptoticVariance[0]{x}$ obtained previously, we conclude that
\[
\begin{aligned}
  \SumFromTo{x}{0}{L} \AsymptoticVariance[1]{x}
  &= \frac{2}{1-\varepsilon} \times \frac{(1+\rho)^{2}}{4(1-\rho)} - \frac{1+\rho}{2} + 1
     - \left(\varepsilon^{2} + \frac{(1-\varepsilon)^{2}}{1+\rho}\right) \\
  &\quad + 2(1-\varepsilon) - 2(1+\rho) + \LittleO_{L}(1) \\
  &= \frac{2}{1-\varepsilon} \times \frac{(1+\rho)^{2}}{4(1-\rho)} - \frac{1+\rho}{2} + 1
     - \left(\varepsilon^{2} + \frac{(1-\varepsilon)^{2}}{1+\rho}\right)
     - 2(\varepsilon+\rho) + \LittleO_{L}(1).
\end{aligned}
\]

\subsection{Proof of Proposition~\ref{prop:birth-death-MSE-parameter}}
\label{app:proof-birth-death-MSE-parameter}

Let
\[
  N_n = \SumFromTo{i}{1}{n-1} \Indicator{X_i \notin \{0,L\},\ X_{i+1} \neq X_i},
\]
let $(\tau_j)_{\IndexBetween{j}{1}{N_n}}$ be the increasing sequence of times $i$ such that $X_{i} \notin \{0,L\}$ and $X_{i+1} \neq X_{i}$, and let $B_j = \Indicator{X_{\tau_j+1} = X_{\tau_j}+1}$, so that, whenever $\ParamEstimator[n]$ is defined,
\[
  \ParamEstimator[n] = \frac{1}{N_n} \SumFromTo{j}{1}{N_n} B_j.
\]
However, for $\IndexBetween{j}{1}{N_n}$,
\[
  \Prob\left(B_j = 1 \mid B_1, \ldots, B_{j-1}\right) = \theta,
\]
thus the random variables $(B_j)_{\IndexBetween{j}{1}{N_n}}$ are independent and identically distributed according to the Bernoulli distribution with parameter $\theta$.
In particular, $\Expect[B_j] = \theta$ and $\Variance(B_j) = \theta(1-\theta)$.
Since $N_n$ goes to infinity almost surely when $n$ goes to infinity, we can apply the central limit theorem for random sample sizes, which yields, when $n$ goes to infinity,
\[
  \sqrt{N_n} \left(\ParamEstimator[n]-\theta\right)
  \ConvergesInDistribution \Gaussian{0}{\theta(1-\theta)}.
\]
Then, since
\[
  \Expect\left[\Indicator{X_i \notin \{0,L\},\ X_{i+1} \neq X_i}\right]
  = (1-\varepsilon)\left(1-\pi(0)-\pi(L)\right)
  = (1-\varepsilon) \frac{1+\rho}{2} + \LittleO_{L}(1),
\]
we deduce from the ergodic theorem for Markov chains that, when $n$ goes to infinity,
\[
  \frac{N_n}{n} \ConvergesAlmostSurely (1-\varepsilon) \frac{1+\rho}{2} + \LittleO_{L}(1).
\]
Finally, since $\sqrt{n}(\ParamEstimator[n]-\theta)$ factors as
\[
  \sqrt{\frac{n}{N_n}} \times \sqrt{N_n} \left(\ParamEstimator[n]-\theta\right),
\]
Slutsky's lemma shows that, when $n$ goes to infinity,
\[
  \sqrt{n} \left(\ParamEstimator[n]-\theta\right)
  \ConvergesInDistribution \Gaussian{0}{s^{2}},
\]
with
\[
  s^{2}
  = \frac{1}{(1-\varepsilon) \frac{1+\rho}{2} + \LittleO_{L}(1)}
    \times \frac{\rho}{(1+\rho)^{2}}
  = \frac{2\rho}{(1-\varepsilon)(1+\rho)^{3}} + \LittleO_{L}(1).
\]

\subsection{Proof of Proposition~\ref{prop:birth-death-comparison}}
\label{app:proof-birth-death-comparison}

First, remark that Proposition~\ref{prop:birth-death-law} ensures that, for any $L \geq 2$ and $\varepsilon \in \, ]0,1[$, the family $(\KernelMatrixParam{\theta,\varepsilon})_{\theta \in \, ]0,1[}$ satisfies the assumptions introduced in Subsection~\ref{subsec:parametric-framework}.
Moreover, $\theta \mapsto \KernelMatrixParam{\theta,\varepsilon}$ is of class $\mathcal{C}^{2}$ and $\theta \mapsto \Gradient \KernelParam{\theta,\varepsilon}{x}{y}$ and $\theta \mapsto \Hessian \KernelParam{\theta,\varepsilon}{x}{y}$ are bounded on $]0,1[$ for all $x,y \in \StateSpace$.
Consequently, since the state space is finite, the family of kernels satisfies Conditions~\ref{ass:regularity1} and~\ref{ass:regularity2} according to Remark~\ref{rem:relaxed-parameter-space}.

By Proposition~\ref{prop:birth-death-MSE-parameter}, the estimator $\ParamEstimator[n]$ satisfies the general assumptions introduced in Subsection~\ref{subsec:parametric-data-partition} and, since it converges at the parametric rate, Corollary~\ref{cor:data-partition3} applies.

Propositions~\ref{prop:birth-death-MSE-pi0}, \ref{prop:birth-death-MSE-pi1} and~\ref{prop:birth-death-MSE-parameter} ensure that the quantities of interest converge when $\varepsilon$ goes to $0$ and then $L$ goes to infinity, their limits being given by
\[
\begin{aligned}
  \lim_{L \to +\infty} \lim_{\varepsilon \to 0}
  \SumFromTo{x}{0}{L} \AsymptoticVariance[0]{x}
  &= \frac{\rho(1+\rho)}{1-\rho}, \\
  \lim_{L \to +\infty} \lim_{\varepsilon \to 0}
  \SumFromTo{x}{0}{L} \AsymptoticVariance[1]{x}
  &= \frac{\rho^{2}(1+3\rho)}{(1+\rho)(1-\rho)}, \\
  \lim_{L \to +\infty} \lim_{\varepsilon \to 0} s^{2}
  &= \frac{2\rho}{(1+\rho)^{3}}.
\end{aligned}
\]
Now, since the parametrisation is one-dimensional,
\[
  \Trace\left(\Gradient\left(\pi \KernelMatrixTrue\right)\Transpose s^{2}
  \Gradient\left(\pi \KernelMatrixTrue\right)\right)
  = s^{2} \SumFromTo{x}{0}{L} \Sensitivity{}{x}^{2},
\]
where, for $\IndexBetween{x}{0}{L}$,
\[
  \Sensitivity{}{x} = \SumFromTo{y}{0}{L} \pi(y)\,
  \partial_{\theta} \KernelParam{\theta,\varepsilon}{x}{y}.
\]
In the limit when $\varepsilon$ goes to $0$, we have
\[
\begin{alignedat}{2}
  &\forall \IndexBetween{x}{0}{1}, &\quad& \Sensitivity{}{x} = -\pi(x+1), \\
  &\forall \IndexBetween{x}{2}{L-2}, && \Sensitivity{}{x} = \pi(x-1)-\pi(x+1), \\
  &\forall \IndexBetween{x}{L-1}{L}, && \Sensitivity{}{x} = \pi(x-1).
\end{alignedat}
\]
Hence,
\[
\begin{aligned}
  \SumFromTo{x}{0}{L} \Sensitivity{}{x}^{2}
  &= \pi(1)^{2}+\pi(2)^{2}+\SumFromTo{x}{2}{L-2} \big(\pi(x-1)-\pi(x+1)\big)^{2}
     + \LittleO_{L}(1) \\
  &= \pi(0)^{2} (1+\rho)^{2} \left[1+\rho^{2}
     + \SumFromTo{x}{2}{L-2} \big(\rho^{x-2}-\rho^{x}\big)^{2}\right] + \LittleO_{L}(1) \\
  &= \pi(0)^{2} (1+\rho)^{2} \left[1+\rho^{2}+1-\rho^{2}\right] + \LittleO_{L}(1) \\
  &= \frac{(1-\rho^{2})^{2}}{2} {}+{} \LittleO_{L}(1).
\end{aligned}
\]
In the limit when $\varepsilon$ goes to $0$ and then $L$ goes to infinity, this finally gives
\[
  \RatioLookAhead = \frac{\rho(1+3\rho)}{(1+\rho)^{2}}
  \TextAnd
  \RatioParameter = \frac{(1-\rho)^{3}}{(1+\rho)^{2}}.
\]
The result is then a direct consequence of the comparison $\sqrt{\RatioLookAhead}+\sqrt{\RatioParameter} < 1$.
Solving this inequality numerically over $\rho \in \, ]0,1[$ yields the critical threshold $\bar{\rho} \simeq 0.4081$, which corresponds to the threshold $\bar{\theta} \simeq 0.289$.

\end{document}